\documentclass[10pt]{amsart}
\usepackage{latexsym, amsmath,amssymb}

\usepackage{xcolor}

\newenvironment{psmallmatrix}
  {\left(\begin{smallmatrix}}
  {\end{smallmatrix}\right)}

\renewcommand{\epsilon}{\varepsilon}
\newcommand{\newsection}[1]
{\subsection{#1}\setcounter{theorem}{0} \setcounter{equation}{0}
\par\noindent}

\newtheorem{theorem}{Theorem}

\newtheorem{lemma}[theorem]{Lemma}
\newtheorem{corr}[theorem]{Corollary}

\newtheorem{proposition}[theorem]{Proposition}
\newtheorem{deff}[theorem]{Definition}

\newcommand{\bth}{\begin{theorem}}
\newcommand{\ble}{\begin{lemma}}
\newcommand{\bcor}{\begin{corr}}

\newcommand{\bdeff}{\begin{deff}}

\newcommand{\bprop}{\begin{proposition}}
\newcommand{\ele}{\end{lemma}}
\newcommand{\ecor}{\end{corr}}
\newcommand{\edeff}{\end{deff}}

\newcommand{\eprop}{\end{proposition}}

\newcommand{\Rn}{{\mathbb R}^n}

\newcommand{\la}{\lambda}

\newcommand{\e}{\varepsilon}

\renewcommand{\Pi}{\varPi}

\renewcommand{\epsilon}{\varepsilon}

\newcommand{\R}{{\mathbb R}}
\newcommand{\Z}{{\mathbb Z}}
\newcommand{\tubek}{{\mathcal T}_k}
\newcommand{\psik}{\psi_{\la_k}}
\newcommand{\ao}{\alpha_{\gamma_0}}
\newcommand{\Go}{G_{\gamma_0}}
\newcommand{\lk}{\la_k}
\newcommand{\dk}{\delta_k}
\newcommand{\akd}{a_{\la_k, \delta_k}}
\newcommand{\etak}{\eta(T_k(\la_k-|\xi|))}

\newcommand{\one}{{\bf 1}}

\newcommand{\diag}{\Upsilon^{\text{diag}}}

\newcommand{\far}{\Upsilon^{\text{far}}}

\begin{document}

\title
[Curvature and growth rates of log-quasimodes]{Curvature and sharp growth rates of log-quasimodes on compact manifolds}

\thanks{The first author was supported in part by an AMS-Simons travel grant. The second author was supported in part by the NSF (DMS-1665373,
DMS-2348996)
and a Simons Fellowship.
This research was partly carried out while the first author was at the University of Maryland.
}


\author{Xiaoqi Huang}
\address[X.H.]{Department of Mathematics, Louisiana State University, Baton Rouge, LA 70803}
\email{xhuang49@lsu.edu
}
\author{Christopher D. Sogge}
\address[C.D.S.]{Department of Mathematics,  Johns Hopkins University,
Baltimore, MD 21218}
\email{sogge@jhu.edu}


\begin{abstract}
 We obtain new optimal estimates for the $L^2(M)\to L^q(M)$, $q\in (2,q_c]$, $q_c=2(n+1)/(n-1)$, operator norms of spectral projection operators associated with spectral windows $[\la,\la+\delta(\la)]$, with $\delta(\la)=O((\log\la)^{-1})$ on compact Riemannian manifolds $(M,g)$ of dimension $n\ge2$ all of whose sectional curvatures are nonpositive or negative.  We show that these two different types of  estimates are saturated on flat manifolds or manifolds all of whose sectional curvatures are negative.  This allows us to classify compact space forms in terms of the size of $L^q$-norms of quasimodes for
each Lebesgue exponent $q\in (2,q_c]$, even though it is impossible to distinguish between ones  of negative or zero curvature sectional curvature
 for any $q>q_c$.
\end{abstract}

\maketitle

\newsection{Introduction and main results}
This paper addresses the question of whether one can
``hear'' the ``shape'' of a connected compact manifold, if
``shape'' refers to the sign of its sectional curvatures.
We answer this question in the affirmative if the manifold is of constant sectional curvature and also if one
uses the correct type of ``radio''.  Specifically,
we shall show that there is a classification of the three genres of
manifolds
of constant sectional curvature (positive, zero and negative)  using the (sharp) growth rate of
$L^q$-norms of log-quasimodes if $q$ is any critical
or subcritical exponent for the universal bounds  of Sogge~\cite{sogge88}, even though, as we shall  review,
norms involving any supercritical exponent cannot 
distinguish between manifolds of negative curvature and flat manifolds.  So, 
``radios'' involving critical or subcritical exponents must be used for this problem.
Similarly, as we shall demonstrate, 
the different geometries exhibit different types of concentration of quasimodes near periodic geodesics, which turns out to be key to answering to this question.  The positive estimates that we obtain for manifolds all of whose sectional curvatures are nonpositive or negative, though, do not require the assumption of 
constant curvature.

The main step to accomplish this classification is to obtain sharp critical and subcritical $L^q$-estimates for log-quasimodes on compact
manifolds of negative and nonpositive sectional curvatures.  
 We improve the earlier estimates in 
\cite{BHSsp} and \cite{SBLog} and obtain, for the first time, different but sharp estimates for both types of 
geometries.  We also are able to characterize compact connected space forms in terms of the size of quasimodes measured
by any critical or subcritical $L^q(M)$-norm.  As we shall indicate, this is impossible to do for {\em any} exponent
in the supercritical range $q>q_c$, with, here and in what follows,
\begin{equation}\label{1.1}
q_c=2(n+1)/(n-1),
\end{equation}
denoting the critical exponent for our compact $n$-dimensional Riemannian manifold $(M,g)$.

Before stating our main results, let us review the local universal estimates of Sogge~\cite{sogge88}.  These concern the
Lebesgue norms of eigenfunctions and quasimodes with spectrum in unit intervals.  

If
$\Delta_g$ is the
Laplace-Beltrami operator associated with the metric $g$ on $M$, we let $0=\la_0<\la_1\le \la_2\le \cdots$ denote the
eigenvalues labeled with respect to multiplicity of the first order operator
$P=\sqrt{-\Delta_g}$
and $e_{\la_j}$ the associated $L^2$-normalized eigenfunctions.  So,
\begin{equation}\label{1.2}
(Pe_{\la_j})(x)=\la_j  e_{\la_j}(x), \quad
\text{and } \, \, \int_M |e_{\la_j}(x)|^2 \, dx=1, \, \, \, \text{if } \, \, P=\sqrt{-\Delta_g}.
\end{equation}
Here, abusing notation a bit, $dx$ denotes the Riemannian volume element, $dV_g$, and $\la_j$ also denotes the
frequency of the Laplace eigenfunctions, which means that $-\Delta_g e_{\la_j}=\la^2_{j} \, e_{\la_j}$.  If $I\subset [0,\infty)$ is an interval,
we shall be concerned with the associated spectral projection operators
\begin{equation}\label{1.3}
\bigl(\chi_I f\bigr)(x)= \sum_{\la_j\in I} E_jf(x),
\quad \text{with } \, \, 
E_jf(x)=  \Bigl(\,  \int_M f(y) \, \overline{e_{\la_j}(y)} \, dy \, \Bigr)\cdot e_{\la_j}(x).
\end{equation}
Also, we shall say that the spectrum of $f$ is in $I$ and write $\text{Spec }f\subset I$ if $E_jf=0$ for 
$\la_j\notin I$.

In \cite{sogge88}, the universal bounds for $q>2$
\begin{multline}\label{1.4}
\bigl\| \chi_{[\la,\la+1]}f\bigr\|_{L^q(M)}\lesssim\la^{\mu(q)} \, \|f\|_{L^2(M)}, \, \, \la\ge 1,
\\
\text{with } \, \,
\mu(q)=
\begin{cases} n(\tfrac12-\tfrac1q)-\tfrac12, \, \, q_c\le q\le \infty
\\
\tfrac{n-1}2(\tfrac12-\tfrac1q), \, \, 2<q\le q_c,
\end{cases}
\end{multline}
were obtained.  Note that when $q=q_c$ the above exponent is just
\begin{equation}\label{1.5}
\mu(q_c)=1/q_c.
\end{equation}
Here, and in what follows, $\lesssim$ refers to an inequality with an implicit, but unspecified constant.
As was shown in \cite{SFIO2} the unit-band spectral estimates \eqref{1.4} are sharp
on {\em any} compact manifold $(M,g)$, regardless of the geometry.

So, as in B\'erard~\cite{Berard} and Safarov~ \cite{Safarov},
 which seem to be the first papers in the program,
  as well as
Sogge and Zeldtich~ \cite{SoggeZelditchMaximal}, 
in order to obtain
improvements over the bounds in \eqref{1.4} under certain geometric assumptions, one must
replace the unit intervals $[\la,\la+1]$ by ones of the form
$[\la,\la+\delta(\la)]$ with $\delta(\la)\searrow 0$ as $\la\to \infty$.  For the most part, we shall take
\begin{equation}\label{1.6}
\delta(\la)=\bigl(\log\la\bigr)^{-1}, \quad \la \gg 1.
\end{equation}
Note that if 
\begin{equation}\label{1.66}
V_{[\la,\la+\delta(\la)]}=\{\Phi_\la: \, \, \text{Spec }\Phi_\la \subset [\la,\la+\delta(\la)]\}
\end{equation}
is the space of $\delta(\la)$-quasimodes, we of course have for $q>2$
\begin{equation}\label{1.67}
\sup_{\Phi_\la \in V_{[\la,\la+\delta(\la)]}}\frac{\|\Phi_\la\|_{L^q(M)}}{\|\Phi_\la\|_{L^2(M)}}=
\|\chi_{[\la,\la+\delta(\la)]}\|_{2\to q}, \quad \text{if } \, V_{[\la,\la+\delta(\la)]}\ne \emptyset.
\end{equation}
Here, and in what follows, $\| \, \cdot \, \|_{p\to q}$ denotes the $L^p(M)\to L^q(M)$ operator norm.

As in many earlier results, our attempts to relate properties of high energy eigenfunctions $e_{\la}$
in terms of the geodesic flow $\Phi_t: \, T^*M \backslash 0\to T^*M\backslash 0$ or the half-wave operators
$U_t=e^{itP}$ is limited by the role of the Ehrenfest time.  Recall that if $a(x,\xi)\in S^0_{1,0}$ is a 
zero-order symbol, and if $Op(a)$ is the associated zero-order pseudo-differential operator, then
Egorov's theorem from microlocal analysis says that $U_{-t} Op(a) U_t- Op(a\circ \Phi_t)$ is a
pseudo-differential operator of order $-1$.  Thus,  if, as above, $e_\la$ is an $L^2$-normalized eigenfunction of $P$ with
eigenvalue $\la$,
we have for small $|t|>0$ that  $\| [ U_{-t}Op(a)U_t - Op(a\circ \Phi_t)]e_\la\|_2 =O(\la^{-1})$.
 If the sectional curvatures of $M$ are negative, though, this estimate breaks down as
$t$ approaches a multiple of $\log\la$.  Indeed there exists an optimal $C_M$ so that for $t>0$
$$\| [ U_{-t}Op(a)U_t - Op(a\circ \Phi_t) ] e_\la \, \bigr\|_{2}
\lesssim \la^{-1} \, \exp(C_Mt).$$
  So, these  improved  estimates
coming from  Egorov's theorem only persist for $|t|$ smaller than the Ehrenfest time
$$T_E=T_E(\la)=\frac{\log\la}{C_M}.$$
Since we need to use this and related tools from microlocal analysis, we are naturally limited to studying spectral
projection operators $\chi_{[\la,\la+\delta(\la)]}$ with $\delta(\la)$ as in \eqref{1.6}.  We refer the reader
to  the excellent exposition in Zelditch~\cite{Zelquantum} for a more thorough discussion of the role of the Ehrenfest time
in settings where the geodesic flow is ergodic or chaotic, such as when $(M,g)$ is of negative
curvature, which is a main focus of this paper.

Keeping this in mind, let us state our main result.

\begin{theorem}\label{thm1.1}  Let $(M,g)$ be an $n$-dimensional connected compact Riemannian manifold.
Then, if all the sectional curvatures are {\em nonpositive}, for $\la\gg1$ we have the uniform bounds
\begin{equation}\label{1.7}
\bigl\| \chi_{[\la, \la+(\log\la)^{-1}]} f\bigr\|_{L^q(M)}\le C\bigl(\la (\log\la)^{-1}\bigr)^{\mu(q)} \|f\|_{L^2(M)}, \, \, \,
2<q\le q_c,
\end{equation}
with $q_c$ and $\mu(q)$ as in \eqref{1.1} and \eqref{1.4}, respectively.  Moreover, if all the sectional
curvatures of $M$ are {\em negative}, for $\la\gg 1$ we have the uniform bounds
\begin{equation}\label{1.8}
\bigl\| \chi_{[\la, \la+(\log\la)^{-1}]} f\bigr\|_{L^q(M)}\le C_q \, \la^{\mu(q)} (\log\la)^{-1/2} \|f\|_{L^2(M)}, \, \, \,
2<q\le q_c,
\end{equation}
with the constant $C_q$ in \eqref{1.8} depending on $q$.
\end{theorem}

As we shall see, the bounds in \eqref{1.7} are always sharp when $M$ is flat, and, as was shown in Blair,
Huang and Sogge ~\cite{BHSsp}, one cannot have stronger bounds
 if $M$ is a product manifold with $S^1$ as a factor.  
Moreover,   Germain and Myerson~\cite{Germain}  proved  negative results for all tori, including ones 
that are not such  a product manifold.  See also Hickman~\cite{Hickman} who proved bounds 
with  numerology essentially as in \eqref{1.7} on ${\mathbb T}^n$ for $q=q_c$.

The estimates in \eqref{1.8} are stronger than those in \eqref{1.7} since $\mu(q)<1/2$ for $2<q\le q_c$, and, hence
$(\log\la)^{-1/2}\ll (\log\la)^{-\mu(q)}$ if $\la \gg 1$.

We  point out that B\'erard~\cite{Berard} and Hassell and Tacy~\cite{HassellTacy}
had shown earlier that
if $M$ is any compact manifold with nonpositive sectional curvature, then the analog of \eqref{1.8} is valid for
all $q>q_c$. This was later generalized to a weaker dynamical assumption in the recent  work of Canzani and
Galkowski~\cite{CGGrowth} .  By well known arguments that we review in \S3 below, the $(\log\la)^{-1/2}$ gains that they obtained for log-quasimodes never
can be improved.
Thus, unlike the case in Theorem~\ref{thm1.1} which concerns critical and subcritical exponents, 
the bounds for supercritical exponents $q>q_c$ do {\em not} distinguish between the two geometries considered above,
i.e.,  compact manifolds of nonpositive or strictly negative sectional curvatures.

We would also like to point out that the analog of the bounds
in \eqref{1.8} do not hold in the Euclidean case when $2<q<q_c$.  These involve the Euclidean spectral projection operators
for $\sqrt{-\Delta_{\Rn}}$,
$$\chi_{[\la,\la+\delta(\la)]}f(x)=
(2\pi)^{-n}\int_{\{\xi\in \Rn: \, |\xi|\in [\la,\la+\delta(\la)]\}} e^{ix\cdot \xi} \Hat f(\xi)\, d\xi,$$
with
$$\Hat f(\xi)=\int_{\Rn} e^{-ix\cdot\xi}f(x)\, dx$$
denoting the Euclidean Fourier transform.  It is an easy exercise to see that the
Stein-Tomas~\cite{TomasWill} restriction theorem is equivalent to the estimates saying that
for $\la\gg1$ we have the uniform bounds
\begin{equation}\label{1.9}
\bigl\| \chi_{[\la, \la+(\log\la)^{-1}]} f\bigr\|_{L^q(\Rn)}\le C \la^{\mu(q)} (\log\la)^{-1/2} \|f\|_{L^2(\Rn)}, \, \, \,
q\in [q_c,\infty],
\end{equation}
with again $q_c$ and $\mu(q)$ as in \eqref{1.1} and \eqref{1.4}, respectively.
Conversely, unlike in \eqref{1.8}, the analog of \eqref{1.9} cannot hold {\em for any}
$q\in (2,q_c)$, since,  this would imply that there must be a
$L^{q'}(\Rn)\to L^2(S^{n-1})$ Fourier restriction theorem with $q'$ being the conjugate
exponent for $q$.  This is impossible by the standard Knapp example since the sharp range
for the $L^p(\Rn)\to L^2(S^{n-1})$ restriction bounds is the one treated in \cite{TomasWill},
that is $p\in [1,\tfrac{2(n+1)}{n+3}]$, which are the exponents that are dual to the exponents 
$q\in [q_c,\infty]$.   
Moreover, the standard Knapp construction for $\Rn$ shows that the bounds coming
from interpolating between the estimate in \eqref{1.9} for $q=q_c$ and the trivial $L^2\to L^2$ estimates, i.e.,
$$\|\chi_{[\la,\la+(\log\la)^{-1}]}\|_{2\to q}=O(\la^{\mu(q)}(\log\la)^{-\frac{n+1}2(\frac12-\frac1q)}),
\quad q\in (2,q_c),
$$
are sharp, and these are less favorable than the bounds in \eqref{1.8} since 
$\tfrac{n+1}2(\tfrac12-\tfrac1q)<\tfrac12$ for $2<q<q_c$.

It seems remarkable that in \eqref{1.8} we are able to obtain the optimal
$(\delta(\la))^{1/2}$, $\delta(\la)=(\log\la)^{-1}$, improvements over the $\delta(\la)\equiv 1$ bounds, since
such
improvements in the Euclidean case break down for any {\em subcritical} exponent
$q\in (2,q_c)$.  The improved bounds in \eqref{1.8} for these exponents is due to the chaotic
nature of the geodesic flow in manifolds with negative sectional curvature, which results in much
more favorable dispersive bounds for solutions of the wave equation in the universal cover that we shall
exploit.

The optimal and perhaps unexpected improvements in \eqref{1.8} can also perhaps be framed as an analog
for compact manifolds of negative curvature
of the classical and celebrated Kunze-Stein~\cite{KunzeStein} phenomenon.
Chen and Hassell~\cite{ChenHassell} established related results in the noncompact case in which they
showed that the analog of \eqref{1.9}  for hyperbolic space ${\mathbb H}^n$ is valid for all
exponents $q\in (2,\infty]$ (see also \cite{SHuangSogge}).  Thus, in this setting (as in  \eqref{1.8}) there is a natural analog of the
Stein-Tomas extension theorem {\em for all exponents} $q>2$.  It is notable that, although the sharp critical and subcritical log-quasimode estimates for
flat compact manifolds are weaker than the ones for Euclidean space, the results that we obtain for compact manifolds all of whose sectional curvatures are negative agree with the Chen and Hassell bounds for such quasimodes in  ${\mathbb H}^n$.

Using the bounds in Theorem~\ref{thm1.1} along with the fact that
the unit-band spectral projection bounds in \eqref{1.4} are always sharp (see \cite{SFIO2})
and a ``Knapp example'' that we present in \S4 for  connected compact flat manifolds, we are able to characterize all compact  connected space
forms in terms of the size of log-quasimodes as measured by
{\em critical and subcritical} Lebesgue exponents, which is a natural  affirmation of Bohr's quantum correspondence principle
(cf. \cite[\S1.5]{Zelbook}).

To state this result we recall that if $f,g\ge0$ then 
$f(\la)=\Theta(g(\la))$ if $\limsup_{\la\to \infty}\tfrac{f(\la)}{g(\la)}\in (0,\infty)$, i.e., $f(\la)=O(g(\la))$ and also
$f(\la)=\Omega(g(\la))$ (the negation of $f(\la)=o(g(\la)$).  Our other main result then is the following result which
says that compact manifolds of constant sectional curvature are characterized by the growth $L^q$-norms of log-quasimodes,
if  $q\in (2,q_c]$.

%

\begin{theorem}\label{thm1.2}  Assume that $(M,g)$ is a connected compact manifold of constant
sectional curvature $K$ and fix {\em any} exponent
$q\in (2,q_c]$.
Then, if $\mu(q)$ is as in \eqref{1.4},
\begin{multline}\label{shape}
 \sup \bigl\{ \|\Phi_\la\|_{L^q(M)}: \, \Phi_\la \in V_{[\la,\la+(\log)^{-1}]}, \, 
\|\Phi_\la\|_{L^2(M)}=1 \bigr\} 
\\
=\begin{cases}
\Theta(\la^{\mu(q)}(\log\la)^{-1/2}) \, \, \iff \, \, K<0
\\
\Theta(\la^{\mu(q)}(\log\la)^{-\mu(q)}) \, \, \iff \, \, K=0
\\
\Theta(\la^{\mu(q)}) \, \, \iff \, \, K>0.
\end{cases}
\end{multline}
Also, if $(\log\la)^{-1}\le \delta(\la)\searrow 0$ as $\la\to \infty$ and $\la\to \la \, \delta(\la)$ is non-decreasing for $\la\ge2$,
\begin{multline}\label{shape2}
 \sup \bigl\{ \|\Phi_\la\|_{L^q(M)}: \, \Phi_\la \in V_{[\la,\la+\delta(\la)]}, \, 
\|\Phi_\la\|_{L^2(M)}=1 \bigr\} 
\\
=\begin{cases}
\Theta(\la^{\mu(q)}(\delta(\la))^{1/2}) \, \, \iff \, \, K<0
\\
\Theta(\la^{\mu(q)}(\delta(\la))^{\mu(q)}) \, \, \iff \, \, K=0
\\
\Theta(\la^{\mu(q)}) \, \, \iff \, \, K>0.
\end{cases}
\end{multline}
\end{theorem}

Here we take the left sides of \eqref{shape} and \eqref{shape2} to be zero if $V_{[\la,\la+\delta(\la)]}=\emptyset$.

As we mentioned before, B\'erard~\cite{Berard} and Hassell and Tacy~\cite{HassellTacy} obtained the sharp bounds
$\|\chi_{[\la,\la+(\log\la)^{-1}]}\|_{2\to q}=O(\la^{\mu(q)}(\log\la)^{-1/2})$ for {\em all supercritical exponents} $q>q_c$
whenever all the sectional curvatures of $(M,g)$ are {\em nonpositive}.  Thus, in Theorem~\ref{thm1.2} we must
consider the range $q\in (2,q_c]$ of subcritical or critical exponents in order to distinguish between flat manifolds
and ones with negative sectional curvatures.  As we shall see in the proof of Theorem~\ref{thm1.2}, the range
$q\in (2,q_c]$ is very sensitive to different types of concentration near periodic geodesics.  On the other hand, the types of 
quasimodes saturating the estimates of B\'erard~\cite{Berard} and Hassell and Tacy~\cite{HassellTacy} concentrate
near points in a manner that is agnostic to the presence of nonpositive versus negative sectional curvatures.

Note that the universal bounds \eqref{1.4} as well as \eqref{1.7} and \eqref{1.8} along with \eqref{1.67}
imply that 
the left side of \eqref{shape} is $O(\la^{\mu(q)}(\log\la)^{-1/2})$ if $K<0$, $O(\la^{\mu(q)}(\log\la)^{-\mu(q)})$ if $K=0$
and $O(\la^{\mu(q)})$ if $K>0$.
  So, to prove the first assertion in Theorem~\ref{thm1.2} we need to show that we can replace each of these ``$O$'' bounds by 
  ``$\Omega$'' lower bounds
under the appropriate curvature assumption.  Obtaining such results  when $K>0$
is relatively easy since any compact connected
space form of positive curvature is the quotient of a round sphere, which allows
us to use results from \cite{sogge86}.  Proving the  $\Omega$-lower bounds for  \eqref{shape} when $K<0$ is also
relatively easy since, as we mentioned before, it is straightforward to see that the $(\delta(\la))^{1/2}$ improvement
in \eqref{1.8} cannot be strengthened.  Establishing the $\Omega$-lower bounds for flat manifolds 
is more difficult.  We do this by using the fact that a flat connected compact manifold must be of the form
$\Rn/\Gamma$ where the deck transformations, $\Gamma$, must be a Bieberbach subgroup of the Euclidean
group, $E(n)$, of rigid motions in $\Rn$ (see, e.g., \cite{bieberbach}, \cite{flat} and \cite{wolfconstant}).
This fact allows us to construct ``Knapp examples'' for flat manifolds using arguments from
Brooks~\cite{BrooksQM} and Sogge and Zelditch~\cite{SoggeZelditchL4} that show that the bounds
in \eqref{1.7} are always sharp for flat compact manifolds, which means that the left side of \eqref{shape} cannot
by $o(\la^{\mu(q)}(\log\la)^{-\mu(q)})$  when $K=0$.  Similar arguments will yield the second assertion \eqref{shape2}
in the theorem.

We should also point out that, although $q_c=\tfrac{2(n+1)}{n-1}$ is the critical exponent for the universal bounds
\eqref{1.9} for unit-band spectral projecction operators, surprisingly, by Theorem~\ref{thm1.2},  unless the sectional curvatures of $(M,g)$ are
positive
 it is not the ``critical exponent'' 
for projecting onto
$(\log\la)^{-1}$ or $\delta(\la)$ bands as in \eqref{shape2}, if, as is customary, a ``critical exponent'' is one for which the 
bounds for other exponents $q_c\ne q\in (2,\infty]$ follow from an interpolation argument using the trivial $L^2$-bounds
and dyadic Sobolev estimates for the ranges $(2,q_c)$ and $(q_c,\infty]$, respectively.

Let us say a few words about how our approach differs
from the earlier works of
Blair, Huang and Sogge~\cite{BHSsp},  Blair and Sogge~\cite{SBLog} and others.

The most difficult step in the proof of our main result, Theorem~\ref{thm1.1}, is to obtain and apply $\ell_\nu^qL^q_x$
estimates involving operators $\{Q_\nu\rho_\la\}_\nu$, where $\rho_\la$ is a smoothed out version of the spectral projection operators there and $\{Q_\nu\}$
is a microlocal partition of unity about geodesic
segments $\gamma_\nu \subset S^*M$, as in 
\cite{BHSsp}, \cite{BlairSoggeRefined}, \cite{SBLog}.  For
$q=q_c$, as in these works, we cannot prove such estimates over all of $M$, due to the way that very
different types of quasimodes may saturate $L^{q_c}(M)$-norms.  Instead, for this exponent, the estimates are taken over certain sublevel sets, as in 
\cite{BHSsp}.  On the other hand, in order to obtain the sharp estimates in Theorem 1.1 for $q\in (2,q_c)$ for manifolds of negative curvature, we require new 
$\ell_\nu^qL^q_x$ estimates 
involving relatively small, compared to $q_c$, exponents $q>2$
that were not considered before.
These are related to the original decoupling estimates
due to Wolff~\cite{wde}, and it is crucial that they do
not involve $\la^\e$ losses as is typically the case
for decoupling estimates. They are over all of $M$, 
unlike those for $q=q_c$.  We are able to obtain the
new $\ell_\nu^qL^q_x$, $q\in (2,2(n+2)/n]$, estimates needed for the
part of Theorem~1.1 concerning manifolds of negative
sectional curvature using a direct application of 
Lee's~\cite{LeeBilinear} bilinear oscillatory integral theorem, and, for the sake of completeness, we also show how
this argument can be modified to obtain the more
restrictive $\ell_\nu^{q_c}L^{q_c}_x$ estimates
from the earlier works \cite{BHSsp} and \cite{SBLog}.

It is a bit difficult to use the $\ell_\nu^q
L^q_x$ decoupling-type estimates involving the microlocalized operators $\{Q_\nu \rho_\la\}_\nu$.  This
is because, unlike in the Euclidean case
for Fourier extension operators (cf. \cite{BoDe}, \cite{TaoVargasVega}), or the related case for 
toral quasimodes (cf. \cite{BoDe}, \cite{Hickman}), 
the cutoffs $Q_\nu$ which we use do not commute with the
global operators $\rho_\la$ that we wish to estimate.
Commutator
terms arise from
the local variable coefficient analysis used to prove the $\ell^q_\nu L^q_x$ estimates,
and they also must be dealt with in order
to use the new $TT^*$ argument and resulting interpolation argument in
the final step of the proof.  In all of the earlier works concerning
manifolds of nonpositive or negative curvature
\cite{BHSsp}, \cite{BlairSoggeRefined}, \cite{blair2015refined}, \cite{SBLog}, \cite{SoggeZelditchL4} the authors dealt these issues by
exploiting the fact that the spectral projection
estimates being considered are $L^2_x\to L^q_x$ ones, which allowed the use of the $L^2$ almost orthogonality
of the $\{Q_\nu\}$ operators when applying the $\ell_\nu^qL^q_x$
estimates.  Even though the use of $L^2$ almost 
orthogonality was straightforward, given the nature
of the spectral projection estimates, it seems that this
approach could not lead to the sharp estimates in 
Theorem~1.1.  Indeed, for manifolds of nonpositive
curvature and $q=q_c$, the use of $L^2$ almost orthogonality 
in  our earlier joint work with Blair~\cite{BHSsp}
only lead to an improvement of
$(\log\la)^{-\e_n}$, $\e_n=\frac{2}{(n+1)q_c}\approx \frac1{n+1}$,
over the universal bounds \eqref{1.4}, compared to the
optimal result of $\e_n=1/q_c\approx 1/2$ in
Theorem~\ref{thm1.1}. Besides this, in \cite{BHSsp}
and earlier works,
 all of the estimates
for $q\in (2,q_c)$ were obtained by interpolating with
this $L^{q_c}$ bound and the trivial $L^2_x\to L^2_x$
estimates for $\rho_\la$.  Of course this straightforward
interpolation scheme can never lead to the optimal
$(\log\la)^{-1/2}$ gains in Theorem~1.1 for
manifolds of negative sectional curvature, even if one were 
using optimal $L^{q_c}$ bounds.

In the current work we are finally able to obtain sharp
log-quasimode estimates by introducing a few new ideas.
In addition to the new $\ell_\nu^qL^q_x$, $q\in (2,2(n+2)/n]$
estimates, a key difference is how we apply them, as
well as the more restrictive $\ell_\nu^{q_c}L^{q_c}_x$
variants described before.  One of the ways that we do this
is through by introducing a new ``$TT^*$ argument''
involving the microlocal cutoffs $\{Q_\nu\}$ 
and their adjoints $\{Q_{\nu'}^*\}$, which gives
rise to  new $\ell^{q'}_{\nu'}L^{q'}_x
\to \ell_\nu^qL^q_x$ estimates that also involve
natural dyadic decompositions of the $\rho_\la$ 
operators.  To prove the new $TT^*$ estimates, which is a key
step in the proof of Theorem~\ref{thm1.1}, we use introduce
a much more efficient interpolation scheme 
than the one in \cite{BHSsp} (and earlier works) that
we described before.  To be able to use the bounds that arise from this argument we also, unlike in earlier works, exploit the $L^q_x$ almost orthogonality
of the operators $\{Q_\nu\}$ as well as the dual version of these estimates, motivated by analogous
Euclidean arguments of Tao, Vargas and Vega~\cite{TaoVargasVega}.  We should also mention
that the new interpolation argument that we are using,
perhaps not surprisingly, now more closely resembles a key step in the proof of the Stein-Tomas restriction theorem
\cite{TomasWill}, which is the classical Euclidean analog
of our quasimode estimates.

We would also like to point out that our arguments
also finally provide a ``black box'' which could potentially lead to further sharp results where the 
curvature assumptions could perhaps be relaxed.  
Most of the estimates arising in the proof of
Theorem~\ref{thm1.1}, as we indicate below, are local and hold
on all compact manifolds.  The only ``global estimates''
that we use are the $\ell_\nu^1L^1_x \to \ell^\infty_\nu
L^\infty_x$ bounds for the microlocalized operators
used in the interpolation argument 
to obtain sharp results for $q\in (2,q_c)$,
while for 
$q=q_c$ one also needs the kernel estimates
\eqref{2.19} for the full non-microlocalized operators
$\rho_\la$.  One might expect the kernel estimates
for microlocalized operators that are needed to be
valid if the geodesic flow on $M$ is Anosov, and
perhaps the ``geodesic beams'' methods of Canzani and Galkowski~ \cite{CGGrowth}, \cite{CGInv},
could be used to
prove the bounds \eqref{2.19} for the non-microlocalized operators that are also needed for $q=q_c$ on a wider
class of manifolds than considered here.

Our paper is organized as follows.  In the next section we shall present the proof of 
Theorem~\ref{thm1.1} which requires global estimates that exploit
the curvature assumptions as well as local harmonic analysis estimates.  The global estimates come
from ones in B\'erard~\cite{Berard}, Hassell and Tacy~\cite{HassellTacy} and our earlier
works Blair and Sogge~\cite{SBLog}, Blair, Huang and Sogge~\cite{BHSsp} and Sogge~\cite{sogge2015improved}.
The local harmonic analysis estimates that we require will be proved in an appendix.  The ones that we need for the 
critical exponent were obtained earlier; however, the ones needed for the subcritical bounds in \eqref{1.8} require
 modifications of the earlier arguments that also give a simplified approach for handling those used for
$q=q_c$ both for manifolds of nonpositive and negative sectional curvatures.
In the third 
section we prove Theorem~\ref{thm1.2}.  As we indicated the main step will be to construct
Knapp examples for compact flat manifolds which involve quasimodes concentrating near a given periodic geodesic.
In \S4, we go over some new results that are a consequence of Theorem~\ref{thm1.1} and measure
concentration properties of quasimodes, such as lower bounds for $L^1$-norms originally studied by
Sogge and Zelditch~\cite{SoggeZelNodal} and later by Hezari and Sogge~\cite{HezariSogge} 
as a tool to analyze properties
of nodal sets of eigenfunctions.  In \S4 we also state
a couple of problems about possible generalizations of our results.

The second author is very grateful to his friend, the late Steve Zelditch, for numerous helpful discussions and debates about the
global harmonic analysis that we have developed
in our joint work and in work with other collaborators.  The authors are similarly grateful to Matthew Blair, and we are
also grateful for many helpful and generous conversations with William Minicozzi, including ones regarding properties
of compact space forms.  The authors also would like to thank the referee for very helpful comments and criticisms
that added to our exposition.

\newsection{Proving  log-quasimode estimates}

We shall now focus on proving the estimates in Theorem~\ref{thm1.1} for the critical exponent $q_c$ defined in \eqref{1.1}.  After we do this we shall
give the proof of the subcritical estimates which
is slightly easier.

In order to exploit calculations involving the half-wave operators we shall consider smoothed out spectral projection operators
of the form
\begin{equation}\label{2.1}
\rho_\la =\rho(T(\la-P)), \quad T=c_0\log\la, \quad P=\sqrt{-\Delta_g},
\end{equation}
where 
\begin{equation}\label{2.2}
\rho\in {\mathcal S}(\R), \, \, \rho(0)=1 \, \, \,
\text{and } \, \, \text{supp }\Hat \rho\subset \delta\cdot [1-\delta_0,1+\delta_0]
=[\delta-\delta_0\delta, \delta+\delta_0\delta],
\end{equation}
with $0<\delta, \delta_0<1/8$ (depending on $M$) to be specified later.  
We shall need that $\delta$ is smaller than the injectivity radius of $M$ and it also must be
chosen small enough so that the phase function $d_g(x,y)$ satisfies the Carelson-Sj\"olin condition when $d_g(x,y)\approx \delta$,
with $d_g(\, \cdot \, ,\, \cdot \, )$ denoting the Riemannian distance function.  Similarly we shall choose $\delta_0$ so that we can
use the bilinear oscillatory integral estimates that arise.
Lastly, in view of the earlier discussion of the Ehrenfest
time, we shall take
with $c_0>0$ in \eqref{2.1} to be eventually specified small enough depending on $(M,g)$.

By a simple orthogonality argument, we would obtain the bounds in \eqref{1.7} if we could show that, for $T$ as in
\eqref{2.2} and 
if the sectional curvatures of $(M,g)$ are nonpositive, then we have for $\la\gg1$ the uniform bounds
\begin{equation}\label{2.3}
\|\rho_\la f\|_{L^{q_c}(M)}\le C\bigl(\la(\log\la)^{-1}\bigr)^{1/q_c}\|f\|_{L^2(M)},
\end{equation}
since, as noted in \eqref{1.5}, $\mu(q_c)=1/q_c$.  Similarly, we would obtain the bounds in \eqref{1.8} for $q=q_c$ if we could 
show that when the sectional curvatures of $(M,g)$ are negative we have for $\la\gg1$
\begin{equation}\label{2.4}
\|\rho_\la f\|_{L^{q_c}(M)}\le C\la^{1/q_c} (\log\la)^{-1/2} \, \|f\|_{L^2(M)}.
\end{equation}

To prove these estimates we shall need to use local harmonic analysis involving the ``local operators''
\begin{equation}\label{2.5}
\sigma_\la =\rho(\la-P).
\end{equation}
It will be convenient to localize a bit more using microlocal cutoffs.  Specifically, let us write
\begin{equation}\label{2.6}
I=\sum_{j=1}^N B_j(x,D)
\end{equation}
where each $B_j\in S^0_{1,0}(M)$ is a zero order pseudo-differential operator with symbol supported in a small neighborhood
of some $(x_j,\xi_j)\in S^*M$.  The size of the support will be described shortly; however, we point out now that these operators
will not depend on the spectral parameter $\la \gg1$.

For present and future use, let us choose a Littlewood-Paley bump function satisfying
\begin{equation}\label{2.7}
\beta\in C^\infty_0((1/2,2)),   \, \, \, 
\beta(\tau)=1 \, \, \, \text{for } \, \tau \, \, \text{near } \, \, 1, \, \,
\text{and } \, \, \sum_{j=-\infty}^\infty \beta(2^{-j}\tau)=1, \, \, \tau>0.
\end{equation}
Then the dyadic operators
\begin{equation}\label{2.8}
B=B_{j,\la}=B_j\circ \beta(P/\la)
\end{equation}
are uniformly bounded on $L^p(M)$, i.e.,
\begin{equation}\label{2.9}
\|B\|_{p\to p}=O(1) \quad \text{for } \, \, 1\le p\le \infty.
\end{equation}
We use these dyadic  microlocal cutoffs to further localize $\sigma_\la$ as follows
\begin{equation}\label{2.10}
\tilde \sigma_\la =B\circ \sigma_\la,
\end{equation}
where $B$ is one of the $N$ operators coming from \eqref{2.6} and \eqref{2.8}.  
We also shall make use of the ``semi-global'' operators
\begin{equation}\label{2.11}
\tilde \rho_\la =\tilde \sigma_\la \circ \rho_\la.
\end{equation}

The $\sigma_\la$ are smoothed out versions of the operators in \eqref{1.4}.  They satisfy
the same operator norms, and the two sets of estimates are easily seen to be equivalent.
Similarly, it is easy to use orthogonality to see that the following uniform
bounds are valid for $q\in (2,\infty]$ and $\la \gg1$
\begin{multline}\label{2.12}
\bigl\|(I-\sigma_\la)\circ \rho_\la \bigr\|_{2\to q}\le CT^{-1}\la^{\mu(q)}, \, \, \text{if } \, T\ge 1
\\
\text{and }\, \, \bigl\|\sigma_\la -\beta(P/\la)\circ \sigma_\la\bigr\|_{2\to q}=O(\la^{-N}), \, \, \forall \, N.
\end{multline}
Consequently, by \eqref{2.6} and \eqref{2.12}, in order to prove \eqref{2.3} and \eqref{2.4} it suffices to show that
for $\la\gg1$ we have
\begin{equation}\label{2.13}
\|\tilde \rho_\la f\|_{L^{q_c}(M)}\le C \bigl(\la(\log\la)^{-1}\bigr)^{1/q_c} \, \|f\|_{L^2(M)}
\end{equation}
under the assumption that all of the sectional curvatures of $(M,g)$ are nonpositive, as well as
\begin{equation}\label{2.14}
\|\tilde \rho_\la f\|_{L^{q_c}(M)}\le C\la^{1/q_c}(\log\la)^{-1/2} \,  \|f\|_{L^2(M)}
\end{equation}
when the sectional curvatures of $(M,g)$ are all negative.

Recall (see e.g., \cite{sogge86}) that on the standard round sphere $S^n$ the improved bounds in
Theorem~\ref{thm1.1} cannot hold for $q=q_c$.  Indeed, the $L^2$-normalized zonal functions
$Z_\la$ and the Gaussian beams $G_\la$ (highest weight spherical harmonics) each have
$L^{q_c}(S^n)$ norms which are comparable to $\la^{1/q_c}$ if
$\la=\sqrt{k(k+n-1)}$, $k\in {\mathbb N}$, is a nonzero eigenvalue of $\sqrt{-\Delta_{S^n}}$.  Thus, in order
to prove Theorem~\ref{1.1} we need to rule out the existence of log-quasimodes under our geometric
assumptions that behave like the $Z_\la$ or $G_\la$.  The $G_\la$ satisfy $\|G_\la\|_{L^\infty(S^n)}\approx \la^{\frac{n-1}4}$ and
have negligible mass outside of a $\la^{-1/2+}$ tube about the equator.   Here, and in what 
follows, $\la^{\sigma+}$ refers to quantities involving $\la^{\sigma+\e}$ with $\e>0$ arbitrary, but with implicit constants of course depending on
$\e$.

Motivated by this, and to be able to use  the kernel estimates that B\'erard~\cite{Berard} and Hassell and Tacy~\cite{HassellTacy} used 
to prove their bounds for supercritical exponents and to also be able to utilize
bilinear oscillatory integral estimates from harmonic analysis, it is natural
to make a height decomposition using the semi-global operators $\tilde \rho_\la$ that essentially corresponds
to the ``height'' of the aforementioned Gaussian beams, $G_\la$.  We shall always assume, as we may, that the
function $f$ in \eqref{2.13} or \eqref{2.14} is $L^2$-normalized:
\begin{equation}\label{2.15}
\|f\|_{L^2(M)}=1.
\end{equation}
We then split our tasks \eqref{2.13} and \eqref{2.14} into estimating the $L^{q_c}$-norms of $\tilde \rho_\la f$ over the two regions
\begin{equation}\label{2.16}
A_+=\{x\in M: \, |\tilde \rho_\la f(x)|\ge \la^{\frac{n-1}4+\frac18}\}
\end{equation}
and
\begin{equation}\label{2.17}
A_-=\{x\in M: \, |\tilde \rho_\la f(x)|< \la^{\frac{n-1}4+\frac18}\}.
\end{equation}

Of course $M=A_+\cup A_-$.  Also, we point out that
there is nothing particularly special about the exponent $1/8$ in the above definitions.  It could be replaced by any sufficiently small 
positive exponent in what follows.
We choose the exponent $1/8$ to hopefully help the reader going through the numerology in some of the calculations that ensue.

Favorable $L^{q_c}(A_+)$ norms rule out quasimodes having large sup-norms as the zonal functions $Z_\la$ do on $S^n$;
however, they do not rule out the existence of modes behaving like the $G_\la$ since $\|G_\la\|_\infty
\ll \la^{\frac{n-1}4+\frac18}$ for $\la\gg1$.  On the other
hand, favorable $L^{q_c}(A_-)$ bounds do rule out log-quasimodes tightly concentrating near periodic geodesics as the 
$G_\la$ do on $S^n$.   Also, the bounds for the region $A_-$ are much harder than the ones for $A_+$ and require
bilinear techniques from harmonic analysis and the use of more refined microlocal cutoffs.  We also remark that when we turn to the improved $L^{q_c}(M)$-norms for $q\in (2,q_c)$
in \eqref{1.8}, we shall not have to make use of the splitting in \eqref{2.16}--\eqref{2.17}
since modes potentially behaving like the $Z_\la$ cannot
saturate subcritical norms.

\noindent{\bf 2.1. High floor estimates.}

Let us now show that for any $(M,g)$ all of whose sectional curvatures are nonpositive we have for $T$ as in \eqref{2.1}
with $c_0>0$ small enough
\begin{equation}\label{2.18}
\|\tilde \rho_\la f\|_{L^{q_c}(A_+)}\le C\la^{1/q_c}T^{-1/2}.
\end{equation}
Although this estimate is not taken over all of $M$, we note that the $T^{-1/2}\approx (\log\la)^{-1/2}$ improvement
matches up with that in \eqref{2.14} and is stronger than the $T^{-1/q_c}\approx (\log\la)^{-1/q_c}$ required for
\eqref{2.13}.  We shall present the simple proof of \eqref{2.18} for the sake of completeness even though this bound
was obtained in the earlier work \cite{BHSsp}, which in turn followed arguments in \cite{SBLog} and \cite{sogge2015improved}.

In order to prove \eqref{2.18} we need a global estimate for a kernel that will arise in a natural ``$TT^*$'' argument.  Specifically,
if $a\in C^\infty_0((-1,1))$ equals one on $(-1/2,1/2))$ we require the pointwise ``global'' kernel bound
\begin{multline}\label{2.19}
G_\la(x,y)=\frac1{2\pi}\int_{-\infty}^\infty \bigl(1-a(t)\bigr) \, 
T^{-1} \Hat \Psi(t/T) e^{it\la} \, \bigl(e^{-itP}\bigr)(x,y) \, dt
\\
=O(\la^{\frac{n-1}2}\exp(C_0T)), \quad \Psi=|\rho|^2, \, \, \, 1\le T\lesssim \log\la.
\end{multline}
This ``global'' estimate is valid whenever the sectional curvatures of $(M,g)$ are nonpositive
(see e.g. \cite{Berard}, \cite{HassellTacy}, \cite{SoggeHangzhou}, \cite{sogge2015improved}).
One proves the bound by standard arguments after lifting the calculation up to the universal cover
and then using the Hadamard parametrix, just as was done by B\'erard~\cite{Berard}.  Since $\Hat \Psi$ is
compactly supported, the number of terms in the sum that arises grows exponentially if $(M,g)$ has
negative curvature.  This accounts for the exponential factor in the right side of \eqref{2.19}
(in the spirit of our earlier discussion of the Ehrenfest time).  We also remark that such a bound cannot hold on $S^n$ 
due to the fact that the half-wave operators there are essentially periodic.  On the other hand, similar estimates
are known under weaker assumptions than nonpositive curvature, such as compact manifolds without conjugate points. See  the recent work of
Canzani and Galkowski~\cite{CGInv} for a detailed discussion.

To use the estimate \eqref{2.19} to prove \eqref{2.18} we first note that by \eqref{2.9}, \eqref{2.12} and \eqref{2.15}
$$\|\tilde \rho_\la f\|_{L^{q_c}(A_+)}\le \|B\rho_\la f\|_{L^{q_c}(A_+)}+C\la^{1/q_c}/\log\la.$$
Consequently, we would have \eqref{2.18} if we could show that
\begin{equation}\label{2.20}
\|B\rho_\la f\|_{L^{q_c}(A_+)}\le C\la^{1/q_c}(\log\la)^{-1/2} +\tfrac12 \|\tilde \rho_\la f\|_{L^{q_c}(A_+)}.
\end{equation}

To prove this we shall adapt an argument of Bourgain~\cite{BourgainBesicovitch} (see also \cite{BHSsp}
and \cite{sogge2015improved}).  So, we choose $g$ satisfying
$$\|g\|_{L^{q_c'}(A_+)}=1 \quad
\text{and } \, \|B\rho_\la f\|_{L^{q_c}(A_+)}=
\int B\rho_\la f \cdot \bigl(\one_{A_+}\cdot g\bigr)\, dx.
$$
Then since  $\Psi(T(\la-P))=\rho_\la \circ \rho^*_\la$ for 
$\Psi$ as in \eqref{2.19}, by \eqref{2.15} and the Schwarz inequality
\begin{align}\label{2.21}
\|B\rho_\la f\|^2_{L^{q_c}(A_+)}&= 
\Bigl(\, \int f\cdot (\rho^*_\la B^*)(\one_{A_+}\cdot g)(x) \, dx \, \Bigr)^2
\\
&\le \int |\rho^*_\la B^*(\one_{A_+}\cdot g)(x)|^2 \, dx \notag
\\
&=\int \bigl(B\circ \Psi(T(\la-P))\circ B^*\bigr)(\one_{A_+}\cdot g)(x) \, \cdot \, \overline{\one_{A_+}(x)g(x)} \, dx
\notag
\\
&=\int \bigl(B\circ L_\la \circ B^*\bigr)(\one_{A_+}\cdot g)(x) \, \cdot \,  \overline{\one_{A_+}(x)g(x)} \, dx
\notag 
\\
&\qquad+\int \bigl(B\circ G_\la \circ B^*\bigr)(\one_{A_+}\cdot g)(x) \, \cdot \,  \overline{\one_{A_+}(x)g(x)} \, dx
\notag 
\\
&=I+II. \notag
\end{align}
Here $G_\la$ is the operator whose kernel is in \eqref{2.19} and so
\begin{equation*}
L_\la =(2\pi T)^{-1}\int a(t)\Hat \Psi(t/T) \, e^{it\la} \, e^{-itP} \, dt.
\end{equation*}
Consequently, $L_\la h= T^{-1}\sum_j m(\la; \, \la_j)E_jh$, where $E_j$ denotes the projection onto the eigenspace
of $P$ with eigenvalue $\la_j$ and the spectral multiplier satisfies
$$m(\la; \, \la_j)=O(\, (1+|\la-\la_j|)^{-N}\, ), \quad \forall \, N.$$
Consequently, by \eqref{1.4},
$$\|L_\la\|_{q_c'\to q_c}\lesssim T^{-1}\la^{2/q_c}.$$
Since $T=c_0\log\la$, if we use H\"older's inequality and \eqref{2.9} we conclude that
\begin{align}\label{2.22}
|I|&\le \|BL_\la B^*(\one_{A_+}\cdot g)\|_{q_c} \cdot \|\one_{A_+}\cdot g\|_{q_c'}
\\
&\lesssim \|L_\la B^*(\one_{A_+}\cdot g)\|_{q_c}\cdot \|\one_{A_+}\cdot g\|_{q_c'}
\notag
\\
&\lesssim \la^{2/q_c}(\log\la)^{-1} \|B^*(\one_{A_+}\cdot g)\|_{q_c'}  \cdot \|\one_{A_+}\cdot g\|_{q_c'} \notag
\\
&\lesssim \la^{2/q_c}(\log\la)^{-1} \|g\|^2_{L^{q_c'}(A_+)} \notag
\\
&=\la^{2/q_c}(\log\la)^{-1}.  \notag
\end{align}

To estimate $II$, we choose $c_0>0$ small enough so that if $C_0$ is the constant in \eqref{2.19}
$$\exp(C_0T)\le \la^{1/8} \quad \text{if } \, \, T=c_0\log\la \, \, \text{and } \, \, \la\gg1.$$
As a result
$$\|G_\la\|_{1\to \infty}\lesssim \la^{\frac{n-1}2+\frac18}.$$
Consequently, since the dyadic operators $B$ are uniformly bounded on $L^1$ and $L^\infty$, if we
repeat the preceding argument we obtain
$$|II|\le C\la^{\frac{n-1}2 +\frac18}\|\one_{A_+}\cdot g\|^2_1
\le C\la^{\frac{n-1}2 +\frac18} \|g\|^2_{L^{q_c'}(A_+)} \cdot 
\|\one_{A_+}\|^2_{q_c}
=C\la^{\frac{n-1}2 +\frac18} \|\one_{A_+}\|^2_{q_c}.$$
After recalling the definition \eqref{2.16} we can estimate the last factor as follows
$$ \|\one_{A_+}\|^2_{q_c}\le \bigl(\la^{\frac{n-1}4+\frac18}\bigr)^{-2} \, \|\tilde \rho_\la f\|^2_{L^{q_c}(A_+)},
$$
which yields
$$|II|\lesssim \la^{-1/8}\|\tilde \rho_\la f\|^2_{L^{q_c}(A_+)} \le
\bigl(\tfrac12 \|\tilde \rho_\la f\|_{L^{q_c}(A_+)}\bigr)^2,$$
assuming, as we may, that $\la$ is large enough.

If we combine this bound with the earlier one \eqref{2.22} for $I$, we conclude that
\eqref{2.20} is valid, which concludes the proof of our estimates \eqref{2.18} involving
the large-height region $A_+$. \qed

\noindent{\bf 2.2. High ceiling estimates.}

To finish the proof of the estimates in Theorem~\ref{thm1.1} for $q=q_c$, in view of \eqref{2.18}, it suffices
to prove the following
\begin{proposition}\label{smallprop}
Suppose that all of the sectional curvatures of $(M,g)$ are nonpositive.  Then for $\la\gg1$ and $T$ as in \eqref{2.1} we have
\begin{equation}\label{2.23}
\|\tilde \rho_\la f\|_{L^{q_c}(A_-)}\le C (\la T^{-1})^{1/q_c},
\end{equation}
and if all the sectional curvatures are negative
\begin{equation}\label{2.24}
\|\tilde \rho_\la f\|_{L^{q_c}(A_-)}\le C \la^{1/q_c} \,  T^{-1/2}.
\end{equation}
\end{proposition}

To prove this proposition we
need to borrow and adapt results from the bilinear harmonic analysis in \cite{LeeBilinear} and
\cite{TaoVargasVega}.

We shall utilize a microlocal decomposition which we shall now describe.
We first recall that the symbol $B(x,\xi)$ of $B$ in \eqref{2.8} is supported in a small
conic neighborhood of some $(x_0,\xi_0)\in S^*M$.  We may assume that its symbol has
small enough support so that we may work in a coordinate chart $\Omega$ and that
$x_0=0$, $\xi_0=(0,\dots,0,1)$ and $g_{jk}(0)=\delta^j_k$ in the local coordinates.
So, we shall assume that $B(x,\xi)=0$ when $x$ is outside a small relatively compact neighborhood
of the origin or $\xi$ is outside of a small conic neighborhood of $(0,\dots,0,1)$.  These reductions
and those that follow will contribute to the number of terms in \eqref{2.6}; however, it will be clear
that the $N$ there will be independent of $\la\gg 1$.  Similarly, the positive numbers $\delta$ and
$\delta_0$ in \eqref{2.2} may depend on the summand in \eqref{2.6}, but, at the end we can just take each to be the minimum
of what is required for each $j=1,\dots,N$.

Next, let us define the microlocal cutoffs that we shall use.   We fix a function
$a\in C^\infty_0({\mathbb R}^{2(n-1)})$ supported in $\{z: \, |z_k|\le 1, \, \, 1\le k\le 2(n-1)\}$
 which satisfies
\begin{equation}\label{m1}
\sum_{j\in {\mathbb Z}^{2(n-1)}}a(z-j)\equiv 1.
\end{equation}
We shall use this function to build our microlocal cutoffs.
By the above, we shall focus on defining them 
 for $(y,\eta)\in S^*\Omega$ with    $y$ near the origin
 and  $\eta$ in a small conic neighborhood of $(0,\dots,0,1)$. 
We shall let
$$\Pi=\{y: \, y_{n}=0\}$$
be the points in $\Omega$ whose last coordinate vanishes.   Let $y'=(y_1,\dots, y_{n-1})$ and
$\eta'=(\eta_1,\dots,\eta_{n-1})$ denote the first $n-1$ coordinates of $y$ and $\eta$, respectively.
 For $y\in \Pi$ near $0$ and $\eta$ near $(0,\dots,0,1)$ we can
just use the functions $a(\theta^{-1}(y',\eta')-j)$, $j\in {\mathbb Z}^{2(n-1)}$ to obtain cutoffs of scale $\theta$.  We will always have
$\theta\in [\la^{-1/8},1]$.

We can then extend the definition to a neighborhood of $(0,(0,\dots,0,1))$ by setting for $(x,\xi)\in S^*\Omega$ in this neighborhood
\begin{equation}\label{m2}
a^\theta_j(x,\xi)=a(\theta^{-1}(y',\eta')-j) \quad
\text{if } \, \, \Phi_s(x,\xi)=(y',0,\eta',\eta_{n}) \, \, \, \text{with } \, \, \, s=d_g(x,\Pi).
\end{equation}
Here $\Phi_s$ denotes geodesic flow in $S^*\Omega$.  Thus, $a^\theta_j(x,\xi)$ is constant on all geodesics
$(x(s),\xi(s))\in S^*\Omega$ with $x(0)\in \Pi$ near $0$ and $\xi(0)$ near $(0,\dots,0,1)$.   As a result,
\begin{equation}\label{m3}
a^\theta_j(\Phi_s(x,\xi))=a^\theta_j(x,\xi)
\end{equation}
for $s$ near $0$ and $(x,\xi)\in S^*\Omega$ near $(0,(0,\dots,0,1))$.

We then extend the definition of the cutoffs to a conic neighborhood of $(0,(0,\dots,0,1))$  in $T^*\Omega \, \backslash \, 0$ by setting
\begin{equation}\label{m4}
a^\theta_j(x,\xi)=a^\theta_j(x,\xi/p(x,\xi)).
\end{equation}

Notice that if $(y'_\nu,\eta'_\nu)=\theta j=\nu$ and $\gamma_\nu$ is the geodesic in $S^*\Omega$ passing through $(y'_\nu,0,\eta_\nu)\in S^*\Omega$
with $\eta_\nu\in S^*_{(y'_\nu,0)}\Omega$ having $\eta'_\nu$ as its first $(n-1)$ coordinates then
\begin{equation}\label{m5}
a^\theta_j(x,\xi)=0 \quad \text{if } \, \, \,
\text{dist }\bigl((x,\xi), \gamma_\nu\bigr)\ge C_0\theta,
\, \, \nu=\theta j,
\end{equation}
for some fixed constant $C_0>0$.  Also,  $a^\theta_j$ satisfies the estimates
\begin{equation}\label{m6}
\bigl|\partial_x^\sigma \partial_\xi^\gamma a^\theta_j(x,\xi)\bigr| \lesssim \theta^{-|\sigma|-|\gamma|}, \, \, \,
(x,\xi)\in S^*\Omega
\end{equation}
related to this support property.

Finally, if $\psi \in C^\infty_0(\Omega)$ equals
one in a neighborhood of the $x$-support of $B(x,\xi)$,
and if $\tilde \beta\in C^\infty_0((0,\infty))$ equals one in a neighborhood of the support of the Littlewood-Paley bump function
in \eqref{2.7} we define
\begin{equation}\label{qnusymbol}
A_\nu^\theta(x,\xi)=\psi(x) \, a_j^\theta(x,\xi) \, \tilde\beta\bigl(p(x,\xi)/\la\bigr),
\quad \nu =\theta j\in \theta \cdot {\mathbb Z}^{2(n-1)}.
\end{equation}
It then follows that the pseudo-differential operators
$A_\nu^\theta(x,D)$ with these symbols belong to a bounded subset
of $S^0_{7/8,1/8}(M)$,
due to our assumption that $\theta\in [\la^{-1/8},1]$.  We have constructed these operators so that
for small enough $\delta>0$ we have
\begin{equation}\label{2.32}A_\nu^\theta(x,\xi)=A_\nu^\theta(\Phi_t(x,\xi)) \quad
\text{on \, supp }B(x,\xi) \, \, \text{if } \, \,
|t|\le 2\delta.\end{equation}

We shall need a few simple but very useful facts about these operators:

\begin{lemma}\label{alemma}  Let $\theta_0=\la^{-1/8}$.  Then
\begin{align}\label{2.33}
\|A^{\theta_0}_\nu h\|_{\ell_\nu^q L^q(M)} &\lesssim \|h\|_{L^q(M)}, \quad 2\le q\le \infty,
\\
\bigl\|\sum_{\nu'}(A^{\theta_0}_{\nu'})^* H(\nu',\, \cdot\, )\bigr\|_{L^p(M)} &\lesssim \|H\|_{\ell^p_{\nu'}L^p(M)}, \quad 1\le p\le 2.
\label{2.34}
\end{align}
Also, if $\delta>0$ in \eqref{2.2} is small enough and $\mu(q)$ is as in \eqref{1.4}
\begin{equation}\label{commute}
\|B\sigma_\la A^{\theta_0}_\nu - BA^{\theta_0}_\nu \sigma_\la \|_{2\to q}=O(\la^{\mu(q)-\frac14}), \quad q\in (2,q_c].
\end{equation}
\end{lemma}

\begin{proof}
To prove \eqref{2.33} we note that, by interpolation, it suffices to prove the inequality for $q=2$ and $q=\infty$.
The estimate for $q=2$ just follows from the fact that the $S^0_{7/8,1/8}$ operators $\{A^{\theta_0}_\nu\}$ are almost
orthogonal due to \eqref{m1}.  The estimate for $q=\infty$ follows from the fact that the kernels satisfy
\begin{equation}\label{2.36}
\sup_x \int |A^{\theta_0}_\nu (x,y)|\, dy \le C.
\end{equation} 
To see \eqref{2.36}, we note that in addition to \eqref{m6} we have that $\partial_r^k a(x,r\omega)=\la^{-k}$, if $\omega\in S^{n-1}$, 
by \eqref{qnusymbol} since $a^\theta_\nu(x,\xi)$ is homogeneous of degree zero in $\xi$.
So, if, as we shall shortly do, we work in local Fermi normal coordinates so that the projection of $\gamma_\nu$ onto $M$ is 
the $n$th coordinate axis, using these estimates for radial derivatives, \eqref{m6} and a simple integration by parts argument yields that
$$A^{\theta_0}_\nu(x,y)=
O\bigl(\la^{\frac{7(n-1)}8+1} (1+\la^{7/8}|(x'-y'|)^{-N}(1+\la|x_n-y_n|)^{-N})\bigr),
$$ for all $N$
which of course yields \eqref{2.36}.

Since \eqref{2.34} follows via duality from \eqref{2.33} we are just left with proving \eqref{commute}.
To do this we recall that 
by \eqref{2.8} the symbol $B(x,\xi)=B_\la(x,\xi)
\in S^{0}_{1,0}$ vanishes when $|\xi|$ is not
comparable to $\la$.  In particular, it vanishes if $|\xi|$ is larger than a fixed
multiple of $\la$, and it belongs to a bounded subset of $S^0_{1,0}$.
Furthermore, if $a^{\theta_0}_\nu(x,\xi)$ is the principal
symbol of our zero-order dyadic microlocal operators, we recall  that by \eqref{2.32}
we have that for $\delta>0$ small enough
\begin{equation}\label{cc2}
a^{\theta_0}_\nu(x,\xi)=a^{\theta_0}_\nu(\Phi_t(x,\xi))
\quad \text{on supp } \, B_\la \, \, \, 
\text{if } \, \, |t|\le 2\delta,
\end{equation}
where $\Phi_t: \, T^*M\, \backslash 0 \to T^*M \, \backslash 0$ denotes
geodesic flow in the cotangent bundle.

By Sobolev estimates for $M$, in order
to prove \eqref{commute}, it suffices to show that for $q\in (2,q_c]$
\begin{equation}\label{cc4}
\bigl\|\,
\bigl(\sqrt{I+P^2} \, \, \bigr)^{n(\frac12-\frac1{q})}
\,
\bigl[ B_\la \sigma_\la A^{\theta_0}_\nu
-B_\la A^{\theta_0}_\nu \sigma_\la \bigr]
\, 
\bigr\|_{2\to 2}
=
O(\la^{\mu(q)-\frac14}).
\end{equation}

To prove this we recall that by  \eqref{2.2} and \eqref{2.5}
$$\sigma_\la=(2\pi)^{-1}\int^{2\delta}_{-2\delta} \Hat \rho(t) e^{it\la} e^{-itP}\, dt.
$$
Therefore by
Minkowski's integral inequality,
we would have
\eqref{cc4} if 
\begin{equation}\label{cc5}
\sup_{|t|\le 2\delta}\, 
\bigl\|\,
\bigl(\sqrt{I+P^2}  \, \, \bigr)^{n(\frac12-\frac1{q})}
 \, 
\bigl[ B_\la e^{-itP} A^{\theta_0}_\nu
-B_\la A^{\theta_0}_\nu e^{-itP} \bigr]
\, 
\bigr\|_{2\to 2}
=
O(\la^{\mu(q)-\frac14}).
\end{equation}

Next, to be able to use Egorov's theorem, we write
$$\bigl[ B_\la e^{-itP} A^{\theta_0}_\nu
-B_\la A^{\theta_0}_\nu e^{-itP} \bigr]
=B_\la
\, \bigl[
(e^{-itP}A^{\theta_0}_\nu e^{itP}) - B_\la A^{\theta_0}_\nu]
\circ e^{-itP}.
$$
Since $e^{-itP}$ also has $L^2$-operator norm one, we
would obtain \eqref{cc5} from
\begin{equation}\label{cc6}
\sup_{|t|\le 2\delta} \, 
\bigl\|\,
\bigl(\sqrt{I+P^2} \, \, \bigr)^{n(\frac12-\frac1{q})}
 \, 
B_\la
\, \bigl[
(e^{-itP}A^{\theta_0}_\nu e^{itP}) - A^{\theta_0}_\nu\bigr]
\, 
\bigr\|_{2\to 2} 
=
O(\la^{\mu(q)-\frac14}).
\end{equation}

By Egorov's theorem (see e.g. Taylor~\cite[\S VIII.1]{TaylorPDO})
$$A^{\theta_0}_{\nu,t}(x,D)= e^{-itP}A^{\theta_0}_\nu e^{itP}$$
is a one-parameter family of 
zero-order  pseudo-differential operators, depending on the parameter $t$, whose principal symbol is 
$a^{\theta_0}_\nu(\Phi_{-t}(x,\xi))$.  By \eqref{cc2}
 and the composition calculus
of pseudo-differential operators the principal
symbol of $B_\la A_{\nu,t}^{\theta_0}$ and $B_\la A^{\theta_0}_\nu$ both
equal $B_\la(x,\xi)a^{\theta_0}_\nu(x,\xi)$ if $|t|\le 2\delta$.  If $\theta=1$ then $A^\theta_\nu
\in S^0_{1,0}$, and, so, in this  case
we would have that $B_\la (e^{-itP}A^{\theta}_\nu e^{itP})
-B_\la A^\theta_\nu$ would be a pseudo-differential
operator of order $-1$ with symbol vanishing
for $|\xi|$ larger than a fixed multiple of $\la$ (see e.g., \cite[Theorem 4.3.6]{SoggeHangzhou}).
Since we are assuming that $\theta_0=\la^{-1/8}$, by the way they were constructed, the symbols $A^{\theta_0}_\nu$
belong to a bounded subset of $S^0_{7/8,1/8}$.  So, by \cite[p. 147]{TaylorPDO}, for $|t|\le 2\delta$,
$B_\la (e^{-itP}A^{\theta_0}_\nu e^{itP})
-B_\la A^{\theta_0}_\nu$ belong to a bounded subset of
$S^{-3/4}_{7/8,1/8}$ with symbols vanishing
for $|\xi|$ larger than a fixed multiple of
$\la$ due to the fact that the symbol  $B_\la(x,\xi)$ has this property
(see e.g., \cite[p. 46]{TaylorPDO}).

We also need to take into account the other operator
inside the norm in \eqref{cc6}.  Since $(\sqrt{I+P^2})^{n(\frac12-\frac1q)}$ is a standard pseudo-differential operator
of order $n(\tfrac12-\tfrac1q)$
 the operators in the left of \eqref{cc6} belong to 
a bounded subset of 
$S^{n(\frac12-\frac1{q})-\frac34}_{7/8,1/8}
(M)$ with symbols vanishing
for $|\xi|$ larger than a fixed multiple of
$\la$.  Consequently, the left side of
\eqref{cc6} is $O(\la^{n(\frac12-\frac1{q})-\frac34})$.
Since $\mu(q)=\tfrac{n-1}2(\tfrac12-\tfrac1q)$ for $q\in (2,q_c]$,
a simple calculation shows that $n(\tfrac12-\tfrac1q)-\tfrac34\le \mu(q)-\tfrac14$ for such
$q$, which yields \eqref{commute} and completes the proof of the lemma.
\end{proof}

Next we note that by \eqref{m1}, \eqref{m2} and \eqref{qnusymbol}, we have that, as operators between any $L^p(M)\to L^q(M)$ spaces,
$1\le p,q\le \infty$, for $\theta\ge \la^{-1/8}$
\begin{equation}\label{m11}
\tilde \sigma_\la =\sum_\nu \tilde \sigma_\la A^\theta_\nu +O(\la^{-N}), \, \forall \, N.
\end{equation}
This just follows from the fact that $R(x,D)=I-\sum_\nu A^\theta_\nu \in S^0_{7/8,1/8}$ has symbol
supported outside of a neighborhood of $B(x,\xi)$, if, as we may, we assume that the latter is small.

In view of \eqref{m11}, if $\delta>0$ in \eqref{2.2} is small enough,
 we have for $\theta_0=\la^{-1/8}$
\begin{equation}\label{m13}
\bigl(\tilde \sigma_\la h\bigr)^2=\sum_{\nu,\nu'} \bigl(\tilde \sigma_\la A^{\theta_0}_\nu h\bigr) \,
\bigl(\tilde \sigma_\la A^{\theta_0}_{\nu'} h\bigr) \, + \, O(\la^{-N}\|h\|_2^2).
\end{equation}

If
$\theta_0=\la^{-1/8}$ then the $\nu=\theta_0\cdot  {\mathbb Z}^{2(n-1)}$
 index a $\la^{-1/8}$-separated set in
${\mathbb R}^{2(n-1)}$.  We need to organize the pairs of indices $\nu,\nu'$ in \eqref{m13} as in many earlier works
(see \cite{LeeBilinear} and \cite{TaoVargasVega}).  We consider dyadic cubes $\tau^\theta_\mu$ in 
${\mathbb R}^{2(n-1)}$ of side length $\theta=2^k\theta_0 = 2^k\la^{-1/8}$, $k=0,1,\dots$, with
$\tau^\theta_\mu$ denoting translations of the cube $[0,\theta)^{2(n-1)}$ by
$\mu=\theta{\mathbb Z}^{2(n-1)}$.  Then two such dyadic cubes of side length $\theta$ are said to be
{\em close} if they are not adjacent but have adjacent parents of side length $2\theta$, and, in that case, we write
$\tau^\theta_\mu \sim \tau^\theta_{\mu'}$.  Note that close cubes satisfy $\text{dist }(\tau^\theta_\mu,\tau^\theta_{\mu'})
\approx \theta$ and so each fixed cube has $O(1)$ cubes which are ``close'' to it.  Moreover, as noted in \cite{TaoVargasVega},
any distinct points $\nu,\nu'\in {\mathbb R}^{2(n-1)}$ must lie in a unique pair of close cubes in this Whitney decomposition
of ${\mathbb R}^{2(n-1)}$.  Consequently, there must be a unique triple $(\theta=\theta_02^k, \mu,\mu')$ such that
$(\nu,\nu')\in \tau^\theta_\mu\times \tau^\theta_{\mu'}$ and $\tau^\theta_\mu\sim \tau^\theta_{\mu'}$.  We remark that by choosing $B$
to have small support we need only consider $\theta=2^k\theta_0\ll 1$.

Taking these observations into account implies that that the bilinear sum in \eqref{m13} can be organized as follows:
\begin{multline}\label{m14}
\sum_{\{k\in {\mathbb N}: \, k\ge 10 \, \, \text{and } \, 
\theta=2^k\theta_0\ll 1\}}
\sum_{\{(\mu, \mu'): \, \tau^\theta_\mu
\sim \tau^\theta_{ \mu'}\}}
\sum_{\{(\nu, \nu')\in
\tau^\theta_\mu\times \tau^\theta_{ \mu'}\}}
\bigl(\tilde \sigma_\la
A^{\theta_0}_\nu h\bigr) 
\cdot \bigl(\tilde \sigma_\la
A^{\theta_0}_{ \nu'} h\bigr)
\\
+\sum_{(\nu, \nu')\in \Xi_{\theta_0}} 
\bigl( \tilde \sigma_\la A^{\theta_0}_\nu h \bigr) 
\cdot \bigl( \tilde \sigma_\la
A^{\theta_0}_{ \nu'} 
h\bigr)
,
\end{multline}
where $\Xi_{\theta_0}$ indexes the remaining pairs such
that $|\nu- \nu'|\lesssim \theta_0=\la^{-1/8}$,
including the diagonal ones where $\nu= \nu'$.

As above, let $\mu(q)=\tfrac{n-1}2(\tfrac12-\tfrac1q)$, $q\in (2,q_c]$ be the exponent in the universal bounds \eqref{1.4}.
Then the key estimate that we shall use and follows from variable coefficient bilinear harmonic analysis arguments then is the following.

\begin{proposition}\label{locprop}
If $n\ge2$ and $\theta_0=\la^{-1/8}$, $\la\gg1$ and if \eqref{2.15} is valid
\begin{equation}\label{2.44}
\|\tilde \sigma_\la h\|_{L^{q_c}(A_-)}\lesssim \Bigl(\, \sum_\nu \|\tilde \sigma_\la A^{\theta_0}_\nu h\|_{L^{q_c}(M)}^{q_c}\, 
\Bigr)^{1/q_c} + \la^{\frac1{q_c}-}, \quad \text{if } \,  h=\rho_\la f,
\end{equation}
assuming that the conic support of $B(x,\xi)$ in \eqref{2.8} is small and that $\delta$ and $\delta_0$ in \eqref{2.2} are also small.
Also, if $2<q\le \tfrac{2(n+2)}n$,
\begin{equation}\label{2.45}
\|\tilde \sigma_\la h\|_{L^{q}(M)}\lesssim \Bigl(\, \sum_\nu \|\tilde \sigma_\la A^{\theta_0}_\nu h\|_{L^{q}(M)}^{q}\, 
\Bigr)^{1/q} + \la^{\mu(q)-}\|h\|_{L^2(M)}.
\end{equation}
\end{proposition}

Here, $\la^{\mu-}$ means a factor involving an unspecified exponent smaller than $\mu$.  Note that $\la$-power gains
are much better than the $\log\la$-power gains in Theorem~\ref{thm1.1}.

The first estimate, \eqref{2.44}, occurred in earlier works (\cite{BlairSoggeRefined}, \cite{blair2015refined}
and \cite{SBLog}) and requires that the norms in the left be taken over $A_-$ and that
$h=\rho_\la f$ so that $\tilde \sigma_\la h=\tilde \rho_\la f$.  The other estimate involving smaller Lebesgue exponents is new and does not require these restrictions.
For the sake of completeness, we shall present the proofs in an appendix.  Our proofs which are based on
a more direct application of Lee's \cite{LeeBilinear} bilinear inequality and the above Whitney decomposition not only 
yield \eqref{2.45}, but also give a simpler and self-contained proof of \eqref{2.44} which we shall present in this appendix.

We need to assemble two more ingredients which, along with \eqref{2.44} and Lemma~\ref{alemma},  will easily allow us to prove Proposition~\ref{smallprop}.
This will 
 give
us the bounds for $q=q_c$ in Theorem~\ref{thm1.1}, and, as we shall see, we will be able to easily obtain
 the nontrivial subcritical estimates in this theorem using \eqref{2.45} and these methods.

The first of these ingredients involves a dyadic decomposition of the ``global'' operator $G_\la$ defined by \eqref{2.19},
which involves the Littlewood-Paley $\beta\in C^\infty_0((1/2,2))$ described in \eqref{2.7}.
As in \eqref{2.19} we let 
 $\Psi=|\rho|^2$.  Its Fourier transform $\Hat \Psi$
then is compactly supported.  By \eqref{2.2}, we may assume that
$\Hat \Psi(t)=0$, $|t|>1/2$.
 Also,
let $\beta_0(s)=1-\sum_{j=1}^\infty \beta(s/2^j)$,
$s>0$ and $\beta_0(0)=1$ so that
$\beta_0(|s|)$ equals one near the origin
and is in $C^\infty_0(\R)$.

If we then let
\begin{equation}\label{local}
L_{\la,T}=(2\pi T)^{-1} \int\beta_0(|t|) \Hat \Psi(t/T)
e^{it\la} e^{-itP}\, dt,
\end{equation}
and   
\begin{equation}\label{global}
G_{\la,T,N}= (2\pi T)^{-1} \int\beta(|t|/N) \Hat \Psi(t/T)
e^{it\la} e^{-itP}\, dt, \quad N=2^j, \, \,
j\in {\mathbb N}.
\end{equation}
It then follows that
$G_{\la,N}=0$ if $N>T$, and, moreover,
\begin{equation}\label{sum}
\rho_\la = L_{\la,T}+\sum_{2\le N=2^j\le T}
G_{\la,T,N}.
\end{equation}
By the universal bounds \eqref{1.4}, we of course
have
\begin{equation}\label{loc}
\|L_{\la,T}\|_{q'\to q}=O(T^{-1}\la^{2\mu(q)}), \quad q>2,
\end{equation}
as we essentially used in the proof of \eqref{2.18}.

Then, in addition to Lemma~\ref{alemma} and Proposition~\ref{locprop}, the
next key ingredient needed to prove Theorem~\ref{thm1.1} is the following.

\begin{proposition}\label{kerprop} Let $\theta_0=\la^{-1/8}$ and
 assume that
for $T=c_0\log\la$ as in \eqref{2.1} 
we have the following bounds for
the microlocalized kernels
\begin{equation}\label{ker}
|\bigl(A_\nu^{\theta_0} G_{\la,T,N}(A_{\nu'}^{\theta_0})^*\bigr)(x,y)|
\le CT^{-1}\la^{\frac{n-1}2}N^{1-\alpha}, \, \, N=2^j, \, \, j\in {\mathbb N}.
\end{equation}
We then have 
\begin{equation}\label{est}
\|A_\nu^{\theta_0} \rho_\la f\|_{\ell^{q_c}_\nu L^{q_c}(M)}
\le C \la^{\frac1{q_c}} \|f\|_{L^2(M)} \cdot
\begin{cases}
T^{-\frac\alpha{n+1}}, \quad \text{if } \, 
\alpha <\tfrac{n+1}2
\\
T^{-\frac12}, \quad \text{if } \, \alpha>\frac{n+1}2.
\end{cases}
\end{equation}
Also, if $q\in (2,q_c)$ we have
\begin{equation}\label{est2}
\|A_\nu^{\theta_0} \rho_\la f\|_{\ell_\nu^q L^q(M)}\le C\la^{\mu(q)} T^{-1/2} \|f\|_{L^2(M)}, \quad
\text{if } \, \alpha>\tfrac{q}{q-2}.
\end{equation}
\end{proposition}

We shall momentarily postpone the simple proof of this Proposition and record one last result that
we need to prove our main estimates.

\begin{lemma}\label{globalker}  Fix a compact manifold $(M,g)$ all of whose sectional curvatures are
nonpositive.  
Then if $T=c_0\log\la$ is as in \eqref{2.1} with $c_0>0$ small enough
we have for $\la\gg1$
\begin{equation}\label{kernonpos}
|\bigl(A^{\theta_0}_\nu G_{\la,T,N}(A_{\nu'}^{\theta_0})^*\bigr)(x,y)|
\le CT^{-1}\la^{\frac{n-1}2}N^{1-\frac{n-1}2}, \quad N\in {\mathbb N}.
\end{equation}
Moreover, if all of the sectional curvatures of $(M,g)$ are negative we have for such $c_0>0$
\begin{equation}\label{kerneg}
|\bigl(A^{\theta_0}_\nu G_{\la,T,N}(A_{\nu'}^{\theta_0})^*\bigr)(x,y)|
\le C_mT^{-1}\la^{\frac{n-1}2}N^{1-m}, \, \,  N\in {\mathbb N},
\, \,  \text{for each } \, m=1,2,\dots.
\end{equation}
\end{lemma}

At the end of the section we shall recall the proof of \eqref{kernonpos} and \eqref{kerneg} which
were obtained in the earlier works \cite{BHSsp} and \cite{SBLog}.

Having assembled all the necessary ingredients, let us now prove Proposition~\ref{smallprop} which, as we noted
before, would complete the proof of the estimates for $q=q_c$ in Theorem~\ref{thm1.1}.

\begin{proof}[Proof of Proposition~\ref{smallprop}]
We first note that by \eqref{2.11} and \eqref{2.44} we have
\begin{equation}\label{2.55}
\|\tilde \rho_\la f\|_{L^{q_c}(A_-)} \lesssim 
\Bigl(\, \sum_\nu \|\tilde \sigma_\la A^{\theta_0}_\nu \rho_\la f\|_{L^{q_c}(M)}^{q_c} \, \Bigr)^{1/q_c}
+\la^{\frac1{q_c}-}\|f\|_2,
\end{equation}
since $\|\rho_\la\|_{2\to2}=O(1)$.  Consequently, to prove the bounds in it suffices to show that the first term in
the right is dominated by the right side of \eqref{2.23} when all of the sectional curvatures of $M$ are nonpositive
and by the right side of \eqref{2.24} if they all are negative.  We recall that we are assuming as in \eqref{2.15} that
$f$ is $L^2$-normalized.

We first note that, since $\mu(q_c)=1/q_c$,
 by using \eqref{2.10}, \eqref{commute} and \eqref{2.33}, we obtain
\begin{multline}\label{2}
\sum_\nu \|\tilde \sigma_\la A^{\theta_0}_\nu \rho_\la f
\|^{q_c}_{L^{q_c}(M)}
=\sum_\nu \|\tilde \sigma_\la A^{\theta_0}_\nu \rho_\la f\|_{q_c}^2
\cdot \|\tilde \sigma_\la A^{\theta_0}_\nu \rho_\la f\|_{q_c}^{q_c-2}
\\
\le
\sum_\nu \|\tilde \sigma_\la A^{\theta_0}_\nu \rho_\la f\|_{q_c}^2
\cdot \| BA^{\theta_0}_\nu\sigma_\la  \rho_\la f\|_{q_c}^{q_c-2}
\\
\qquad \qquad \qquad\qquad \qquad+ \sum_\nu \|\tilde \sigma_\la A^{\theta_0}_\nu \rho_\la f\|_{q_c}^2
\cdot \| (BA^{\theta_0}_\nu\sigma_\la-B\sigma_\la A^{\theta_0}_\nu) \rho_\la f\|_{q_c}^{q_c-2}
\\
\lesssim \sum_\nu \|\tilde \sigma_\la A^{\theta_0}_\nu \rho_\la f\|_{q_c}^2
\cdot \| B A^{\theta_0}_\nu\sigma_\la  \rho_\la f\|_{q_c}^{q_c-2}
+\sum_\nu \la^{\frac2{q_c}}\|A^{\theta}_\nu\rho_\la f\|_2^2
\cdot \la^{(\frac1{q_c}-\frac14)(q_c-2)}\|\rho_\la f\|_2^{q_c-2}
\\
\lesssim 
\sum_\nu \|\tilde \sigma_\la A^{\theta_0}_\nu \rho_\la f\|_{q_c}^2
\cdot \| B A^{\theta_0}_\nu\sigma_\la  \rho_\la f\|_{q_c}^{q_c-2}
+\la^{1-\frac14(q_c-2)}
\\
\le C
\bigl(\sum_\nu \|\tilde \sigma_\la A^{\theta_0}_\nu \rho_\la f
\|_{q_c}^{q_c}\bigr)^{\frac2{q_c}}
\bigl(\sum_\nu \|B A^{\theta_0}_\nu \sigma_\la \rho_\la f\|_{q_c}
^{q_c}\bigr)^{\frac{q_c-2}{q_c}} +C\la^{1-\frac14(q_c-2)},
\end{multline}
using also H\"older's inequality in the last line.

By Young's inequality, we can bound the second to last
term as follows
\begin{multline*}C\bigl(\sum_\nu \|\tilde \sigma_\la A^{\theta_0}_\nu \rho_\la f
\|_{q_c}^{q_c}\bigr)^{\frac2{q_c}}
\bigl(\sum_\nu \|B A^{\theta_0}_\nu \sigma_\la \rho_\la f\|_{q_c}
^{q_c}\bigr)^{\frac{q_c-2}{q_c}}
\\
=C \delta \bigl(\sum_\nu \|\tilde \sigma_\la A^{\theta_0}_\nu \rho_\la f
\|_{q_c}^{q_c}\bigr)^{\frac2{q_c}}
\cdot \delta^{-1}
\bigl(\sum_\nu \|B A^{\theta_0}_\nu \sigma_\la \rho_\la f\|_{q_c}
^{q_c}\bigr)^{\frac{q_c-2}{q_c}}
\\
\le C \Bigl[ \tfrac2{q_c}\delta^{\frac{q_c}2} \sum_\nu \|\tilde \sigma_\la A^{\theta_0}_\nu \rho_\la f
\|_{q_c}^{q_c}
+ 
\tfrac{q_c-2}{q_c}\delta^{-\frac{q_c}{q_c-2}}
\sum_\nu \| BA^{\theta_0}_\nu \sigma_\la \rho_\la f\|_{q_c}
^{q_c} \Bigr].
\end{multline*}
If $\delta>0$ is small enough so that $C\tfrac2{q_c}\delta^{\frac{q_c}2}$
 is smaller than $1/2$, we can absorb the contribution
of the first term in the right side of the preceding inequality into the left side
of \eqref{2} and conclude that
\begin{align}\label{3}
\sum_\nu \|\tilde \sigma_\la A^{\theta_0}_\nu \rho_\la f
\|^{q_c}_{L^{q_c}(M)}
&\lesssim \sum_\nu \|B A^{\theta_0}_\nu \sigma_\la \rho_\la f\|_{q_c}
^{q_c}+\la^{1-\frac14(q_c-2)}.
\\
&\lesssim  \sum_\nu \| A^{\theta_0}_\nu \sigma_\la \rho_\la f\|_{q_c}
^{q_c}+\la^{1-\frac14(q_c-2)}, \notag
\end{align}
using \eqref{2.9} in the last line.  Next, if we use \eqref{2.12} along with the  $L^{q_c}$ almost orthogonality bounds in \eqref{2.33} we can control
the nontrivial term on the right as follows
\begin{equation}\label{2.58}
 \sum_\nu \| A^{\theta_0}_\nu \sigma_\la \rho_\la f\|_{q_c}^{q_c} 
 \lesssim  \sum_\nu \| A^{\theta_0}_\nu \rho_\la f\|_{q_c}^{q_c} + T^{-q_c}\la.
\end{equation} 

If we combine \eqref{2.55}--\eqref{2.58}, we conclude that
\begin{equation}\label{2.59}
\|\tilde \rho_\la f\|_{L^{q_c}(A_-)}\lesssim \|A^{\theta_0}_\nu \rho_\la f\|_{\ell_\nu^{q_c}L^{q_c}(M)}+O(\la^{\frac1{q_c}-}+T^{-1}\la^{\frac1{q_c}}).
\end{equation}
If the sectional curvatures of $M$ are all nonpositive we conclude from Lemma~\ref{globalker} that \eqref{ker} is valid for
$\alpha=\tfrac{n-1}2$, and so, by \eqref{est} in Proposition~\ref{kerprop} we obtain in this case
\begin{multline}\label{2.60}
\|\tilde \rho_\la f\|_{L^{q_c}(A_-)}
\lesssim \|A^{\theta_0}_\nu \rho_\la f\|_{\ell_\nu^{q_c}L^{q_c}(M)} +O(\la^{\frac1{q_c}-}+T^{-1}\la^{\frac1{q_c}})
\\
\lesssim T^{-\frac1{n+1}\cdot
\frac{n-1}2} \la^{\frac1{q_c}}
=(T^{-1}\la)^{\frac1{q_c}} \approx (\la(\log\la)^{-1})^{\mu(q_c)},
\end{multline}
which along with the earlier bound \eqref{2.18} yields \eqref{1.7} for $q=q_c=\tfrac{2(n+1)}{n-1}$.

If all of the sectional curvatures of $M$ are negative, then Lemma~\ref{globalker} says that \eqref{ker} is valid for any $\alpha\in {\mathbb N}$ and so we can use the more favorable case of \eqref{est} involving $T^{-1/2}$.  As a result, if we repeat the arguments leading to \eqref{2.60}
we conclude that if all the sectional curvatures of $M$ are negative we have
\begin{equation}\label{2.61}
\|\tilde \rho_\la f\|_{L^{q_c}(A_-)}\lesssim \la^{\mu(q_c)} (\log\la)^{-1/2},
\end{equation}
which along with \eqref{2.18} yields \eqref{1.8} for the critical index $q=q_c$.
\end{proof}

Let us also now handle the subcritical bounds in Theorem~\ref{thm1.1}.

\begin{proof}[Proof of subcritical estimates in Theorem~\ref{thm1.1}]
Inequality \eqref{1.7} for $q\in (2,q_c)$ just follows via interpolation from the $q=q_c$ estimate we just obtained and the fact that the 
projection operators are bounded on $L^2(M)$ with norm one.

Thus, we only have to prove the subcritical bounds in \eqref{1.8} for $q\in (2,q_c)$, assuming as there that all of the sectional curvatures
of $M$ are negative.  By interpolation ,
we see that it suffices to prove the estimates for $q\in (2,\tfrac{2(n+2)}n]$.  We make this reduction in order to use \eqref{2.45}, which as
we noted before is an estimate over all of $M$ (unlike \eqref{2.44}).  As before, in order to prove \eqref{1.8} for $q$ in the above range,
it suffices to show that when $f$ is $L^2$-normalized as in \eqref{2.15} we have
\begin{equation}\label{2.62}
\| \tilde \rho_\la f\|_{L^q(M)}\lesssim T^{-1/2}\la^{\mu(q)}, \quad q\in (2,\tfrac{2(n+2)}n],
\end{equation}
for $T$ as in \eqref{2.1} with $c_0>0$ sufficiently small depending on our manifold $M$ of negative curvature.

If we use \eqref{2.45} in place of \eqref{2.44}, and repeat the proof of \eqref{2.59} we obtain
\begin{equation}\label{2.63}
\| \tilde \rho_\la f\|_{L^q(M)}\lesssim \|A^{\theta_0}_\nu \rho_\la f\|_{\ell^q_\nu L^q(M)} +O(\la^{\mu(q)-}+T^{-1}\la^{\mu(q)}),
\quad q\in (2,\tfrac{2(n+2)}n].
\end{equation}
Since, as we just exploited, \eqref{ker} is valid for all $\alpha\in {\mathbb N}$ under our curvature assumption, by \eqref{est2}
the first term in the right hand side is $O(T^{-1/2}\la^{\mu(q)})$, which yields \eqref{2.62} and completes the proof.
\end{proof}

Let us now prove Proposition~\ref{kerprop}.

\begin{proof}[Proof of Proposition~\ref{kerprop}]
If $Uf(x,\nu)=A_\nu \rho_\la f(x)$, then \eqref{est}
is equivalent to
\begin{equation}\label{8}
\|UU^*\|_{\ell^{q_c'}_{\nu'}L^{q_c'}\to \ell^{q_c}_\nu L^{q_c}}
\lesssim 
\begin{cases}
T^{-\frac{2\alpha}{n+1}} \la^{\frac2{q_c}}, \, \, \text{if } \, \alpha<\tfrac{n+1}2,
\\
T^{-1} \la^{\frac2{q_c}}, \, \, \text{if } \, \alpha>\tfrac{n+1}2,
\end{cases}
\end{equation}
with
\begin{align}\label{9}
\bigl(UU^*F\bigr)(x,\nu)&=\sum_{\nu'}
\Bigl(\bigl(A^{\theta_0}_\nu \circ \rho^2(T(\la-P))\circ (A^{\theta_0}_{\nu'})^*\bigr)F(\, \cdot \, ,\nu')\Bigr)(x)
\\
&= \sum_{\nu'}\Bigl(\bigl(A^{\theta_0}_\nu \circ L_{\la,T}\circ (A^{\theta_0}_{\nu'})^* \bigr)F(\, \cdot \, ,\nu')\Bigr)(x) \notag
\\ &\qquad 
+\sum_{2\le N=2^j\le T} 
\Bigl[\sum_{\nu'}\Bigl(\bigl(A^{\theta_0}_\nu \circ G_{\la,T,N}\circ (A^{\theta_0}_{\nu'})^* \bigr)F(\, \cdot \, ,\nu')\Bigr)(x)
\Bigr].
\notag
\end{align}

If we use \eqref{2.33} and \eqref{2.34} along with \eqref{loc} 
 we obtain
\begin{align}\label{localpart}
\Bigl\| \sum_{\nu'}\Bigl(\bigl(A^{\theta_0}_\nu \circ &L_{\la,T}\circ (A^{\theta_0}_{\nu'})^* \bigr)F(\, \cdot \, ,\nu')\Bigr)(\, \cdot \,)
\Bigr\|_{\ell^{q_c}_\nu L^{q_c}}
\\
&\le 
\Bigl\| \sum_{\nu'} \bigl(L_{\la,T}\circ (A^{\theta_0}_{\nu'})^*\bigr)F(\, \cdot \, ,\nu')\Bigr)(\, \cdot \,)
\Bigr\|_{L^{q_c}} \notag
\\
&\le T^{-1}\la^{\frac2{q_c}}
\Bigl\| \sum_{\nu'} (A^{\theta_0}_{\nu'})^* F(\, \cdot \, ,
\nu')\Bigr\|_{L^{q_c'}} \notag
\\
& \le T^{-1}\la^{\frac2{q_c}} \|F\|_{\ell^{q_c'}_{\nu'}L^{q_c'}}
\notag
\end{align}
which is better than the bounds in \eqref{8} if $\alpha<\tfrac{n+1}2$ and agrees with them
for $\alpha>\tfrac{n+1}2$.

To finish the proof of \eqref{8}, we also need
to estimate the $N$-summands in \eqref{9},
\begin{equation}\label{W}
W_N F = \sum_{\nu'}\Bigl(\bigl(A^{\theta_0}_\nu \circ G_{\la,T,N}\circ (A^{\theta_0}_{\nu'})^* \bigr)F(\, \cdot \, ,\nu')\Bigr)(x),
\quad N=2^j, \, j\in {\mathbb N}.
\end{equation}

By \eqref{global} we clearly have
$$\|G_{\la,T,N}\|_{L^2(M)\to L^2(M)}=O(T^{-1}N).$$
So, if we use \eqref{2.33} and \eqref{2.34}  for $q=2$  the preceding argument yield
for $2\le 2^j=N$
\begin{equation}\label{w2}
\|W_N\|_{\ell^2_{\nu'}L^2\to \ell^2_\nu L^2}=O(T^{-1}N).
\end{equation}
We
also obtain from \eqref{ker}
\begin{equation}\label{winfty}
\|W_N\|_{\ell^1_{\nu'}L^1\to \ell^\infty_\nu L^\infty}=O(T^{-1}\la^{\frac{n-1}2}N^{1-\alpha}).
\end{equation}

If we interpolate between these two estimates
we obtain
\begin{equation}\label{interpolate}
\|W_N\|_{\ell^{q_c'}_{\nu'} L^{q_c'}\to \ell^{q_c}_\nu L^{q_c}}
=O(T^{-1}\la^{\frac2{q_c}}N^{1-\frac{2\alpha}{n+1}}).
\end{equation}
Whence,
\begin{equation}\label{globalpart}
\sum_{2\le 2^j=N\le T}
\|W_N\|_{\ell^{q_c'}_{\nu'} L^{q_c'}\to \ell^{q_c}_\nu L^{q_c}}
\lesssim \begin{cases}
T^{-\frac{2\alpha}{n+1}} \la^{\frac2{q_c}}, \, \, \text{if } \, \alpha<\tfrac{n+1}2,
\\
T^{-1} \la^{\frac2{q_c}}, \, \, \text{if } \, \alpha>\tfrac{n+1}2.
\end{cases}
\end{equation}

If we combine \eqref{9}, \eqref{localpart} and
\eqref{globalpart}, we obtain \eqref{8}.
The same argument yields \eqref{est2} which
finishes the proof of Proposition~\ref{kerprop}.
\end{proof}

Let us now recall the arguments that yield the bounds in Lemma~\ref{globalker}.  The arguments that we shall sketch are almost identical to ones in \cite[\S 3]{BSTop}.

First, in view of \eqref{2.36}, in order to prove
\eqref{kernonpos}, it suffices to show that
\begin{multline*}
(A^{\theta_0}_\nu G_{\la,T,N})(x,y)
=(2\pi T)^{-1} \int \beta(|t|/N)
\Hat \Psi(t/T) e^{it\la}
(A^{\theta_0}_\nu \circ e^{-itP})(x,y) \, dt
\\
=(\pi T)^{-1} \int \beta(|t|/N)
\Hat \Psi(t/T) e^{it\la}
(A^{\theta_0}_\nu \circ \cos t\sqrt{-\Delta_g})(x,y) \, dt
\\ + (2\pi T)^{-1} A_\nu \circ\int
\beta(|t|/N)
\Hat \Psi(t/T) e^{it\la} e^{itP}(x,y) \, dt
\\
=O(T^{-1}\la^{\frac{n-1}2}N^{1-\frac{n-1}2}), \, \,
N\ge2.
\end{multline*}
Using \eqref{2.36} again shows via a simple
integration by parts argument that the second to last term is $O(\la^{-m})$ for all $m\in {\mathbb N}$ since
$\la\gg1$ and $P$ is a nonnegative operator.  Thus, in
order to establish \eqref{kernonpos}, it suffices to 
show that
\begin{multline}\label{2.72}
T^{-1} \int \beta(|t|/N)
\Hat \Psi(t/T) e^{it\la}
(A^{\theta_0}_\nu \circ \cos t\sqrt{-\Delta_g})(x,y) \, dt 
\\
=O(T^{-1}\la^{\frac{n-1}2}N^{1-\frac{n-1}2}), \, \,
N\ge2.\end{multline}

To do this, as in B\'erard~\cite{Berard} and many other
works, we lift the calculation up to the universal
cover  of $(M,g)$ using the formula
(see e.g., \cite[(3.6.4)]{SoggeHangzhou})
$$(\cos t\sqrt{-\Delta_{g}})(x,y)
=\sum_{\alpha\in \Gamma}
(\cos t\sqrt{-\Delta_{\tilde g}}(\tilde x,\alpha(\tilde y)).
$$
Here $(\Rn,\tilde g)$ is the universal cover of 
$(M,g)$, with $\tilde g$ being the Riemannian metric
on $\Rn$ obtained by pulling back the metric $g$
via the covering map, also $\Gamma: \Rn\to \Rn$ are the 
deck transformations and we have chosen a 
Dirichlet domain $D\subset \Rn$, which we identify
with $M\simeq \Rn/\Gamma$ and $\tilde x\in D$ is the 
lift of $x\in M$.

Thus, we can rewrite the left side of \eqref{2.72} as
\begin{equation}\label{2.73}
 T^{-1} \sum_{\alpha\in \Gamma} \int \beta(|t|/N)
\Hat \Psi(t/T) e^{it\la}
(A^{\theta_0}_\nu \circ \cos t\sqrt{-\Delta_{\tilde g}})(\tilde x,\alpha(\tilde y)) \, dt.
\end{equation}
By finite propagation speed of solutions to the wave equation the summand vanishes if
$d_{\tilde g}(\tilde x,\alpha(\tilde y))>T$; however,
in general there can be $\approx \exp(C_MT)$ nonzero
terms due to our curvature assumptions.  As exploited though in \cite{BSTop} one can
use the Hadamard parametrix to see that
 the
presence of
microlocal operators $A_\nu^{\theta_0}$ means that
there are only $O(N)$ nontrivial terms in
\eqref{2.73}.

To this end, we recall (see \cite{Berard} and
\cite{SoggeHangzhou}) that the Hadamard parametrix
tells us that we can write
\begin{equation}\label{2.74}
(\cos t\sqrt{-\Delta_{\tilde g}})(\tilde x,\tilde z)
=(2\pi)^{-n} w(\tilde x,\tilde z)
\int_{\Rn}e^{id_{\tilde g}(\tilde x,\tilde z)\xi_1} 
\cos t|\xi|\, d\xi +R(t;\tilde x,\tilde z),
\end{equation}
where the remainder term, $R$, will not contribute
significantly to the bounds and where the coefficient of the leading term
satisfies
\begin{multline}\label{2.75}
w(\tilde x,\tilde z)=O(1) \, \, 
\text{if the principal curvatures of } \, M \, \,
\text{are nonpositive},
\\
\text{and } \, w(\tilde x,\tilde z)=O((1+d_{\tilde g}(\tilde x,\tilde z))^{-m}) \, \forall \, m \, \,
\text{if the principal curvatures are all negative}.
\end{multline}
Standard arguments as in \cite{Berard}, \cite{BSTop} and \cite{SoggeHangzhou} show that if one replaces
$(\cos t\sqrt{-\Delta_{\tilde g}}(\tilde x,\alpha(\tilde y))$ with $R(t;\tilde x,\alpha(\tilde y))$ the resulting
expression is $O(\la^{\frac{n-1}2-})$ which is much better
than that required for \eqref{2.72} in view of the 
fact that, by \eqref{2.1}, we are assuming that
$T=c_0\log\la$ with $c_0>0$ allowed to be small.
Also, by a simple stationary phase argument,
if we replace $(\cos t\sqrt{-\Delta_{\tilde g}})(\tilde x,\alpha(\tilde y))$ with the first term in the right
side of \eqref{2.74} with $\tilde z=\alpha(\tilde y)$,
then the resulting expression is always
$O(T^{-1}\la^{\frac{n-1}2}(d_{\tilde g}(\tilde x,
\alpha(\tilde y))^{-\frac{n-1}2})$, and, since the
amplitude in \eqref{2.72} is supported in the region
where $|t|\approx N$ due to \eqref{2.7}, by another
simple integration by parts each of these terms is
$O(\la^{-m})$ for all $m\in {\mathbb N}$ if
$d_{\tilde g}(\tilde x,\alpha(\tilde y))\notin
[C_0^{-1},C_0]$ for some fixed $C_0$ since we are assuming that $N\ge2$.  Moreover, if we argue as in the proof of (3.8) in \cite{BSTop}, we see that if $\alpha(D)$ is not within a fixed distance of the lift of the extension of the geodesic in $M$ associated with the
microlocal operator $A_\nu^{\theta_0}$ then the
resulting kernel is also $O(\la^{-m})$ $\forall \, m$.
So, arguing almost identically as in the proof of (3.8)
in \cite{BSTop} shows that there are only $O(N)$ terms arising from the main term in the Hadamard parametrix,
each of which, as we just mentioned, is
$O(T^{-1}\la^{\frac{n-1}2}N^{\frac{n-1}2})$, while all
the others, as well as the contribution of the
remainder term $R$ in \eqref{2.74} collectively 
contribute to a $O(\la^{\frac{n-1}2-})$ error term.
This of course leads to the bounds in \eqref{kernonpos}.

In this argument, we merely used the fact that the
leading coefficient $w(\tilde x,\tilde z)$ of the
Hadamard parametrix is $O(1)$ if the principal curvatures
of $(M,g)$ are nonpositive.  As noted in 
\eqref{2.75}, though, it is $O((d_{\tilde g}(\tilde x,
\tilde z))^{-m})$ $\forall \, m$ if the principal
curvatures of $(M,g)$ are all negative.  Consequently,
if one repeats the above argument each nontrivial
term that arises must be $O(T^{-1}\la^{\frac{n-1}2}
N^{-m})$ \, $\forall \, m$, which yields the other
estimate, \eqref{kerneg}, in Lemma~\ref{globalker}
and completes the sketch of its proof. \qed

\newsection{Characterizing  compact space forms using log-quasimode estimates}\label{spaceforms}

In this section we shall prove Theorem~\ref{thm1.2}.  We shall only prove the results for $\delta(\la)=(\log\la)^{-1}$, i.e., \eqref{shape}, since the proof
of Theorem~\ref{thm1.1} shows that the estimates in \eqref{1.7} and \eqref{1.8} remain valid if $(\log\la)^{-1}$ is replaced by
$\delta(\la)$ as in the statement of Theorem~\ref{thm1.2}.
Using this and simple modifications of the negative results to follow, one obtains the second assertion, \eqref{shape2},
in Theorem~\ref{thm1.2}.

Proving \eqref{shape} is equivalent to proving the following three assertions for compact connected manifolds of constant sectional curvature $K$:
\begin{align}\label{1.10}
\limsup_{\la\to \infty}\, &\la^{-\mu(q)} \bigl\| \, \chi_{[\la,\la+(\log\la)^{-1}]} \, \bigr\|_{2\to q}\in
(0,\infty) \, \, \text{if and only if } \, \, K>0,
\\
\limsup_{\la\to \infty}&\bigl(\la(\log\la)^{-1}\bigr)^{-\mu(q)} \bigl\| \, \chi_{[\la,\la+(\log\la)^{-1}]} \, \bigr\|_{2\to q}\in
(0,\infty) \, \, \text{if and only if } \, \, K=0,
\label{1.11}
\end{align}
and
\begin{equation}\label{1.12}
\limsup_{\la\to \infty} \la^{-\mu(q)} (\log\la)^{1/2} \bigl\| \, \chi_{[\la,\la+(\log\la)^{-1}]} \, \bigr\|_{2\to q}\in
(0,\infty) \, \, \text{if and only if } \, \, K<0.
\end{equation}

Note that by \eqref{1.4}, \eqref{1.7} and \eqref{1.8} each of these three ``limsups'' is finite.
Thus, in order to prove \eqref{shape} it suffices to prove that each one is nonzero.

Let us first prove that this is the case for \eqref{1.10}.  Without loss of generality we may assume that $K=1$. 
It then is a classical theorem 
(see e.g. \cite[4.3 Proposition, Chapter 8]{doCarmoRiemannian}
or \cite{wolfconstant}) that our compact
manifold $(M,g)$, all of whose sectional curvatures equal one,
is isometric to $S^n/\Gamma$ where
$\Gamma$ is a subgroup of the group of isometries on the standard round
$S^n$.  Consequently, the eigenfunctions
on our compact manifold $M$ of constant
curvature $K=1$ are precisely the 
$\Gamma$-invariant eigenfunctions on $S^n$ (i.e.,  $\Gamma$-invariant spherical
harmonics), and the spectrum of our first order operator $\sqrt{-\Delta_g}$ on $(M,g)$ must be contained in that of the round sphere.

Recall that the distinct eigenvalues of
$\sqrt{-\Delta_{S^n}}$ are
$(k(k+n-1))^{1/2}$, $k=0,1,2,\dots$.  
Note that, for $k$ is larger than a fixed constant depending on $n$, the gap between consecutive distinct
eigenvalues of $\sqrt{-\Delta_{S^n}}$ and hence those of  the operator $\sqrt{-\Delta_M}$ on our $(M,g)$ of constant sectional curvature 1
is larger than one.  So every interval $[\la,\la+1]$ with $\la\gg1$ contains at most one of the distinct eigenvalues of 
$P=\sqrt{-\Delta_g}$.  This means
that if $0\le \delta(\la)\le1$ then $[\la,\la+\delta(\la)]\cap \text{Spectrum }\sqrt{-\Delta_g}$ is either empty
for $\la\gg1$ or is just a single point $\{\sqrt{k(k+n-1)}\}$ for some $k\in {\mathbb N}$.  Thus, when the 
sectional curvatures all equal one we have for $q\in (2,q_c]$
$$\limsup_{\la\to \infty} \la^{-\mu(q)}\|\chi_{[\la,\la+(\log\la)^{-1}]}\|_{2\to q}
=\limsup_{\la\to \infty} \la^{-\mu(q)}\|\chi_{[\la,\la+1]}\|_{2\to q}.$$
In \cite{SFIO2} it was shown that the last ``limsup'' is positive on {\em any} $(M,g)$ (meaning that the bounds
in \eqref{1.4} are sharp), and so we conclude that the ``limsup'' in \eqref{1.10} must be nonzero, as desired.

It is also easy to see that this is the case for \eqref{1.12} which involves the assumption that $(M,g)$ is of 
constant sectional curvature $K<0$.  Indeed, if we note that any interval $[\la,\la+1]$ can be covered
by $\log\la +1$ intervals of length $(\log\la)^{-1}$ for $\la\gg1$ we can use the Cauchy-Schwarz inequality to see that 
for $q\in (2,q_c]$ we have
$$\|\chi_{[\la,\la+1]}\|_{2\to q}\lesssim \sup_{\tau\in [\la,\la+1]} (\log\tau)^{1/2}\|\chi_{[\tau,\tau+(\log\tau)^{-1}]}\|_{2\to q}.$$
Consequently, if, for such $q$,
$$\limsup_{\la\to \infty} \, (\log\la)^{1/2}\la^{-\mu(q)} \|\chi_{[\la,\la+(\log\la)^{-1}]}\|_{2\to q}=0,$$
we would have
$\limsup_{\la\to \infty}\la^{-\mu(q)}\|\chi_{[\la,\la+1]}\|_{2\to q}=0$, which, as we just mentioned is impossible
on any compact manifold.  So, the ``limsup'' in \eqref{1.12} must also be nonzero.

The proof of Theorem~\ref{thm1.2} would therefore be complete if we could show that whenever
$(M,g)$ is a connected compact flat manifold the ``limsup'' in \eqref{1.11} also must be nonzero,
which is much more difficult than the two cases that we have just dealt with.  To deal with the
case of flat manifolds we need to construct appropriate ``Knapp examples'' as we shall do in the next subsection.

\noindent{\bf 3.1. Characterizing flat compact manifolds.}

By a classical theorem of Cartan and Hadamard, if $(M,g)$ is a compact flat manifold, it must be of the form
$\Rn/ \Gamma$.  A theorem from 1912  of
Bierbach \cite{bieberbach} (see e.g., Corollary 5.1 and 
Theorem 5.3 in Chapter 2 in \cite{flat}
or \cite{wolfconstant}) says that the deck
transformations,
$\Gamma$, must be a Bieberbach subgroup of the group rigid motions, $E(n)$, of $\Rn$.  This means that: 
i) $\Gamma$ must be a discrete subgroup of $E(n)$, ii) $\Gamma$ must be cocompact (i.e. $\Rn/\Gamma$ is compact), and iii)
$\Gamma$ must act freely on $\Rn$ (i.e., if $\alpha\in \Gamma$ and $\alpha(x)=x$ for some $x\in \Rn$, then $\alpha$ must be
the identity).  Subgroups of $E(n)$ satisfying these conditions are also called crystallographic subgroups.
Bieberbach also showed that for each $n$ there are only finitely many types (i.e. isomorphism classes) of Bieberbach
subgroups of $E(n)$, which solved Hilbert's 18th problem.  In 2-dimensions, there are only two\footnote{When $n=2$,
up to isomorphisms, the two examples are the subgroups $\Gamma\subset E(2)$ whose generators are as follows
$\{I+e_1, I+e_2\}$
and
$\{\begin{psmallmatrix}1 & 0\\0 & -1\end{psmallmatrix}+ e_1, I+e_2\}$.
Both have $Q=[0,1]\times [0,1]$ as a fundamental domain.  The quotient of the first is the 2-torus and the second
the Klein bottle.}--the quotients are 2-tori or
Klein bottles, which are the two connected compact flat Riemannian manifolds in dimension two.  In 3-dimensions it has been
known since the 1930s
there are ten, but the classification
is incomplete in higher dimensions (see \cite{flat} and \cite{wolfconstant}).

We shall use these facts to construct our Knapp examples for our compact flat $(M,g)$.
We recall that the rigid motions of $\Rn$ are of the form
$\alpha (y)=my+j$,
where $m\in O_n$ is an orthogonal matrix and $j\in \Rn$ is a translation.
So, if $M\simeq \Rn/\Gamma$ and $\alpha\in \Gamma$, then 
$\alpha$ must be a particular element of $E(n)$
of this form, since $\Gamma \subset E(n)$ must be a Bieberbach subgroup associated with
 $(M,g)$.

To construct a Knapp example for our flat compact manifold $M$, we choose a periodic geodesic $\gamma_0\subset M$.  The Knapp example then
will simply be the standard Knapp example for $\Rn$ projected to $M$ via the covering map for $M= \Rn/\Gamma$.  Recall that Knapp examples
in $\Rn$ are quasimodes which are essentially supported in long thin tubes.  To obtain ones for $M$ we choose
the central axis of this tube so that it projects to $\gamma_0\subset M$ via the covering map.  This leads to many windings around a thin tube about
$\gamma_0$ in $M$, and, hence, potential concentration on a subset of $M$ having much smaller volume than that of the Knapp tube in $\Rn$ from which
it arises.  This explains why, on compact flat manifolds, we only can have the bounds in \eqref{1.7} despite the fact that, by the
Stein-Tomas theorem \cite{TomasWill}, the stronger analogues given by \eqref{1.8} with $q=q_c$ hold in $\Rn$.  Also, to obtain concentration near
$\gamma_0$, as we shall see, we need to choose the frequencies of our quasimodes based on the length of $\gamma_0$.


Let $\ell_0$ be the length of the chosen periodic geodesic $\gamma_0\in M$.
 We are not assuming that $\gamma_0$ is simply
closed.   It can cross itself.  We can, however, pick a point $x_0\in \gamma_0$, though, which is not a crossing point.  So if $\gamma_0(t)$, $t\in [0,\ell_0)$ parameterizes the geodesic by arc length with $\gamma_0(0)=x_0$,   then
$\gamma_0(t)\ne x_0$ for $t\in [0,\ell_0)$.  If $y\in \gamma_0$ is close to $x_0$ then $y$ must also not be a crossing point.
We may assume that $g_{jk}(x_0)=\delta^k_j$ in our local coordinate system about $x_0$.

Next, let $p=\exp_{x_0}: \Rn\to M$.  Then $p$ is a covering map.  If $D\subset \Rn$ is a Dirichlet domain containing the origin,
then we identify $D$ with $M$ by setting $p(\tilde x)=x$ if $\tilde x\in D$.  We then have $p(0)=x_0$ and the lift $\tilde \gamma$ of
$\gamma_0$ is a straight line through the origin which we may assume is the $x_1$-axis:  $(t,0,\dots,0)=\tilde \gamma$.  If
$f\in C^\infty(M)$ and $\tilde f(\tilde x)=f(x)$ and $\Delta=\partial^2/\partial x_1^2+\cdots + \partial^2/\partial x_n^2$ is the standard
Laplacian we have
\begin{equation}\label{4.1}
\Delta_g f(x)=\Delta \tilde f(\tilde x)
\end{equation}
for $\tilde x$ in the interior of $D$.

Also, $\gamma_0\subset D$ is a finite union of straight line segments some of which
may cross (as in Klein bottles), but not at the origin since $p(0)=x_0$
and $x_0$ is not a crossing point of $\gamma_0$.  Also, since
$\gamma_0(\ell_0)=\gamma_0(0)=x_0$ and its lift is $\tilde \gamma(t)=(t,0,\dots,0)$, there must
be a unique $\alpha_{\gamma_0}\in \Gamma$ so that 
$\alpha_{\gamma_0}(0)=\tilde\gamma(\ell_0)=(\ell_0,0,\dots,0)$.
Also, $\alpha_{\gamma_0}(\tilde \gamma)=\tilde \gamma$.
It follows that $\{\alpha_{\gamma_0}^j\}_{j\in \Z}$, the stabilizers of $\tilde \gamma$, is a cyclic subgroup of $\Gamma$
generated by $\alpha_{\gamma_0}$.  Here, by $\alpha^j_{\gamma_0}$, we mean for $j>0$, $\alpha^j_{\gamma_0} =
\alpha_{\gamma_0} \circ\cdots \circ \alpha_{\gamma_0}$ ($j$ times),  $\alpha^{-j}_{\gamma_0}$ for $j>0$ means the 
$j$-fold composition of the 
inverse of $\alpha_{\gamma_0}$, and $\alpha^j_{\gamma_0}=Identity$ if $j=0$.  Thus, $\alpha^j_{\gamma_0}(0)=(j\ell_0,0,\dots,0).$

Since $\alpha_{\gamma_0}\in \Gamma$, it follows that
$\alpha_{\gamma_0}(y)=m_0y+j_0$,
for some $m_0\in O_n$ and $j_0\in \Rn$.  Since $\alpha(0)=(\ell_0,0,\dots,0)$, we must have $j_0=(\ell_0,\dots,0)$.  Also since $\alpha_{\gamma_0}(\tilde \gamma)=\tilde \gamma$, $m_0$ must preserve the $x_1$-axis.  Since $m_0\in O_n$, it follows that
for some $\overline{m}\in O_{n-1}$,
\begin{equation*}
 m_0=   \left(
    \begin{array}{ccccc}
    \pm1  &0 \qquad \cdots & 0                                  
\\
       0     &   &        \\
                  \vdots &     \qquad \mbox{ \huge$\overline{m}$}  &        
\\
0 & &
    \end{array}
    \right).
\end{equation*}
We cannot have ``$-1$'' in the top left corner, since, in this case, we would have $\alpha^2_{\gamma_0}(0)=0\ne (2\ell_0,0,\dots,0)$.
Consequently, we must have
\begin{equation}\label{4.3}
\alpha_{\ell_0}(y)=m_0 y+(\ell_0,\dots,0).
\end{equation}
for some
\begin{equation}\label{4.2}
m_0=    \left(
    \begin{array}{ccccc}
    1  &0 \qquad \cdots & 0                                  
\\
       0     &   &        \\
                  \vdots &     \qquad \mbox{ \huge$\overline{m}$}  &        
\\
0 & &
    \end{array}
    \right),
\quad \text{with } \,   \overline{m}\in O_{n-1}.
\end{equation}

Let us use these facts to build our Knapp example for $M\simeq D$ about our periodic geodesic $\gamma_0$.  The argument
is somewhat like that in Brook~ \cite{BrooksQM} or Sogge~\cite[\S 5.1]{SFIO2}.  It also uses ideas from
Sogge and Zelditch \cite{SoggeZelditchL4}.
Our construction of quasimodes concentrating near $\gamma_0$ is a bit easier than that in \cite{BrooksQM}
given the form \eqref{4.3} of the generator of the stabilizer group of our periodic geodesic $\gamma_0$.  Not surprisingly,
we also are able to obtain much tighter concentration of our log-quasimodes since we are working in the flat case as
opposed to the much more difficult case where $K<0$ as in \cite{BrooksQM}.  

We shall use the following elementary result of Sogge and Zelditch~\cite[Proposition 1.3]{SZqm}, which is valid on any compact
manifold $(M,g)$.

\begin{lemma}\label{lemma4.1}
Suppose that for $q\in (2,q_c]$ and $\delta\in (0,1]$
$$\|\chi_{[\tau,\tau+\delta]}\|_{L^2(M)\to L^{q}(M)}\le C(\la,\delta), \quad
\text{if } \, \, \tau \in [\la/2,2\la].$$
Then for some uniform constant $C_0=C_0(M)$ we have for $\la\gg1$
$$\|f\|_{q}\le C_0 \, C(\la,\delta) \, \bigl[ \, 
\|f\|_2+(\la\delta)^{-1}\|(\Delta_g+\la^2)f\|_2\, \bigr].
$$
\end{lemma}

Thus, in order to prove that a flat manifold satisfies
$$\limsup_{\la\to \infty} \la^{-\mu(q)}(\log\la)^{\mu(q)}\|\chi_{[\la,\la+(\log\la)^{-1}]}\|_{2\to q}
\in (0,\infty), \quad \text{if } \, \, q\in (2,q_c],$$
which is the most difficult step in the proof of Theorem~\ref{thm1.2}, in view of our positive results \eqref{1.7},
it suffices to construct a sequence $\la_k\to \infty$ and associated ``log-quasimodes''
$\psi_{\la_k}$ so that with $\delta=\delta(\la)=(\log\la)^{-1}$ above, we have
\begin{equation}\label{4.4}
\| \psi_{\la_k}\|_{L^2(D)}+
\la_k^{-1} \log\la_k \, \|(\Delta+\la^2)\psi_{\la_k}\|_{L^2(D)} \lesssim 1,
\end{equation}
and, for uniform $c>0$,
\begin{equation}\label{4.5}
\|\psi_{\la_k}\|_{L^{q}(D)} \ge c\la^{\mu(q)}_k (\log\la_k)^{-\mu(q)}.
\end{equation}

To do this by constructing a ``Knapp" example for our flat compact manifold $M$, first fix $\eta\in {\mathcal S}(\R)$
satisfying
\begin{equation}\label{4.6}
\hat\eta\ge 0, \, \, \hat\eta(0)=1 \quad 
\text{and } \, \text{supp } \Hat \eta\subset[-c_0,c_0],
\end{equation}
where $c_0 \in (0,1)$ will be specified later.  Fix also a function 
\begin{equation}\label{4.7}
0\le a\in C^\infty_0((-1,1)) \, \, \quad
\text{with } \, a(s)=1, \, \, |s|\le 1/2,
\end{equation}
and set
\begin{multline}\label{4.8}
a_{\la,\delta}(\xi)=a\bigl(\la^{1/2}\delta^{-1/2} \, |e_1-\xi/|\xi|  \,| \bigr) \beta(|\xi|/\la), \, \,
e_1=(1,0,\dots,0), 
\\
\text{and } 
\beta\in C^\infty((1/4,4)), \, \, 
\text{satisfies } \, \, \beta(s)=1, \, \, s\in [1/2,2].
\end{multline}
Thus, $a_{\la,\delta}$ is supported in a dyadic region of a cone of aperture $\sim \la^{-1/2}\delta^{1/2}$
about the positive part of the first coordinate axis.  This function satisfies the related bounds
\begin{equation}\label{4.9}
\partial_{\xi_1}^j \partial^\sigma_{\xi'}a_{\la,\delta}(\xi)
=
O\bigl(\la^{-j}(\la^{-1/2}\delta^{-1/2})^{|\sigma|}\bigr), 
\quad \text{if } \, \, \xi'=(\xi_2,\dots,\xi_n).
\end{equation}

We now define our ``log-quasimodes'' $\psi_\la(y)$, $y\in D$, as follows
\begin{equation}\label{4.10}
\psi_\la(y)=\sum_{\alpha\in \Gamma} \la^{-\frac{n-1}4}\delta^{-\frac{n-1}4}
\int_{\Rn} e^{i\alpha(y)\cdot \xi}
a_{\la,\delta}(\xi) \eta(T(\la-|\xi|)) \, d\xi, 
\quad T=\delta^{-1}=\log\la.
\end{equation}
This is analogous to the traditional Knapp example for Euclidean space that showed that the
Stein-Tomas restriction theorem \cite{TomasWill} was sharp.  The amplitude in \eqref{4.10}
is essentially supported in a $\delta$ by $\la^{1/2}\delta^{1/2}$ plate through $(\la,0,\dots,0)$,
where $\delta$ is the thickness and $\la^{1/2}\delta^{1/2}$ is the ``vertical'' cross section of the plate.

In order to achieve the lower bounds in \eqref{4.5}, we shall need to assume that the
frequencies of the quasimode are of the form
\begin{equation}\label{4.11}
\la=\la_k=2\pi k/\ell_0, \quad \text{some } \, \, 1\ll k\in {\mathbb N},
\end{equation}
with, as above, $\ell_0$ denoting the length of our periodic geodesic on which the Knapp modes
in \eqref{4.10} will concentrate.  As we shall see, this choice of frequencies ensures that there is minimal
cancellation near the non-crossing point $x_0$ as the function in \eqref{4.10} wraps itself around and around $\gamma_0$.

Let us first prove that we have \eqref{4.4}.  We shall use a simple argument that is based on ideas 
from \cite{BSTop}.  To do so we shall need the following simple lemma about our Euclidean
Knapp functions.

\begin{lemma}\label{lemma4.2}  Let for $T=\delta^{-1}=\log\la$ as above
\begin{equation}\label{4.12}
K_\la(z)=\int e^{iz\cdot \xi}a_{\la,\delta}(\xi) \, \eta(T(\la-|\xi|)) \, d\xi.
\end{equation}
Then if $c_0>0$  in \eqref{4.6} is fixed small enough and $c>0$ is fixed we have
for $\la\gg 1$ the uniform bounds
\begin{equation}\label{4.13}
|K_\la(z)|+|\nabla K_\la(z)|\le C_N \la^{-N} \, \, \forall \, N,
\quad \text{if } \, \, |z'|>c.
\end{equation}
Also, assuming that $T\ge1$, 
\begin{equation}\label{4.14}
|K_\la(z)|\le C_N\bigl( \, |z|+\la\, \bigr)^{-N} \, \, \forall \, N, \, \,
\, \text{if } \, |z|\ge 2T.
\end{equation}
\end{lemma}

\begin{proof}
Note that
$$K_\la(z)=(2\pi)^{-1} \iint e^{iz\cdot \xi -it|\xi|} T^{-1}\Hat \eta(t/T)
a\bigl(\la^{1/2}\delta^{-1/2}|e_1-\xi/|\xi| \, |\bigr) \, \beta(|\xi|/\la) \, e^{i\la t} \, d\xi dt.$$
One gets \eqref{4.14} by integrating by parts in $\xi$ if $c_0$ in \eqref{4.6}
is small enough.  One obtains \eqref{4.13} by integrating by
parts in $\xi'$ and using \eqref{4.9}.
\end{proof}

\noindent{\bf Proof of \eqref{4.4}:}  Let us start by bounding the $L^2(D)$ norm of 
$\psi_\la(y)$.  Note that if $\alpha, \tilde\alpha\in \Gamma$ and $\alpha\ne \tilde \alpha$, then
$\alpha(D)$ and $\tilde \alpha(D)$ are disjoint since $\Gamma$ acts freely
on $\Rn$ and $D$ is a fundamental domain.  Also, each of these sets contains a ball
of radius $r_0>0$ centered at the pre-image of $0$ under the covering map, and is contained
in a ball of radius $r^{-1}_0$ with this center for some fixed $r_0>0$.    Thus, by \eqref{4.14},
\begin{multline*}\psi_\la(y)=
\\
(\la\delta)^{-\frac{n-1}4}
\sum_{\{\alpha\in \Gamma: \, \text{dist}(\tilde \gamma, \alpha(D))\le 10, \, 
\text{dist}(0,\alpha(D))\le 10T\}}
\int e^{i\alpha(y)\cdot \xi}a_{\la,\delta}(\xi)\eta(T(\la-|\xi|)) \, d\xi +O(\la^{-N}),
\end{multline*}
and the number of terms in the sum is $O(T)$.  Consequently,
by the Cauchy-Schwarz inequality,
\begin{equation*}
|\psi_\la(y)|
\lesssim \la^{-\frac{n-1}4} \delta^{-\frac{n-1}4}T^{\frac12}
\Bigl(\, \sum_{\alpha\in \Gamma}\bigl| \, 
\int e^{i\alpha(y)\cdot \xi}a_{\la,\delta}(\xi) \, \eta(T(\la-|\xi|)) \, d\xi \, \bigr|^2 \,
\Bigr)^{1/2}+O(\la^{-N}).
\end{equation*}
Thus, by Plancherel's theorem, modulo $O(\la^{-N})$
\begin{align*}
\int_D |\psi_\la(y)|^2 \, dy
&\lesssim T \la^{-\frac{n-1}2}\delta^{-\frac{n-1}2} \sum_{\alpha\in \Gamma}
\int_D \bigl| \int_{\Rn}e^{i\alpha(y)\cdot\xi}a_{\la,\delta}(\xi) \, \eta(T(\la-|\xi|)) \, d\xi\bigr|^2 \, dy
\\
&=T \la^{-\frac{n-1}2}\delta^{-\frac{n-1}2} \sum_{\alpha\in \Gamma}
\int_{\alpha(D)} \bigl| \int_{\Rn}e^{i y\cdot\xi}a_{\la,\delta}(\xi) \, \eta(T(\la-|\xi|)) \, d\xi\bigr|^2 \, dy
\\
&=T \la^{-\frac{n-1}2}\delta^{-\frac{n-1}2} \int_{\Rn} 
\bigl| \int_{\Rn}e^{i y\cdot\xi}a_{\la,\delta}(\xi) \, \eta(T(\la-|\xi|)) \, d\xi\bigr|^2 \, dy
\\
&=T \la^{-\frac{n-1}2}\delta^{-\frac{n-1}2} \int_{\Rn}|a_{\la,\delta}(\xi) \, \eta(T(\la-|\xi|))|^2 \, d\xi
\\
&=O(1),
\end{align*}
using in the last step polar coordinates along with the fact that $\eta\in {\mathcal S}(\R)$ and
$a_{\la,\delta}$ is supported in a cone of aperture $\sim \la^{-1/2}\delta^{1/2}$.

This gives us the desired upper bound for $\|\psi_\la\|_{L^2(D)}$ in \eqref{4.4}.  We also need to see that
\begin{equation}\label{4.15}
(\la\delta)^{-1} \, \|(\Delta+\la^2)\psi_\la\|_{L^2(D)}=O(1), \quad \delta=(\log\la)^{-1}.
\end{equation}
To prove this, we shall use the fact that, as we mentioned before, by Bieberbach's theorem, each
$\alpha\in \Gamma$ must be a rigid motion, i.e.,
$\alpha(y)=m_\alpha y+j_\alpha$ with $m_\alpha\in O_n$ and $j_\alpha\in \Rn$.  Thus,
since $|m_\alpha \xi|=|\xi|$ and the transpose of $m_\alpha$ is its inverse
\begin{align*}
\Delta_y\int e^{i\alpha(y)\cdot \xi} a_{\la,\delta}(\xi)\eta(T(\la-|\xi|)) \, d\xi
&=\Delta_y \int e^{iy\cdot m^{-1}_\alpha \xi} \,  e^{ij_\alpha \cdot \xi} \, a_{\la,\delta}(\xi)\eta(T(\la-|\xi|)) \, d\xi
\\
&=\Delta_y \int e^{iy\cdot  \xi}  e^{ij_\alpha \cdot m_\alpha \xi} a_{\la,\delta}(m_\alpha \xi)\eta(T(\la-|\xi|)) \, d\xi
\\
&=
\int (-|\xi|^2) e^{iy\cdot  \xi}  e^{ij_\alpha \cdot m_\alpha \xi} a_{\la,\delta}(m_\alpha \xi)\eta(T(\la-|\xi|)) \, d\xi
\\
&= \int (-|\xi|^2) \cdot e^{i\alpha(y)\cdot \xi} a_{\la,\delta}(\xi)\eta(T(\la-|\xi|)) \, d\xi.
\end{align*}
Consequently,
\begin{equation*}
(\Delta+\la^2)\psi_\la(y)=
\la^{-\frac{n-1}4}\delta^{-\frac{n-1}4}
\sum_{\alpha\in \Gamma}
\int e^{i\alpha(y)\cdot\xi}
a_{\la,\delta}(\xi) \, \bigl(\la^2-|\xi|^2\bigr) \eta(T(\la-|\xi|)) \, d\xi.
\end{equation*}
The proof of Lemma~\ref{lemma4.2} shows that if we let
$$\tilde K_\la(z)=
\int e^{iz\cdot\xi}a_{\la,\delta}(\xi) \cdot (\la^2-|\xi|) \eta(T(\la-|\xi|)) \, d\xi,$$
then the analogs of \eqref{4.13} and \eqref{4.14} must be valid.  So, if we argue
as above, we find that, modulo $O(\la^{-N})$,
\begin{align*}
\int_D |(\Delta+\la^2)\psi_\la|^2 \, dy 
&\lesssim T\la^{-\frac{n-1}2}\delta^{-\frac{n-1}2}
\int |a_{\la,\delta}(\xi) \cdot (\la^2-|\xi|^2) \eta(T(\la-|\xi|)) |^2 \, d\xi
\\
&
\lesssim T^{-1}\la^{-\frac{n-1}2}\delta^{-\frac{n-1}2} \la^2 \int |a_{\la,\delta}(\xi)
\cdot T(\la-|\xi|)\eta(T(\la-|\xi|))|^2 \, d\xi
\\
&=O(\la^2T^{-2}),
\end{align*}
since $\tau \eta(\tau)\in {\mathcal S}$.

Thus, we have for $T=\delta^{-1}=\log\la$
$$\la^{-1}\log\la \, \|(\Delta+\la^2)\psi_\la\|_{L^2(D)} =O(1),
$$
giving us \eqref{4.15}, which is the remaining part of \eqref{4.4}. \qed

\noindent{\bf Proof of \eqref{4.5}:}  To complete the proof of our results for flat compact manifolds, we must prove
\eqref{4.5}.  Unlike the proof of \eqref{4.4}, to prove this lower bound, we shall need to assume \eqref{4.11} to ensure
that there is no cancellation in the nontrivial terms in the sum \eqref{4.10} defining $\psi_{\la_k}$.

To prove this lower bound, consider our unit speed geodesic $\tilde \gamma(t)$ in $D$, and the
associated $\la^{-1/2}\delta^{-1/2}$--tubes about the
segment $\tilde \gamma(t)$, $|t|\le \overline{c}$, where $\overline{c} \in (0,1)$ will be fixed small enough in the ensuing calculation.  In particular,
it will be small enough so that there are no crossing points along $\gamma_0(t)\in M$ for $|t|\le \overline{c}$.

As before, $\delta=(\log\la)^{-1}$, and so with $\la=\la_k$ as above, we let
\begin{equation}\label{4.16}
\tubek =\bigl\{y\in D: \, \text{dist}((\tilde \gamma(t),y)\le \la_k^{-1/2}(\log\la_k)^{1/2}, \, \, |t|\le \overline{c}\bigr\}
\end{equation}
denote a $\la_k^{-1/2}\delta^{-1/2}_k$-tube about this segment.  Note that
$|\tubek|\approx \la^{-\frac{n-1}2}_k(\log\la_k)^{\frac{n-1}2}$, and so, by H\"older's inequality, if $q\in (2,q_c]$
\begin{align*}
\| \psik\|_{L^2(\tubek)}&\lesssim \bigl(\la_k^{-\frac{n-1}2}(\log\la_k)^{\frac{n-1}2}\bigr)^{(\frac12-\frac1{q})} \,
\|\psik\|_{L^{q}(D)}
\\
&=\la^{-\mu(q)}_k \delta_k^{-\mu(q)}\|\psik\|_{L^{q}(D)},  \, \, \, \delta_k=(\log\la_k)^{-1},
\end{align*}
since $\mu(q)=\tfrac{n-1}2(\tfrac12-\tfrac1q)$ for $q\in (2,q_c]$.
As a result, we would have the lower bound \eqref{4.5} and be done if we could show that for fixed $c_0>0$ as in \eqref{4.6} we have
\begin{equation}\label{4.17}
\|\psik\|_{L^2(\tubek)}\ge c_1
\end{equation}
for some uniform $c_1>0$ when $\la_k\gg 1$.

To prove \eqref{4.17}, by calculus, it suffices to show that we have the uniform bounds
\begin{equation}\label{4.18}
|\psi_k(\tilde \gamma(t))|\ge c_1 \la_k^{\frac{n-1}4}\delta_k^{\frac{n-1}4},  \, \, \text{some }
c_1>0, \, \, \text{if } \, \, |t|\le \overline{c},
\end{equation}
if $\overline{c}>0$ as above is small enough, and also the upper bound
\begin{equation}\label{4.19}
|\nabla_{y'}\psik(y)|\le C
\la_k^{\frac{n-1}4}\delta_k^{\frac{n-1}4} \cdot (\la_k\delta_k)^{1/2}, \, \, \,
\delta_k=(\log\la_k)^{-1}, \, \, y\in \tubek.
\end{equation}
We can use \eqref{4.19} along with \eqref{4.17} to obtain \eqref{4.17} since
$\tilde \gamma(t)=(t,0,\dots,0)$.

As we shall see, \eqref{4.19} follows from the proof of \eqref{4.15} and does not require that $\la$
be as in \eqref{4.11}.  So, let us focus first on \eqref{4.18}, which is the more difficult to prove.

To prove \eqref{4.18} we shall need to use the properties of the stabilizer group $G_{\gamma_0}$ of 
our periodic geodesic of length $\ell_0$ that we described before.  It is a cyclic subgroup of 
$\Gamma$ generated by $\ao\in \Gamma$ as in \eqref{4.2} and \eqref{4.3}.  Since $\gamma_0$ loops
back smoothly through $0$ with no other crossings there, as we mentioned before, 
$\alpha(0)\in \tilde \gamma$, the lift of $\gamma_0$, if and only $\alpha\in\Go$.

From this we deduce that if $\alpha\notin\Go$ we must have that
$\text{dist}(\alpha(0),\tilde \gamma)\ge c$ for some uniform constant $c>0$.    This implies that
if $\overline{c}$ in \eqref{4.16} is fixed small enough then for large enough $\lk$ we have
for $c_1=c/2$
\begin{equation}\label{4.20}
\text{dist}(\alpha(\tubek), \tilde \gamma)\ge c_1, \quad \text{if } \, \alpha\notin\Go.
\end{equation}
Thus, by Lemma~\ref{lemma4.2}, if $\psi_{\lk,\alpha}$ denotes the $\alpha$-summand
in the definition \eqref{4.10} of $\psik$, i.e.,
\begin{equation}\label{4.21}
\psi_{\lk,\alpha}(y)=(\la_k\delta_k)^{-\frac{n-1}4}
\int_{\Rn} e^{i\alpha(y)\cdot\xi} \akd(\xi) \etak \, d\xi,
\, \,  T_k=\dk^{-1}=\log\lk,
\end{equation}
we must have
\begin{multline}\label{4.22}
|\psi_{\lk,\alpha}(y)|+|\nabla \psi_{\lk,\alpha}(y)|
\le C_N \, \bigl(\alpha(y)+\la_k\bigr)^{-N},
\\
\text{if } \, \, y\in \tubek \, \, \text{and } \, \, \alpha\notin\Go \, \, \,
\text{or } \, \, \, \text{dist}(\alpha(y),0)\ge 2T_k.
\end{multline}

Thus, if $\ao$ as in \eqref{4.3} is the generator of $\Go$, we have for $T_k$ and $\dk$ as above
\begin{multline}\label{4.23}
\psik(y)=
\sum_{\{j\in \Z: \, \text{dist}(\ao^j(D),0)\le 2T_k\}} \la_k^{-\frac{n-1}4}\dk^{-\frac{n-1}4}
\int e^{i\ao^j(y)\cdot \xi} \akd(\xi) \etak \, d\xi
\\
+O(\la_k^{-N}) \quad \text{if } \, \, y\in \tubek.
\end{multline}
If $y=\tilde \gamma(t)=(t,0,\dots,0)$, $|t|\le \overline{c}$, then by \eqref{4.2} and \eqref{4.3}
$$\ao^j(\tilde \gamma(t))\cdot \xi = (t+j\ell_0) \xi_1, \quad |t|\le \overline{c},$$
since $\ao^j(0)=(j\ell_0,0,\dots,0)$.  Thus since $\beta(|\xi|/\lk)=1$ for $|\xi|\in (\lk/2,2\lk)$ and $\eta\in {\mathcal S}$,
by \eqref{4.8},
\begin{multline}\label{4.24}
\psik((\tilde \gamma(t))=
\\
\sum_{\{j\in \Z: \, |j\ell_0|\le 2T_k\}}
(\lk\dk)^{-\frac{n-1}4}\int
e^{i (t+j\ell_0)\xi_1} a(\lk^{1/2}\dk^{-1/2}|e_1-|\xi/|\xi| \, |) \, \etak \, d\xi
\\
+O(\lk^{-N}), 
\, \, |t|\le \overline{c}.
\end{multline}

Next, let us polar coordinates $\xi=r\omega$, $r>0$, $\omega\in S^{n-1}$ so that $\xi/|\xi|=\omega$.
We then write each summand as above as
\begin{align*}
&(\lk\dk)^{-\frac{n-1}4}
\int_{S^{n-1}} \int_0^\infty e^{i(t+j\ell_0)r\omega_1} a(\lk^{1/2}\dk^{-1/2}|\omega-e_1|) \, \eta(T_k(\lk-r)) \,
r^{n-1}dr d\omega
\\
&= (\lk\dk)^{-\frac{n-1}4}
\int_{S^{n-1}}
\Bigl(\int_{-\lk}^{\infty} e^{i(t+j\ell_0)(\lk+r)\omega_1} \eta(-T_k r) \,
(\lk+r)^{n-1}dr \Bigr) a(\lk^{1/2}\dk^{-1/2}|\omega-e_1|) \, d\omega
\\
&= (\lk\dk)^{-\frac{n-1}4}
\int_{S^{n-1}}
\Bigl(\lk^{n-1}\int_{-\infty}^{\infty}  e^{i(t+j\ell_0)(\lk-r)\omega_1} \eta(T_k r) \,
dr \Bigr) a(\lk^{1/2}\dk^{-1/2}|\omega-e_1|) \, d\omega
\\
&\qquad\qquad \qquad +O(\lk^{n-2}\lk^{-\frac{3(n-1)}4}\dk^{\frac{n-1}4}),
\end{align*}
since $r^m\eta(r)\in {\mathcal S}(\R)$, $0\le m\le n-2$, and, by \eqref{4.7},
\begin{equation}\label{4.25}
\int_{S^{n-1}}a(\lk^{1/2}\dk^{-1/2}|\omega-e_1|) \, d\omega
\approx \lk^{-\frac{n-1}2}\dk^{\frac{n-1}2}.
\end{equation}
Since there are $O(T_k)=O(\dk^{-1})$ terms in the sum in \eqref{4.24}, we conclude that for $|t|\le \overline{c}$
\begin{multline}\label{4.26}
\psik(\tilde \gamma(t))
=
\\
\lk^{n-1}(\la_k \dk)^{-\frac{n-1}4}
\sum_{ |j\ell_0|\le 2T_k} \int_{S^{n-1}} e^{i(t+\ell_0j)\lk \omega_1}
 T_k^{-1} \Hat \eta(T_k^{-1}(t+j\ell_0)w_1) \, a(\lk^{1/2}\dk^{-1/2}|\omega-e_1|) \, d\omega
 \\
 +O\bigl( (\lk/\dk)^{-1} \lk^{\frac{n-1}4}\dk^{\frac{n-1}4}\bigr).
 \end{multline}
 If the integrand here is nonzero, then by \eqref{4.6}, \eqref{4.7} and \eqref{4.8}, we must have that
 $\omega_1-1=O(\lk^{-1}\dk)$ and $t+j\ell_0=O(c_0T_k)=O(c_0\dk^{-1})$, and so
 \begin{align}\label{4.27}
 e^{i(t+\ell_0j)\lk \omega_1}&=e^{it\lk} e^{i\ell_0j\lk}+O(c_0)
 \\
 &=e^{it\lk}e^{2\pi ijk}+O(c_0) \notag
 \\
 &=e^{it\la_k}+O(c_0), \, \, \, \text{if } \, |t|\le \overline{c}<1,
 \notag
 \end{align}
 due to our choice in \eqref{4.11} of the frequency of our log-quasimode.
 By \eqref{4.6} each $\Hat \eta$ factor in the sum in \eqref{4.26} is nonnegative.  So, since
 for $|t|\le \overline{c}$, with $\overline{c}$ small enough,
 $|(t+j\ell_0)\omega_1| \approx |j|$ if $j\ne0$ and $a(\lk^{1/2}\dk^{-1/2}|\omega-e_1|)\ne0$,
 we see by \eqref{4.6} and \eqref{4.25},
 \begin{equation}\label{4.28}
 \sum_{ |j\ell_0|\le 2T_k} \int_{S^{n-1}}  
 T_k^{-1} \Hat \eta(T_k^{-1}(t+j\ell_0)w_1) \, a(\lk^{1/2}\dk^{-1/2}|\omega-e_1|) \, d\omega
 \approx \lk^{-\frac{n-1}2}\dk^{\frac{n-1}2}.
 \end{equation}
 
 We obtain \eqref{4.18} by combining \eqref{4.26}, \eqref{4.27} and \eqref{4.28} if 
 $c_0$ and $\overline{c}$ are small enough and $\lk\gg 1$.
 
 To prove the remaining inequality \eqref{4.19}, we use Lemma~\ref{lemma4.2} and the above arguments
 to see that
 \begin{equation}\label{4.29}
 \nabla_{y'}\psik(y)=\sum_{\{j\in \Z: \, |j\ell_0|\le 2T_k\}} \nabla_{y'}
 \psi_{\lk,\ao^j}(y)+O(\lk^{-N}), \quad \text{if } \, y\in \tubek,
 \end{equation}
 where $\psi_{\lk,\alpha}$ is as in \eqref{4.21}.  Recall that $\ao$ is as in \eqref{4.3}.
 If $\overline{m}\in O_{n-1}$ is as in \eqref{4.2}, then
 $$ \psi_{\lk,\ao^j}(y)=\lk^{-\frac{n-1}4}\dk^{-\frac{n-1}4}
 \int
 e^{i(y_1+j\ell_0)\xi_1} e^{iy'\cdot (\overline{m})^j \xi'}
 \akd(\xi) \eta(T_k((\la-|\xi|)) \, d\xi.$$
 Since $|(\overline{m})^j\xi'|=|\xi'|=O(\lk^{1/2}\dk^{1/2})$ on the support of the integral,
 we can argue as above to deduce
 that
 $$\nabla_{y'}\psi_{\lk,\ao^j}(y)=O\bigl(T_k^{-1}\lk^{\frac{n-1}4}\dk^{\frac{n-1}4} \, (\lk\dk)^{1/2}\bigr),
 \quad \text{if } \, y\in \tubek,$$
 which yields \eqref{4.19} after recalling \eqref{4.29}.  \qed

\newsection{Some other problems related to the concentration of quasimodes}

Let us now see how we can use the estimates in Theorem~\ref{thm1.1} to make further progress on
problems related to the concentration of eigenfunctions and quasimodes that were discussed before, for instance,
in \cite{HezariSogge}, \cite{SoggeCon} and \cite{SoggeZelNodal}. 

 In Sogge and Zelditch~\cite{SoggeZelNodal}, lower bounds for the
$L^1$-norms of quasimodes,
\begin{equation}\label{a.1}
\la^{-\frac{n-1}4} \|\Phi_\la\|_{L^2(M)}
\lesssim \|\Phi_\la \|_{L^1(M)},
\quad \text{if } \, \text{Spec }\Phi_\la \subset [\la,\la+1]
\end{equation}
were obtained.  These universal lower bounds are saturated by the Gaussian beam spherical harmonics
(highest weight spherical harmonics) on $S^n$.  The
bounds in \eqref{a.1} were used in \cite{HezariSogge}
and \cite{SoggeZelNodal} to obtain progress on the problem of establishing lower bounds for the 
size of nodal sets of eigenfunctions, which was subsequently fully 
resolved by other methods by Logunov~\cite{Logunov}.

In \cite{BSTop} improvements were made to \eqref{a.1}
under the assumption of nonpositive curvatures by
including factors involving certain  positive powers of $\log\la$ in the left.  Let us
now see how we can use \eqref{1.7} to obtain
further improvements.  If we use H\"older's inequality
we see that if all of the sectional curvatures
of $(M,g)$ are nonpositive, we have for 
$q\in (2,q_c]$ and $L^2$-normalized $\Phi_\la$
with spectrum in $[\la,\la+(\log\la)^{-1}]$
$$1\le \|\Phi_\la\|_1^{q-2}\|\Phi_\la\|^q_q
\lesssim \|\Phi_\la\|_1^{q-2}
(\la(\log\la)^{-1})^{\frac{n-1}4(q-2)}.
$$
This of course yields the lower bound
\begin{equation}\label{a.2}
\la^{-\frac{n-1}4}(\log\la)^{\frac{n-1}4}\|\Phi_\la\|_2
\lesssim \|\Phi_\la\|_1, 
\quad \text{if } \, \,
\text{Spec }\Phi_\la \subset [\la,\la+(\log\la)^{-1}].\end{equation}
The Knapp example in \S3.1 suggests that \eqref{a.2}
is an optimal bound although the functions constructed
there satisfied the weaker but related variant
of the above spectral assumption that
$\|\psi_\la\|_2+(\la/\log\la)\|(\Delta_g+\la^2)\psi_\la\|_2 \approx 1$.

If one uses \eqref{1.8} with $q$ close to $2$, one
can use H\"older's inequality as above to see that if all
of the sectional curvatures of $(M,g)$ are negative
we have for every $N$ that if $\la\gg 1$
\begin{equation}\label{a.3}
\la^{-\frac{n-1}4}(\log\la)^{N}\|\Phi_\la\|_2
\lesssim \|\Phi_\la\|_1, 
\quad \text{if } \, \,
\text{Spec }\Phi_\la \subset [\la,\la+(\log\la)^{-1}],
\end{equation}
which is a significant improvement over \eqref{a.1}
as well as \eqref{a.2}.  It would be interesting to know to what
extent this lower bound could be improved. For instance,
are $\lambda$-power improvement possible?

We would also like to point out similar differences
between the concentration of log-quasimodes
near periodic geodesics in manifolds with negative
sectional curvatures compared to flat manifolds.  Recall
that in the flat case the modes $\psi_{\la_k}$ satisfying
\eqref{4.4} had nontrivial $L^2$-mass \eqref{4.17}
in a $\la^{-1/2}_k (\log\la_k)^{1/2}$ tube about the
periodic geodesic $\gamma_0$ in the flat manifold
$(M,g)$.  If one uses Lemma~\ref{lemma4.1} along
with \eqref{1.8}, though, one can repeat the arguments
in \S3.1 to see that if all the sectional curvatures
of $(M,g)$ are negative one can never have a
sequence of modes $\psi_{\la_k}$ satisfying
\eqref{4.5} and also have for a fixed periodic
geodesic $\gamma_0$
\begin{equation}\label{a.4}
\liminf_{\la_k\to \infty}
\|\psi_{\la_k}\|_{L^2({\mathcal T}_{\delta_{k,N}(\la)}(\gamma_0))}>0, \quad \delta_{k,N}(\la)=\la_k^{-1/2}
(\log\la_k)^N,
\end{equation}
for any $N$, with the set in the $L^2$-norm
being a $\la_k^{-1/2}(\log\la_k)^N$ tube about the periodic geodesic
$\gamma_0$.  It would also be interesting to show that the analog of
\eqref{a.4} can never hold on manifolds of negative curvature
if $\delta_{k,N}$ is replaced by $\la_k^{-1/2+\sigma}$ for some $\sigma>0$.
Breaking this log-power barrier, as in the analogous problem regarding \eqref{a.3}, is probably
difficult due to the role of the Ehrenfest time.

The fact that \eqref{a.4} can never
hold when the sectional curvatures are negative
seems to be somewhat related to the assumptions
of Brooks~\cite{BrooksQM} who showed that if
${\mathcal N}(\gamma_0)$ is any {\em fixed} 
neighborhood of $\gamma_0$ and $(M,g)$ has 
constant negative sectional curvatures one can
construct $\psi_{\la_k}$ as above satisfying
$$\liminf_{\la_k\to \infty}
\|\psi_{\la_k}\|_{L^2({\mathcal N}(\gamma_0))}>c_0,
$$
for some fixed constant $c_0>0$ depending on $\gamma_0$ (but not on ${\mathcal N}(\gamma_0)$).
The log-quasimodes that Brooks constructs do not equidistribute; however,
the fact that  \eqref{a.4} can never hold for any $N$ on manifolds with negative
curvature quantifies that the 
rate at which this is manifested is much slower than exhibited by the $\psi_{\la_k}$ constructed
in \S3.1 for flat manifolds, as well of course compared to
much faster rate exhibited by the Gaussian beams on $S^n$.  

We also would like to mention that it would be
interesting to see to what extent one could weaken
the hypotheses in Theorem~\ref{thm1.1} and still
obtain similar improvements over the universal 
bounds \eqref{1.4}.  For quite a while, starting in
Sogge and Zelditch~\cite{SoggeZelditchMaximal},
and more recently in important improvements of Canzani and Galkowski~ \cite{CGGrowth} , \cite{CGInv}, it
has been known that for generic manifolds one can
always improve the estimates in \eqref{1.4} for
{\em supercritical} exponents by considering projection
operators associated with intervals $[\la,\la+\delta(\la)]$ with $\delta(\la)\to 0$.  Despite the
fact that it has been over 20 years since this was
proved in \cite{SoggeZelditchMaximal}, no such results
have been obtained for the critical exponent $q_c$
or subcritical exponents $q\in (2,q_c)$.  We have 
simplified considerably the approach to handle such exponents.  Indeed, one would obtain improved bounds if
one could establish pointwise kernel estimates like those in
\eqref{2.19} and Lemma~\ref{globalker}.  These are the
only places where we used the curvature assumptions
in Theorem~\ref{thm1.1}.  The other estimates, which
were mostly  local ones, are valid for all compact
manifolds.

There is one more problem suggested by our estimates \eqref{1.8} for spectral projection operators on compact manifolds
all of whose sectional curvatures are {\em negative}.  For the critical exponent $q_c=\tfrac{2(n+1)}{n-1}$ we conjecture
that on such manifolds we have
\begin{equation}\label{con1}
\bigl\|\chi_{[\la,\la+\delta(\la)]}\bigr\|_{2\to q_c}=O(1) \quad
\text{if } \, \, \delta(\la)=\la^{-\frac{n-1}{n+1}}.
\end{equation}
This seems very hard to verify even when $(M,g)$ is a space form of negative curvature.  For subcritical
exponents on manifolds all of whose sectional curvatures are negative, we similarly conjecture that
\begin{multline}\label{con2}
\bigl\|\chi_{[\la,\la+\delta(\la)]}\bigr\|_{2\to q}=O(1) \quad
\text{if } \, \, \delta(\la)=\la^{-(n-1)(\frac12-\frac1q)} \quad
\text{for } \, q\in (2,q_c),
\\
\text{with } \, \, \delta(\la)=\la^{-(n-1)(\frac12-\frac1q)}.
\end{multline}
These results would be optimal since such $O(1)$ bounds easily can be seen to be impossible
for $\delta(\la)=\la^{-(n-1)(\frac12-\frac1q)+\sigma}$ with $\sigma>0$ if
$q\in (2,q_c]$.  We also note that neither \eqref{con1} or \eqref{con2} can hold on flat compact manifolds.

\newsection{Appendix: Bilinear oscillatory integrals and local harmonic analysis on manifolds}

It remains to prove Proposition~\ref{locprop}.  We shall first prove \eqref{2.44} and then turn to the proof of \eqref{2.45}. We shall also first prove \eqref{2.44} for $n\ge3$ and then turn to the modifications needed to handle $n=2$.

Recall that the $A^{\theta_0}_\nu$ there are pseudo-differential cutoffs at the scale $\theta_0=\la^{-1/8}$ belonging to a bounded
subset of $S^0_{7/8,1/8}$.

We first note that by \eqref{m1}
\begin{equation}\label{5.1}
\tilde \sigma_\la -\sum_\nu \tilde \sigma_\la A^{\theta_0}_\nu =R_\la
\quad \text{where } \, \, \|R_\la\|_{2\to \infty}=O(\la^{-N}) \, \forall \, N.
\end{equation}

Thus,
\begin{equation}\label{5.2}
(\tilde \sigma_\la h)^2 =\sum_{\nu,\nu'}(\tilde \sigma_\la A^{\theta_0}_\nu h) \cdot (\tilde \sigma_\la A^{\theta_0}_{\nu'} h)
+O(\la^{-N}\|h\|_2^2).
\end{equation}

Let us set
\begin{equation}\label{diag}
\diag(h)=\sum_{(\nu,\nu')\in \Xi_{\theta_0}} (\tilde \sigma_\la A^{\theta_0}_\nu h)
\cdot (\tilde \sigma_\la A^{\theta_0}_{\nu'} h),
\end{equation}
and
\begin{equation}\label{5.4}
\far(h)=\sum_{(\nu,\nu')\notin \Xi_{\theta_0}} (\tilde \sigma_\la A^{\theta_0}_\nu h)
\cdot (\tilde \sigma_\la A^{\theta_0}_{\nu'} h) + O(\la^{-N}\|h\|_2^2),
\end{equation}
with the last term containing the error terms in \eqref{5.2}.  Thus,
\begin{equation}\label{5.5}
(\tilde \sigma_\la h)^2 = \far(h)+\diag(h).
\end{equation}
Note that the summation in $\diag(h)$ is over near diagonal pairs $(\nu,\nu')$ by \eqref{m14}.  In particular,
for $(\nu,\nu')\in \Xi_{\theta_0}\subset \theta_0\cdot {\mathbb Z}^{2(n-1)}$ we have $|\nu-\nu'|\le C\theta_0$ for 
some uniform constant.  The other term $\far(h)$ in \eqref{5.5} includes the remaining pairs, many of which are
far from the diagonal, and this will contribute to the last term in \eqref{2.44}.

When $n=2$, let us further define
\begin{equation}\label{Tnu}
   T_\nu h=\sum_{\nu': \,(\nu, \nu')\in
\Xi_{\theta_0}}( \tilde\sigma_\la A^{\theta_0}_\nu h )(\tilde\sigma_\la A^{\theta_0}_{\nu'} h), 
\end{equation}
and write
\begin{equation}\label{b255}
\begin{aligned}
    ( \diag(h))^{2} &=\big(\sum_\nu T_\nu h\big)^2
    \\&= \sum_{\nu_1, \nu_2}  T_{\nu_1} hT_{\nu_2} h.
    \end{aligned}
\end{equation}
As in \eqref{m14}, if we assume that $B(x,\xi)$ has small conic support, the 
sum in \eqref{b255} can be organized as
\begin{equation}\label{organize}
\begin{aligned}
    &\big( \sum_{\{k\in {\mathbb N}: \, k\ge 20 \, \, \text{and } \, 
\theta=2^k\theta_0\ll 1\}} \,  \, 
\sum_{\{(\mu_1, \mu_2): \, \tau^\theta_{\mu_1}
\sim \tau^\theta_{ \mu_2}\}}
\sum_{\{(\nu_1,\nu_2)\in
\tau^\theta_{\mu_1}\times \tau^\theta_{\mu_2}\}}
+\sum_{(\nu_1,  \nu_2)\in \overline\Xi_{\theta_0}}
\big)T_{\nu_1} hT_{\nu_2} h,  \\
&={\overline\Upsilon^{\text{far}}}(h)+{\overline\Upsilon^{\text{diag}}}(h)
\end{aligned}
\end{equation}
Here $\overline\Xi_{\theta_0}$ indexes the  near diagonal pairs. This is another Whitney decomposition similar to \eqref{m14}, but the diagonal set $\overline\Xi_{\theta_0}$ is much larger than the set $\Xi_{\theta_0}$ in \eqref{m14}. More explicitly, when $n=2$, it is not hard to check that  $|\nu-\nu'|\le 2^{11}\theta_0$ if $(\nu,\nu')\in \Xi_{\theta_0}$
while  $|\nu_1-\nu_2|\le 2^{21}\theta_0$ if $(\nu_1,\nu_2)\in \overline\Xi_{\theta_0}$. This will help us simplify the calculations needed for ${\overline\Upsilon^{\text{far}}}(h)$.

These terms will be treated differently as was previously done in analyzing parabolic restriction theorems or bilinear oscillatory
integrals.

We can treat the terms involving  ${\Upsilon^{\text{diag}}}(h)$ and ${\overline\Upsilon^{\text{diag}}}(h)$ as in \cite{SBLog} by using a variable coefficient variant of 
Lemma 6.1 in Tao, Vargas and Vega~\cite{TaoVargasVega}.

\begin{lemma}\label{diag1}  If $\diag(h)$
is as in \eqref{diag} and $n\ge3$, then we have the
uniform bounds
\begin{equation}\label{star}
\|\diag(h)\|_{L^{q_c/2}}
\lesssim \bigl(\sum_\nu \|\tilde \sigma_\la
A^{\theta_0}_\nu h\|_{L^{q_c}}^{q_c})^{2/q_c}
+O(\la^{\frac2{q_c}-}\|h\|^2_2).
\end{equation}
And for all $n\ge 2$, if $q\in (2, \frac{2(n+2)}{n}]$ and $\mu(q)$ is as in \eqref{1.4},  we have 
\begin{equation}\label{star1}
\|\diag(h)\|_{L^{q/2}}
\lesssim \bigl(\sum_\nu \|\tilde \sigma_\la
A^{\theta_0}_\nu h\|_{L^{q}}^{q})^{2/q}
+O(\la^{2\mu(q)-}\|h\|^2_2).
\end{equation}
Also if $n=2$ and ${\overline\Upsilon^{\text{diag}}}(h)$ is  as in \eqref{organize}, we have 
\begin{equation}\label{star2}
\|{\overline\Upsilon^{\text{diag}}}(h)\|_{L^{3/2}}
\lesssim \bigl(\sum_\nu \|\tilde \sigma_\la
A^{\theta_0}_\nu h\|_{L^{6}}^{6})^{2/3}
+O(\la^{\frac2{3}-}\|h\|^4_2).
\end{equation}
\end{lemma}

We also require the following estimates for $\far(h)$ and ${\overline\Upsilon^{\text{far}}}(h)$ which will be proved using
bilinear oscillatory integral estimates of Lee~\cite{LeeBilinear} and slightly simplified 
variants of arguments in \cite{BlairSoggeRefined}, \cite{blair2015refined} and 
\cite{SBLog}.

\begin{lemma}\label{leelemma}  Let $n\ge2$.
 If $\far(h)$ is as in \eqref{5.4}, and, as above $\theta_0=\la^{-1/8}$ then for
all $\e>0$ we have
\begin{equation}\label{5.7}
\int_M |\far(h)|^{q/2}\, dx \le C_\e \la^{1+\e}
\bigl(\la^{7/8}\bigr)^{\frac{n-1}2 (q-q_c)} \, 
\|h\|^q_{L^2(M)}, \quad \,q= \tfrac{2(n+2)}n,
\end{equation}
assuming, as in Proposition~\ref{locprop}, that the conic support of $B(x,\xi)$ in \eqref{2.8} as well as $\delta$ and $\delta_0$
in \eqref{2.2}
are sufficiently small.
Similarly, for all $n\ge 2$, we have
\begin{equation}\label{5.7q}
\int_M |\far(h)|^{q/2}\, dx \le C \la^{\mu(q)\cdot q-}
 \, 
\|h\|^q_{L^2(M)}, \quad \,2<q< \tfrac{2(n+2)}n.
\end{equation}
Also if $n=2$ and ${\overline\Upsilon^{\text{far}}}(h)$ is as in \eqref{organize}, we have
\begin{equation}\label{5.7'}
\int_M |{\overline\Upsilon^{\text{far}}}(h)|\, dx \le C_\e \la^{1+\e}
\la^{-7/8} \, 
\|h\|^4_{L^2(M)}.
\end{equation}
\end{lemma}
Here as before, $\la^{\mu-}$ means a factor involving an unspecified exponent smaller than $\mu$. In Lemma~\ref{leelemma} we assume  that $B(x,\xi)$ has small conic support so that in \eqref{m14} we only need to consider $\theta=2^k\theta_0$ which are
small compared to one.  We want $\delta, \delta_0>0$ to be small in order to apply the oscillatory integral estimates in
\cite{LeeBilinear} for $n\ge3$ and H\"ormander~\cite{HormanderFLP} for $n=2$.

Let us postpone the proofs of these lemmas for a bit and see how they can be used to prove  Proposition~\ref{locprop}.

\begin{proof}[Proof of \eqref{2.44} for $n\ge3$]
Let $q=\tfrac{2(n+2)}n$ as in Lemma~\ref{leelemma} and note that $q<q_c$.  Also,
$$| \tilde \sigma_\la h \,  \tilde \sigma_\la h |^{q_c/2}
\le 2^{q/2} | \tilde \sigma_\la h \,  \tilde \sigma_\la h |^{\frac{q_c-q}2}
\cdot
\bigl(|\diag(h)|^{q/2}+|\far(h)|^{q/2}\bigr).
$$
As a result, taking $h=\rho_\la f$ and norms over $A_-$ as in \eqref{2.44},
\begin{multline}\label{5.8}
\|\tilde \sigma_\la h\|^{q_c}_{L^{q_c}(A_-)}=\int_{A_-}  | \tilde \sigma_\la h \cdot \tilde \sigma_\la h|^{q_c/2} \, dx
\\
\lesssim \int_{A_-} | \tilde \sigma_\la h \cdot \tilde \sigma_\la h|^{\frac{q_c-q}2} \, |\diag(h)|^{q/2} \, dx
\\
+ \int_{A_-} | \tilde \sigma_\la h \cdot \tilde \sigma_\la h|^{\frac{q_c-q}2} \, |\far(h)|^{q/2} \, dx  
= I + II.
\end{multline}

To estimate $II$ we use \eqref{5.7}, the ceiling for $A_-$ in \eqref{2.17} and the fact that, by \eqref{2.11}, $\tilde \sigma_\la =\tilde \rho_\la f$,
to conclude that
\begin{multline*}
II\lesssim \|\tilde \rho_\la f\|^{q_c-q}_{L^\infty(A_-)} \cdot \la^{1+\e} (\la^{7/8})^{\frac{n-1}2(q-q_c)} \|h\|^q_2
\\
\lesssim \la^{(\frac{n-1}4+\frac18)(q_c-q)} \, \la^{-(q_c-q)(\frac78 \cdot \frac{n-1}2)} \cdot \la^{1+\e}
=O(\la^{1-\delta_n+\e}).
\end{multline*}
Here $\delta_n>0$ since $(q_c-q)(\tfrac{3(n-1)}{16}-\tfrac18)>0$, and we also used the fact that
$\|h\|_2=\|\rho_\la f\|_2=O(1)$ by \eqref{2.15}.

Since we may take $\e<\delta_n$, $II^{1/q_c}$ is dominated by the last term in \eqref{2.44}.  Consequently, to finish the 
proof of this inequality, we just need to see that we also have suitable bounds for $I^{1/q_c}$.  To do so we use H\"older's inequality
followed by Young's inequality and \eqref{star} to see that
\begin{align*}
I&\le \|\tilde \sigma_\la h\cdot \tilde \sigma_\la h\|_{L^{q_c/2}(A_-)}^{\frac{q_c-q}2} \cdot C \| \diag(h)\|^{q/2}_{L^{q_c/2}}
\\
&\le \tfrac{q_c-q}{q_c} \|\tilde \sigma_\la h\cdot \tilde \sigma_\la h\|_{L^{q_c/2}(A_-)}^{q_c/2}
+\tfrac{q}{q_c} C \|\diag(h)\|_{L^{q_c/2}}^{q_c/2}
\\
&\le \tfrac{q_c-q}{q_c} \|\tilde \sigma_\la h \|^{q_c}_{L^{q_c}(A_-)} + C'
\sum_\nu \|\tilde \sigma_\la A^{\theta_0}_\nu h\|^{q_c}_{L^{q_c}}+O(\la^{1-}).
\end{align*}
Since $\tfrac{q_c-q}{q_c}<1$, the first term in the right can be absorbed in the left side
of \eqref{5.8}, and this along with the earlier estimate for $II$ yields \eqref{2.44}  when $n\ge3$.
\end{proof}

\begin{proof}[Proof of \eqref{2.44} for $n=2$.] If $n=2$, we shall still use \eqref{5.8}, and note that the estimate for the term $II$ also hold for $n=2$, so it suffices to modify the arguments for the first term $I$. Since , $q=\frac{2(n+2)}{n}=4$ and $q_c=6$ if $n=2$, by \eqref{organize}, we have 
\begin{equation}\label{5.8n2}
\begin{aligned}
    I&=
\int_{A_-} | \tilde \sigma_\la h \cdot \tilde \sigma_\la h| \, |\diag(h)|^{2} \, dx  \\
&\le \int_{A_-} | \tilde \sigma_\la h \cdot \tilde \sigma_\la h| \, |{\overline\Upsilon^{\text{diag}}}(h)| \, dx +\int_{A_-} | \tilde \sigma_\la h \cdot \tilde \sigma_\la h| \, |{\overline\Upsilon^{\text{far}}}(h)| \, dx \\
&= A+B 
\end{aligned}
\end{equation}

To estimate $B$ we use \eqref{5.7'}, the ceiling for $A_-$ in \eqref{2.17} and the fact that, by \eqref{2.11}, $\tilde \sigma_\la =\tilde \rho_\la f$,
to conclude that
\begin{equation*}
    B\lesssim \|\tilde \rho_\la f\|^{2}_{L^\infty(A_-)} \cdot \la^{1+\e} \la^{-7/8} \|h\|^4_2
\lesssim \la^{(\frac{1}4+\frac18)(2)} \, \la^{-\frac78} \cdot \la^{1+\e}
=O(\la^{1-\frac18+\e}).
\end{equation*}
Since we may take $\e<\frac18$, $B^{1/6}$ is dominated by the last term in \eqref{2.44}. Thus, we just need to see that we also have suitable bounds for $A^{1/6}$. By  H\"older's inequality,  Young's inequality and \eqref{star2}, we have
\begin{align*}
A
&\le \|\tilde\sigma_\la h \,  \tilde\sigma_\la h\|_{L^{3}(A_-)} \cdot \|{\overline\Upsilon^{\text{diag}}}(h)\|_{L^{3/2}(M)}
\\
&\le \tfrac{1}{3}  \|\tilde\sigma_\la h \,  \tilde\sigma_\la h\|_{L^{3}(A_-)}^{3} + \tfrac{2}{3} \|{\overline\Upsilon^{\text{diag}}}(h)\|_{L^{3/2}(M)}^{3/2}
\\
&\le \tfrac{1}{3}  \|\tilde\sigma_\la h \,  \tilde\sigma_\la h \|_{L^{3}(A_-)}^{3} +C
 \sum_{\nu}  \, \|\tilde\sigma_\la A^{\theta_0}_{\nu}  h\|_{L^{6}(M)}^{6}\, \, + \, \la^{1-}.
\end{align*}
The first term in the right can be absorbed in the left side
of \eqref{5.8}, and this along with the earlier estimates yields \eqref{2.44}  when $n=2$.
\end{proof}

\begin{proof}[Proof of \eqref{2.45}.] The proof of \eqref{2.45} is much simpler since we do not have to restrict to the set $A_{-}$. Since 
$$| \tilde \sigma_\la h \,  \tilde \sigma_\la h |^{q/2}
\le 2^{q/2} 
\cdot
\bigl(|\diag(h)|^{q/2}+|\far(h)|^{q/2}\bigr),
$$
we have 
\begin{equation}\label{5.8q}
\|\tilde \sigma_\la h\|^{q}_{L^{q}(M)}=\int  | \tilde \sigma_\la h \cdot \tilde \sigma_\la h|^{q/2} \, dx
\lesssim \int \, |\diag(h)|^{q/2} \, dx
+ \int |\far(h)|^{q/2} \, dx .
\end{equation}
Thus \eqref{2.45} simply follows from applying \eqref{star1} for the first term and \eqref{5.7q} for the second term on the right side.
\end{proof}

\begin{proof}[Proof of Lemma~\ref{diag1}]  Let us
define the  wider cutoffs, after recalling \eqref{2.7} and
\eqref{qnusymbol}, by setting
\begin{equation}\label{a}
\tilde A^{\theta_0}_\nu(x,\xi)=\psi(x)
\sum_{\{k\in {\mathbb Z}: \, |k|\le C_0\}} \, \, 
\sum_{\{\ell\in {\mathbb Z}^{2(n-1)}: \, |\theta_0\ell-\nu|\le C_0\theta_0\}}
a^{\theta_0}_\ell(x,\xi)
\beta(2^kp(x,\xi)/\la).
\end{equation}
If $C_0$ is fixed large enough we then clearly have
\begin{equation}\label{b}
\|A^{\theta_0}_\nu -
A^{\theta_0}_\nu \tilde A^{\theta_0}_\nu\|_{p\to p}
=O(\la^{-N}) \, \forall N \, \, 
\text{if } \, 1\le p\le \infty,
\end{equation}
if $\tilde A^{\theta_0}_\nu(x,D)$ is the operator with
symbol $\tilde A^{\theta_0}_\nu(x,\xi)$.

For later use, let us also recall that, by \eqref{m2} and 
\eqref{m3}, for each fixed $x$ the support of
$\xi\to A^{\theta_0}_\nu(x,\xi)$ is contained in a
cone of aperture $\lesssim\theta_0= \la^{-1/8}$.  So if
$(\nu,\nu')\in \Xi_{\theta_0}$  then both 
$\xi \to A^{\theta_0}_\nu(x,\xi)$ and
$\xi \to A^{\theta_0}_{\nu'}(x,\xi)$ are supported
for every fixed $x$ in a common cone of aperture $O(\la^{-1/8})$
since $\nu-\nu'=O(\la^{-1/8})$ when
$(\nu,\nu')\in \Xi_{\theta_0}$.  Thus,
it is not difficult to check
we can also fix $C_0$ large enough so that
that we also have that
\begin{multline}\label{c}
\text{if } \, (\nu,\nu')\in \Xi_{\theta_0} \, \,
\text{and } \, \, \bigl(1- \overline{\tilde A^{\theta_0}_\nu(y,\zeta)} \bigr)
A^{\theta_0}_\nu(y,\xi)A^{\theta_0}_{\nu'}(y,\eta)\ne0,
\\
\text{then } \, \, |\zeta-(\xi+\eta)|\ge c\theta_0\la,
\end{multline}
for some fixed constant $c>0$.  In what follows
we fix $C_0$ large enough so that we have \eqref{b}
and \eqref{c}.

To use \eqref{b} we note that,
since \eqref{1.4} yields
$\|\tilde \sigma_\la\|_{2\to q_c}=O(\la^{1/q_c})$,
we conclude that, in order to prove \eqref{star},
it suffices to prove
\begin{equation}\label{d}
\bigl\| \sum_{(\nu,\nu')\in \Xi_{\theta_0}}
(\tilde \sigma_\la A^{\theta_0}_\nu \tilde A^{\theta_0}_\nu h) (\tilde \sigma_\la A^{\theta_0}_{\nu'} \tilde A^{\theta_0}_{\nu'} h)\|_{L^{q_c/2}}
\lesssim \bigl(\sum_\nu \|\tilde \sigma_\la A^{\theta_0}_\nu h\|_{q_c}^{q_c}\bigr)^{2/q_c} +O(\la^{\frac2{q_c}-}
\|h\|_2^2).
\end{equation}
To do this we require the following variant of 
\eqref{commute}
\begin{equation}\label{e}
\|\tilde \sigma_\la A^{\theta_0}_\nu
-A^{\theta_0}_\nu \tilde \sigma_\la\|_{2\to q_c}
=O(\la^{\frac1{q_c}-\frac14 }),
\end{equation}
which follows from the same argument that was used
to obtain \eqref{commute}.

Since the proof of \eqref{2.33} also yields due to \eqref{a}
\begin{equation}\label{r}
\sum_\nu \|\tilde A^{\theta_0}_\nu f\|^r_r
\lesssim \|f\|_r^r, \quad 2\le r\le \infty,
\end{equation}
we can use this inequality for $r=2$ along with \eqref{e} to see
that we would obtain \eqref{d} if we could show that
\begin{equation*}
\bigl\| \sum_{(\nu,\nu')\in \Xi_{\theta_0}}
(A^{\theta_0}_\nu \tilde \sigma_\la \tilde A^{\theta_0}_\nu h) \cdot (A^{\theta_0}_{\nu'} \tilde \sigma_\la \tilde A^{\theta_0}_{\nu'} h)
\bigr\|_{L^{q_c/2}}
\le C\bigl(\sum_\nu \|\tilde \sigma_\la A^{\theta_0}_\nu h\|_{q_c}^{q_c}\bigr)^{2/q_c}+O(\la^{\frac2{q_c}-}
\|h\|_2^2).
\end{equation*}

Next, if we take $r=(q_c/2)'$ so that $r$ is the conjugate
exponent, we conclude that it suffices to show that
\begin{multline}\label{g}
\Bigl| \sum_{(\nu,\nu')\in \Xi_{\theta_0}}
\int (A^{\theta_0}_\nu \tilde \sigma_\la
\tilde A^{\theta_0}_\nu h) 
\cdot (A^{\theta_0}_{\nu'} \tilde \sigma_\la
\tilde A^{\theta_0}_{\nu'} h) \cdot \overline{f} \,
dx \Bigr|
\\
\le C\bigl(\sum_\nu \|\tilde \sigma_\la
A^{\theta_0}_\nu h\|_{q_c}^{q_c}\bigr)^{2/q_c}
+O(\la^{\frac2{q_c}-}\|h\|_2^2), \quad
\text{if } \, \, \|f\|_r=1.
\end{multline}
To do this, we note that by \eqref{c} and a simple
integration by parts argument we have
$$\|(I-\tilde A^{\theta_0}_\nu)^*
\bigl(A^{\theta_0}_\nu h_1 \cdot A^{\theta_0}_{\nu'} h_2\bigr)
\|_{L^\infty}
\le C_n\la^{-N}\|h_1\|_1 \cdot \|h_2\|_1 \, \,
\forall \, N, \, \, \,
\text{if } \, \, (\nu,\nu')\in \Xi_{\theta_0}.
$$
Thus, modulo $O(\la^{-N}\|h\|_2^2)$ errors, the left
side of \eqref{g} is dominated by
\begin{equation}\label{gg}
    \begin{aligned}
      \Bigl| &\sum_{(\nu,\nu')\in \Xi_{\theta_0}}
\int (A^{\theta_0}_\nu \tilde \sigma_\la
\tilde A^{\theta_0}_\nu h) 
\cdot (A^{\theta_0}_{\nu'} \tilde \sigma_\la
\tilde A^{\theta_0}_{\nu'} h) \cdot \overline{\tilde A^{\theta_0}_\nu f} \,
dx \Bigr|
\\
&\le \bigl(\sum_{(\nu,\nu')\in \Xi_{\theta_0}}
\| (A^{\theta_0}_\nu \tilde \sigma_\la
\tilde A^{\theta_0}_\nu h) 
\cdot (A^{\theta_0}_{\nu'} \tilde \sigma_\la
\tilde A^{\theta_0}_{\nu'} h)\|_{L^{q_c/2}}^{q_c/2}
\bigr)^{2/q_c}
\cdot \bigl(\sum_{(\nu,\nu')\in \Xi_{\theta_0}}
\|\tilde A^{\theta_0}_\nu f\|_r^r\bigr)^{1/r}
\\
&\lesssim 
\bigl(\sum_\nu \|A^{\theta_0}_\nu 
\tilde \sigma_\la \tilde A^{\theta_0}_\nu h\|^{q_c}_{q_c}
\bigr)^{2/q_c} \cdot \bigl(
\sum_\nu\| \tilde A^{\theta_0}_\nu f\|_r^r\bigr)^{1/r}
\\
&\lesssim \bigl(\sum_\nu \|A^{\theta_0}_\nu 
\tilde \sigma_\la \tilde A^{\theta_0}_\nu h\|^{q_c}_{q_c}
\bigr)^{2/q_c},  
    \end{aligned}
\end{equation}
using H\"older's inequality, the fact that
if $\nu$ is fixed there are just $O(1)$ indices $\nu'$ with
$(\nu,\nu')\in \Xi_{\theta_0}$, followed by
\eqref{r} for the estimate of $(\sum_\nu \|\tilde A^{\theta_0}_\nu f\|^r_r)^{1/r})$
and the fact that $q_c\le 4$ if $n\ge3$
and so $r\ge2$.  Based on this, modulo
$O(\la^{\frac2{q_c}-}\|h\|_2^2)$, the left side
of \eqref{g} is dominated by
$(\sum_\nu \|A^{\theta_0}_\nu \tilde \sigma_\la
\tilde A^{\theta_0}_\nu h\|_{q_c}^{q_c})^{2/q_c}$.
So, if we repeat the earlier arguments and use
\eqref{b} again we conclude that this last expression
is dominated by $(\sum_\nu \|\tilde \sigma_\la 
A^{\theta_0}_\nu h\|_{q_c}^{q_c})^{2/q_c}+O(\la^{\frac2{q_c}-}\|h\|_2^2)$, which 
yields \eqref{star2}.

The proof of \eqref{star1} is exactly the same, one can just repeat the arguments and  use the following variant of \eqref{e},
\begin{equation}\label{ee}
\|\tilde \sigma_\la A^{\theta_0}_\nu
-A^{\theta_0}_\nu \tilde \sigma_\la\|_{2\to q}
=O(\la^{\mu(q)- }),
\end{equation}
which is a consequence of interpolation between \eqref{e} and the trivial $L^2\rightarrow L^2$ estimates. 
Also note that in \eqref{star1} holds for all $n\ge 2$, since we are assuming $q\le \frac{2(n+2)}{n}$, which implies that $q\le 4$ for all $n\ge 2$, and thus $r\ge 2$ in \eqref{gg}.

Now we shall prove \eqref{star2}, we have to treat the $n=2$ case separately due to the failure of \eqref{gg} when $q_c=6$. If we repeat the arguments in \eqref{a}-\eqref{r}, it suffices to show that 
\begin{multline}\label{g'}
\Bigl|\sum_{(\nu_1,  \nu_2)\in \overline\Xi_{\theta_0}}
\int (A^{\theta_0}_{\nu_1} \tilde \sigma_\la
\tilde A^{\theta_0}_{\nu_1} h) (A^{\theta_0}_{\nu'_1} \tilde \sigma_\la
\tilde A^{\theta_0}_{\nu'_1} h) (A^{\theta_0}_{\nu_2} \tilde \sigma_\la
\tilde A^{\theta_0}_{\nu_2} h) (A^{\theta_0}_{\nu'_2} \tilde \sigma_\la
\tilde A^{\theta_0}_{\nu'_2} h) \cdot \overline{f} \,
dx \Bigr|
\\
\le C\bigl(\sum_\nu \|\tilde \sigma_\la
A^{\theta_0}_\nu h\|_{6}^{6}\bigr)^{2/3}
+O(\la^{\frac23-}\|h\|_2^4), \quad
\text{if } \, \, \|f\|_3=1.
\end{multline}
Here $(\nu_1,  \nu'_1)\in \Xi_{\theta_0}$,  $(\nu_2,  \nu'_2)\in \Xi_{\theta_0}$, and the set $\overline\Xi_{\theta_0}$ is as in \eqref{organize}.  So $\nu_1, \nu'_1, \nu_2, \nu'_2$ in \eqref{g'} satisfy $|\nu_1-\nu'_1|+|\nu_1-\nu_2|+|\nu_1-\nu'_2|=O(\la^{-1/8})$.
Thus, if we choose $C_0$ in \eqref{a} large enough, by a 
 simple integration by parts argument we have
 \begin{multline*}
     \|(I-\tilde A^{\theta_0}_{\nu_1})^*
\bigl(A^{\theta_0}_{\nu_1} h_1 \cdot A^{\theta_0}_{\nu'_1} h_2A^{\theta_0}_{\nu_2} h_3 \cdot A^{\theta_0}_{\nu'_2} h_4\bigr)
\|_{L^\infty} \\
\le C_n\la^{-N}\|h_1\|_1 \cdot \|h_2\|_1\cdot\|h_3\|_1 \cdot \|h_4\|_1 \, \,
\forall \, N.
 \end{multline*}
Thus, modulo $O(\la^{-N}\|h\|_2^4)$ errors, the left
side of \eqref{g'} is dominated by
\begin{equation}\label{gg'}
    \begin{aligned}
   \Bigl|&\sum_{(\nu_1,  \nu_2)\in \overline\Xi_{\theta_0}}
\int (A^{\theta_0}_{\nu_1} \tilde \sigma_\la
\tilde A^{\theta_0}_{\nu_1} h) (A^{\theta_0}_{\nu'_1} \tilde \sigma_\la
\tilde A^{\theta_0}_{\nu'_1} h) (A^{\theta_0}_{\nu_2} \tilde \sigma_\la
\tilde A^{\theta_0}_{\nu_2} h) (A^{\theta_0}_{\nu'_2} \tilde \sigma_\la
\tilde A^{\theta_0}_{\nu'_2} h) \cdot \overline{\tilde A^{\theta_0}_{\nu_1} f}  \,
dx \Bigr|
\\
&\le \bigl(\sum_{(\nu_1,  \nu_2)\in \overline\Xi_{\theta_0}}
\|  (A^{\theta_0}_{\nu_1} \tilde \sigma_\la
\tilde A^{\theta_0}_{\nu_1} h) (A^{\theta_0}_{\nu'_1} \tilde \sigma_\la
\tilde A^{\theta_0}_{\nu'_1} h) (A^{\theta_0}_{\nu_2} \tilde \sigma_\la
\tilde A^{\theta_0}_{\nu_2} h) (A^{\theta_0}_{\nu'_2} \tilde \sigma_\la
\tilde A^{\theta_0}_{\nu'_2} h)\|_{L^{3/2}}^{3/2}
\bigr)^{2/3} \\
&\qquad
\cdot \bigl(\sum_{(\nu_1,  \nu_2)\in \overline\Xi_{\theta_0}}
\|\tilde A^{\theta_0}_{\nu_1} f\|_3^3\bigr)^{1/3}
\\
&\lesssim 
\bigl(\sum_\nu \|A^{\theta_0}_\nu 
\tilde \sigma_\la \tilde A^{\theta_0}_\nu h\|^{6}_{6}
\bigr)^{2/3} \cdot \bigl(
\sum_\nu\| \tilde A^{\theta_0}_\nu f\|_3^3\bigr)^{1/3}
\\
&\lesssim \bigl(\sum_\nu \|A^{\theta_0}_\nu 
\tilde \sigma_\la \tilde A^{\theta_0}_\nu h\|^{6}_{6}
\bigr)^{2/3},  
    \end{aligned}
\end{equation}
using H\"older's inequality, the fact that
if $\nu_1$ is fixed there are just $O(1)$ indices $\nu'_1, \nu_2$ and $\nu'_2$ with
$(\nu_1,\nu_2)\in \overline\Xi_{\theta_0}$, $(\nu_1,\nu'_1)\in \Xi_{\theta_0}$ and $(\nu_2,\nu'_2)\in \Xi_{\theta_0}$, followed by
\eqref{r} with $r=3$.  Based on this, modulo
$O(\la^{\frac2{3}-}\|h\|_2^4)$, the left side
of \eqref{g'} is dominated by
$(\sum_\nu \|A^{\theta_0}_\nu \tilde \sigma_\la
\tilde A^{\theta_0}_\nu h\|_{6}^{6})^{2/3}$.
So, if we repeat the earlier arguments and use
\eqref{b} again we conclude that this last expression
is dominated by $(\sum_\nu \|\tilde \sigma_\la 
A^{\theta_0}_\nu h\|_{6}^{6})^{2/3}+ O(\la^{\frac2{3}-}\|h\|_2^4)$, which 
yields \eqref{star2} and completes
the proof of Lemma~\ref{diag1}.
\end{proof}

\noindent{\bf 5.1.  Proof of Lemma~\ref{leelemma}.}

In this subsection we shall start the proof the other lemma, Lemma~\ref{leelemma}, which is a bit more difficult.  We shall see that
it is a consequence of the bilinear estimates of Lee~\cite{LeeBilinear}.

To prove \eqref{5.7} we recall \eqref{m14} and \eqref{5.4} and note that for a given $\theta=2^k\theta_0$, $k\ge10$, we have for
each fixed $c_0>0$
\begin{equation}\label{5.16}
\tilde \sigma A^{\theta_0}_\nu h
=\sum_{\tilde \mu \in (c_0\theta)\cdot {\mathbb Z}^{2(n-1)}} 
\tilde \sigma_\la A^{c_0\theta}_{\tilde \mu}A^{\theta_0}_\nu h +O(\la^{-N}\|h\|_2).
\end{equation}
We are only considering $k\ge10$ due to the organization of the sum in the left side of \eqref{m14}.  As in 
\cite{BlairSoggeRefined}, we shall choose $c_0=2^{-m_0}<1$ to be specified later to ensure that we have the
separation needed to apply bilinear oscillatory integral estimates.

Keeping this in mind fix $k\ge10$ in the first sum in \eqref{m14}.  We then have for a given $c_0$ as above and
pairs of dyadic cubes $\tau^\theta, \tau^\theta_{\mu'}$ with $\tau^\theta_{\mu}\sim \tau^\theta_{\mu'}$
\begin{multline}\label{5.17}
\sum_{(\nu,\nu')\in \tau^\theta_\mu \times \tau^\theta_{\mu'}} (\tilde \sigma_\la A^{\theta_0}_\nu h) \, (\tilde \sigma_\la A^{\theta_0}_{\mu'} h)
\\
=\sum_{(\nu,\nu')\in \tau^\theta_\mu \times \tau^\theta_{\mu'}}  \sum_{\substack{\tau^{c_0\theta}_{\tilde \mu}\cap \overline{\tau}^\theta_\mu
\ne \emptyset \\ \tau^{c_0\theta}_{\tilde \mu'}\cap \overline{\tau}^\theta_{\mu'}\ne \emptyset}}
(\tilde \sigma_\la A^{c_0\theta}_{\tilde \mu}A^{\theta_0}_\nu h) \, (\tilde \sigma_\la A^{c_0\theta}_{\tilde \mu'}A^{\theta_0}_{\nu'} h)
+O(\la^{-N} \|h\|_2^2),
\end{multline}
if $\overline{\tau}^\theta_\mu$ and $\overline{\tau}^\theta_{\mu'}$ are cubes with the same centers but 11/10 times the side length
of $\tau^\theta_\mu$ and $\tau^\theta_{\mu'}$, respectively, so that we have
$\text{dist }(\overline{\tau}^\theta_\mu, \overline{\tau}^\theta_{\mu'})\ge \theta/2$ when $\tau^\theta_\mu \sim \tau^\theta_{\mu'}$.
We obtain \eqref{5.17} from the fact that the product of the symbol of $A^{c_0\theta}_{\tilde \mu}$ and $A^{\theta_0}_\nu$
vanishes if $\tau^{c_0\theta}_{\tilde \mu}\cap \overline{\tau}^\theta_\mu = \emptyset$ and $\nu\in \tau^\theta_\mu$ since
$\theta=2^k\theta_0$ with $k\ge 10$.  Also note that we then have for fixed $c_0=2^{-m_0}$ small enough
\begin{equation}\label{5.18}
\text{dist }(\tau^{c_0\theta}_{\tilde \mu}, \tau^{c_0\theta}_{\tilde \mu'})
\in [4^{-1}\theta, 4^n\theta], \, \, \text{if }  \, \tau_\mu^\theta\sim \tau^\theta_{\mu'}, \, \,
\tau^{c_0\theta}_{\tilde \mu}\cap \overline{\tau}^\theta_\mu \ne \emptyset
\, \, \text{and } \, \tau^{c_0\theta}_{\tilde \mu'}\cap \overline{\tau}^\theta_{\mu'} \ne \emptyset.
\end{equation}
Also, of course, for each $\mu$ there are $O(1)$ indices $\tilde \mu$ with $\tau^{c_0\theta}_{\tilde \mu}\cap \overline{\tau}^\theta_\mu \ne \emptyset$
with $c_0>0$ fixed.  Note also that if we fix $c_0$ then for our pair $\tau_\mu^\theta\sim \tau^\theta_{\mu'}$ of $\theta$-cubes there are only $O(1)$
summands involving $\tilde \mu$ and $\tilde \mu'$ in the right side of \eqref{5.17}.

Based on this we claim that we would have favorable bounds for the $L^{q/2}$-norm, $q=\tfrac{2(n+2)}n$, of the first term in \eqref{2.44} and hence
$\far(h)$ if we could prove the following key result.

\begin{proposition}\label{prop5.3}  Let $\theta=2^k\theta_0=2^k\la^{-1/8}\ll 1$ with $k\in {\mathbb N}$.  Then we can fix
$c_0=2^{-m_0}$ small enough so that whenever
\begin{equation}\label{5.19} \text{dist }(\tau^{c_0\theta}_\nu, \tau^{c_0\theta}_{ \nu'})\in [4^{-1}\theta, 4^n \theta]
\end{equation}
one has the uniform bounds for each $\e>0$
\begin{equation}\label{5.20}
\int \bigl| (\tilde \sigma_\la A^{c_0\theta}_\nu h_1) \, (\tilde \sigma_\la A^{c_0\theta}_{\nu'} h_2)\bigr|^{q/2} \, dx
\le C_\e \la^{1+\e}\bigl(2^k\la^{7/8}\bigr)^{\frac{n-1}2(q-q_c)} \|h_1\|_{L^2}^{q/2} \, 
\|h_2\|_{L^2}^{q/2},
\end{equation}
with, as in \eqref{5.7}, $q=\tfrac{2(n+2)}n$.
\end{proposition}

The proof of this proposition is based on the bilinear oscillatory integral estimates of Lee~\cite{LeeBilinear},  we shall postpone the proof  to the next section. Now
 let us verify the above claim.
We first note that if 
$h_1=\sum_{\nu\in \tau^\theta_\mu} A^{\theta_0}_\nu h$
and 
$h_2=\sum_{\nu'\in \tau^\theta_{\mu'}} A^{\theta_0}_{\nu'} h,
$
then by  the almost orthogonality of the $A^{\theta_0}_\nu$ operators, 
$$\|h_1\|_2^2 \lesssim \sum_{\nu \in \tau^\theta_\mu} \|A^{\theta_0}_\nu h\|_2^2
\quad \text{and } \, \, 
\|h_2\|_2^2 \lesssim \sum_{\nu' \in \tau^\theta_{\mu'}} \|A^{\theta_0}_{\nu'} h\|_2^2.
$$
Thus, \eqref{5.16}, \eqref{5.18}, \eqref{5.20} and Minkowski's inequality yield the following estimates for
the $k$-summand in \eqref{m14} with $k\ge10$, $\theta=2^k\theta_0$ and $q=\tfrac{2(n+2)}n$:
\begin{multline}\label{5.21}
\bigl\| \sum_{(\mu,\mu'): \tau^\theta_\mu \sim \tau^\theta_{\mu'}}
\sum_{(\nu,\nu')\in \tau^\theta_\mu \times \tau^\theta_{\mu'}} (\tilde \sigma A^{\theta_0}_\nu h)
(\tilde \sigma A^{\theta_0}_{\nu'} h) \bigr\|_{L^{q/2}}
\\
\le \sum_{(\mu,\mu'): \tau^\theta_\mu \sim \tau^\theta_{\mu'}}
\bigl\|
\sum_{\substack{\tau^{c_0\theta}_{\tilde \mu}\cap \overline{\tau}^\theta_\mu
\ne \emptyset \\ \tau^{c_0\theta}_{\tilde \mu'}\cap \overline{\tau}^\theta_{\mu'}\ne \emptyset}}
(\tilde \sigma_\la A^{c_0\theta}_{\tilde \mu}(\sum_{\nu\in \tau^\theta_\mu}A_\nu^{\theta_0}h)) \cdot
(\tilde \sigma_\la A^{c_0\theta}_{\tilde \mu'}(\sum_{\nu'\in \tau^\theta_{\mu'}}A_{\nu'}^{\theta_0}h))\bigr\|_{L^{q/2}}
+O(\la^{-N}\|h\|_2^2)
\\
\lesssim_\e \la^{(1+\e)\frac2q}
(2^k\la^{7/8})^{\frac{n-1}q(q-q_c)} \sum_{(\mu,\mu'): \tau^\theta_\mu \sim \tau^\theta_{\mu'}}
\bigl(\sum_{\nu\in \tau^\theta_\mu}\|A_\nu^{\theta_0}h\|_2^2\bigr)^{1/2}
\bigl(\sum_{\nu'\in \tau^\theta_{\mu'}}\|A_{\nu'}^{\theta_0}h\|_2^2\bigr)^{1/2} 
+O(\la^{-N}\|h\|_2^2)
\\
\lesssim  \la^{(1+\e)\frac2q}
(2^k\la^{7/8})^{\frac{n-1}q(q-q_c)} \sum_\mu \sum_{\nu\in \tau^\theta_\mu}\|A_\nu^{\theta_0}h\|_2^2+O(\la^{-N}\|h\|_2^2)
\\
\lesssim \la^{(1+\e)\frac2q}
(2^k\la^{7/8})^{\frac{n-1}q(q-q_c)}  \|h\|_2^2 +O(\la^{-N}\|h\|_2^2).
\end{multline}
In the above we used the fact that for each $\tau^\theta_\mu$ there are $O(1)$ cubes $\tau^{c_0\theta}_{\tilde \mu}$ with
$\tau^{c_0\theta}_{\tilde \mu}\cap \overline{\tau}^\theta_\mu\ne \emptyset$ and $O(1)$ $\tau^\theta_{\mu'}$ with
$\tau^\theta_\mu \sim \tau^\theta_{\mu'}$ and we also used \eqref{2.33}.

Since $q-q_c<0$, we conclude from this that if we replace $\far(h)$ by the first term in \eqref{m14} then the resulting
expression satisfies the bounds in \eqref{5.7}.  Since, by \eqref{5.4} ,the additional part of $\far(h)$ is pointwise bounded
by $O(\la^{-N}\|h\|_2^2)$, the proof of \eqref{5.7} is complete.

To prove \eqref{5.7q}, note that \eqref{5.7} implies that \eqref{5.7q} is valid if $q=\frac{2(n+2)}{n}$, since $1+\e+\frac 78
\tfrac{n-1}2 (q-q_c)<\mu(q)\cdot q$ if we choose $\e$ to be small enough.  By interpolation, it suffices to show that for the other endpoint $q=2$, we have  
\begin{equation}\label{5.7q2}
\int_M |\far(h)|\, dx \le C 
 \, 
\|h\|^2_{L^2(M)}.
\end{equation}
Recall that as in \eqref{5.5}, $ \far(h)=(\tilde \sigma_\la h)^2-\diag(h)$. By \eqref{2.33} and triangle inequality, it is not hard to see that 
\begin{equation}\label{5.7q2'}
\int_M |\sum_{(\nu, \nu')\in \Xi_{\theta_0}} 
\bigl( \tilde \sigma_\la A^{\theta_0}_\nu \bigr) 
\cdot \bigl( \tilde \sigma_\la
A^{\theta_0}_{ \nu'} 
h\bigr)|\, dx \le C 
 \, 
\|h\|^2_{L^2(M)}.
\end{equation}
Thus \eqref{5.7q2} just follow from \eqref{5.7q2'} and the fact that  $\tilde \sigma_\la$ is a bounded operator on $L^2$.

Now we shall see how we can use Proposition~\ref{prop5.3} to prove \eqref{5.7'}. Note that by \eqref{organize} and triangle inequality, it suffices to show that 
for fixed $\theta=2^k\theta_0$ with $k\ge 20$,
\begin{equation}
    \label{far1tri}
\sum_{\{(\mu_1, \mu_2): \, \tau^\theta_{\mu_1}
\sim \tau^\theta_{ \mu_2}\}}
\sum_{\{(\nu_1, \nu_2)\in
\tau^\theta_{\mu_1}\times \tau^\theta_{ \mu_2}\}}\int |T_{\nu_1} hT_{\nu_2} h| \, dx \le C_\e \la^{1+\e}
\bigl(\la^{\frac78} 2^k\bigr)^{-1}
\|h\|_{L^2(M)}^4.
\end{equation}

To see this, since  $|\nu-\nu'|\le 2^{11}\theta_0$ if $(\nu,\nu')\in \Xi_{\theta_0}$,
by the definition of $T_\nu$ in \eqref{Tnu} and the  Schwarz inequality,  we have 
\begin{equation}\nonumber
|T_{\nu_1} hT_{\nu_2} h|\le C\big(\sum_{\nu'_1: \,|\nu'_1-\nu_1|\le 2^{11}\theta_0}| \tilde\sigma_\la A^{\theta_0}_{\nu'_1} h|^2\big)\big(\sum_{\nu'_2: \,|\nu'_2-\nu_2|\le 
 2^{11}\theta_0}|\tilde\sigma_\la A^{\theta_0}_{\nu'_2}  h|^2\big).
\end{equation}
Thus the integrand in the left side of \eqref{far1tri} is dominated by 
\begin{equation}\label{far1tri1}
\begin{aligned}
   \sum_{\{(\mu_1, \mu_2): \, \tau^\theta_{\mu_1} 
\sim \tau^\theta_{ \mu_2}\}}
\sum_{\{(\nu_1, \nu_2)\in
\tilde\tau^\theta_{\mu_1}\times \tilde\tau^\theta_{ \mu_2}\}} |\tilde\sigma_\la  A^{\theta_0}_{\nu_1}h|^2\cdot|\tilde\sigma_\la A^{\theta_0}_{\nu_2}  h|^2.
    \end{aligned}
\end{equation}
Here $\tilde{\tau}^\theta_{\mu_1}$ and $\tilde{\tau}^\theta_{ \mu_2}$ are the cubes with the same centers but $11/10$ times the
side length
of $\tau^\theta_{\mu_1}$ and $\tau^\theta_{ \mu_2}$, respectively, we used the fact that the side length of $\tau^\theta_{\mu_1}$ is $\ge 2^{20}\theta_0$, so $0.1*\text{side length} \gg 2^{10}\theta_0$.

Furthermore,  if we use \eqref{5.16} again,  for a 
given fixed $c_0=2^{-m_0}$, with $m_0\in {\mathbb N}$ small enough, and pair of dyadic cubes $\tau^\theta_{\mu_1}$, $\tau^\theta_{ \mu_2}$
with $\tau^\theta_{\mu_1} \sim \tau^\theta_{ \mu_2}$
and $\theta=2^k\theta_0$, we have the following analog of \eqref{5.17}
\begin{multline}\label{l2}
\sum_{(\nu_1, \nu_2)\in
\tilde\tau^\theta_{\mu_1}\times \tilde\tau^\theta_{ \mu_2}} 
|\tilde\sigma_\la A^{\theta_0}_{\nu_1} h|^2\cdot| \tilde\sigma_\la A^{\theta_0}_{\nu_2} h|^2
\\
=\sum_{(\nu_1, \nu_2)\in
\tilde\tau^\theta_{\mu_1}\times \tilde\tau^\theta_{ \mu_2}}  \, 
\sum_{\substack{\tau^{c_0\theta}_{\tilde\mu_1} \cap \overline{\tau}^\theta_{\mu_1} \ne \emptyset
\\ \tau^{c_0\theta}_{ \tilde\mu_2} \cap \overline{\tau}^\theta_{ \mu_2} \ne \emptyset}}
|\tilde\sigma_\la A^{c_0\theta}_{\tilde\mu_1}A^{\theta_0}_{\nu_1} h|^2\cdot|\tilde\sigma_\la A^{c_0\theta}_{\tilde\mu_2}A^{\theta_0}_{\nu_2}  h|^2
+O(\la^{-N}\|h\|_2^4),
 \end{multline}
if $\overline{\tau}^\theta_{\mu_1}$ and $\overline{\tau}^\theta_{\mu_2}$ the cubes with the same centers but $12/10$ times the
side length
of $\tau^\theta_{\mu_1}$ and $\tau^\theta_{\mu_2}$, respectively, so that we have
$\text{dist}(\overline{\tau}^\theta_{\mu_1}, \overline{\tau}^\theta_{ \mu_2})\ge \theta/2$ when
$\tau^\theta_{\mu_1}  \sim \tau^\theta_{ \mu_2}$.  
This follows from the fact that for $c_0$ small enough the product of the symbol of $A_{\tilde\mu_1}^{c_0\theta}$ and $A_{\nu_1}^{\theta_0}$ 
vanishes identically if $\tau_{\tilde\mu_1}^{c_0\theta}\cap \overline{\tau}^\theta_{\mu_1}=\emptyset$ and $\nu_1 \in \tilde\tau^\theta_{\mu_1}$, since
$\theta=2^k\theta_0$ with $k\ge 20$.  
 And we also have 
for fixed $c_0$ small enough
\begin{equation}\label{sep}
\text{dist}(\tau^{c_0\theta}_{\tilde\mu_1}, \tau^{c_0\theta}_{ \tilde\mu_2})\in [4^{-1}\theta, 4^2  \theta],
\quad \text{if } \,\tau_{\mu_1}^\theta\sim \tau^\theta_{\mu_2}, \, \tau^{c_0\theta}_{\tilde\mu_1} \cap \overline{\tau}^\theta_{\mu_1} \ne \emptyset, \, \, \,
\text{and } \, \, \tau^{c_0\theta}_{ \tilde\mu_2} \cap \overline{\tau}^\theta_{ \mu_2} \ne \emptyset.
\end{equation}

By applying Proposition~\ref{prop5.3} for $n=2$ and repeating the arguments in \eqref{5.21}, we have 
 \begin{align}\label{l4}
 &\sum_{\{(\mu_1, \mu_2): \, \tau^\theta_{\mu_1}
\sim \tau^\theta_{ \mu_2}\}}
\sum_{\{(\nu_1, \nu_2)\in
\tilde\tau^\theta_{\mu_1}\times \tilde\tau^\theta_{ \mu_2}\}} \int |\tilde\sigma_\la A^{\theta_0}_{\nu_1} h|^2\cdot|\tilde\sigma_\la A^{\theta_0}_{\nu_2} h|^2 \, dx
 \\ \notag
 &\le
 \sum_{(\mu_1, \mu_2): \, \tau^\theta_{\mu_1}\sim \tau^\theta_{ \mu_2}}
 \sum_{\substack{\tau^{c_0\theta}_{\tilde\mu_1} \cap \overline{\tau}^\theta_{\mu_1} \ne \emptyset
\\ \tau^{c_0\theta}_{ \tilde\mu_2} \cap \overline{\tau}^\theta_{\mu_2} \ne \emptyset}}\sum_{\{(\nu_1, \nu_2)\in
\tilde\tau^\theta_{\mu_1}\times \tilde\tau^\theta_{ \mu_2}\}} \int
|\tilde\sigma_\la A^{c_0\theta}_{\tilde\mu_1} A^{\theta_0}_{\nu_1} h|^2\cdot|\tilde\sigma_\la A^{c_0\theta}_{\tilde\mu_2} A^{\theta_0}_{\nu_2} h|^2  dx  \\
&\qquad \qquad \qquad \qquad\qquad \qquad+O(\la^{-N}\|h\|^4_{L^2_{x}}) \notag
\\
& \le C_\e \la^{1+\e} \, \bigl(2^k\la^{7/8}\bigr)^{-1} 
\sum_{(\mu_1, \mu_2): \, \tau^\theta_{\mu_1}\sim \tau^\theta_{ \mu_2} }
\bigl(\sum_{\nu_1\in \tilde\tau^\theta_{\mu_1}} \|A^{\theta_0}_{\nu_1} h\|_{L^2_{x}}^2 \bigr)
\bigl(\sum_{ \nu_2\in \tilde\tau^\theta_{\mu_2}} \|A^{\theta_0}_{\nu_2} h\|_{L^2_{x}}^2\bigr)
+O(\la^{-N}\|h\|^4_{L^2_{x}}) \notag
\\
& \le C_\e \la^{1+\e} \, \bigl(2^k\la^{7/8}\bigr)^{-1}
\sum_\mu \sum_{\nu\in \tau^\theta_\mu} \|A^{\theta_0}_\nu h\|_{L^2_{x}}^4 
+O(\la^{-N}\|h\|^4_{L^2_{x}}) \notag
\\
& \le C_\e \la^{1+\e} \, \bigl(2^k\la^{7/8}\bigr)^{-1}
\|h\|^4_{L^2_{x}} +O(\la^{-N}\|h\|^4_{L^2_{x}}) \notag.
\end{align}
In the above we used the fact that for each $\tau^\theta_{\mu_1}$ there are $O(1)$ cubes $\tau^{c_0\theta}_{\tilde \mu_1}$ with
$\tau^{c_0\theta}_{\tilde \mu_1}\cap \overline{\tau}^\theta_{\mu_1}\ne \emptyset$ and $O(1)$ $\tau^\theta_{\mu_2}$ with
$\tau^\theta_{\mu_1} \sim \tau^\theta_{\mu_2}$ and we also used \eqref{2.33}.

This completes the proof of \eqref{5.7'}. So we conclude that we
have reduced the proof of Lemma~\ref{leelemma}
 to proving Proposition~\ref{prop5.3}.

\noindent{\bf 5.2.  Proof of Proposition~\ref{prop5.3}.}

Let us collect some facts about the kernels of the operators $\tilde \sigma_\la A^{c_0\theta}_\nu$ in \eqref{5.20} that we shall use.
As we shall shortly see they are highly concentrated near certain geodesics in $M$.  Recall that $A^{c_0\theta}_\nu(x,D)$ is a 
``directional operator'' with $\nu\in c_0\theta\cdot {\mathbb Z}^{2(n-1)}$ and, by \eqref{m5}, symbol $A^{c_0\theta}_\nu(x,\xi)$ highly
concentrated near a unit speed geodesic
\begin{equation}\label{5.22}
\gamma_\nu(s)=(x_\nu(s),\xi_\nu(s))\in S^*\Omega.
\end{equation}
Since $\gamma_\nu$ is of unit speed, we have $d_g(x_\nu(s),x_\nu(s'))=|s-s'|$.

To state the properties of the kernels $K^{c_0\theta}_\nu(x,y)$ of the operators $\tilde \sigma_\la A^{c_0\theta}_\nu$, as in 
earlier works, it is convenient to work in Fermi normal coordinates about the spatial geodesic $\overline{\gamma}_\nu =\{x_\nu(s)\}$.
In these coordinates the geodesic becomes part of the last coordinate axis, i.e., $(0,\dots,0,s)$ in $\Rn$, with, as in the earlier
construction of the symbols of the $A^{c_0\theta}_\nu$, $s$ being close to $0$.  For the remainder of this section we shall let
$x=(x_1,\dots,x_n)$ denote these Fermi normal coordinates about our geodesic $\overline{\gamma}_\nu$ associated with
$A^{c_0\theta}_\nu$.  We then have
\begin{equation}\label{5.23}
d_g((0,\dots,0,x_n),(0,\dots,0,y_n))=|x_n-y_n|,
\end{equation}
and, moreover, on $\overline{\gamma}_\nu$ we have that the metric is just $g_{jk}(x)=\delta^k_j$ if $x=(0,\dots,0,x_n)$, and, 
additionally, all of the Christoffel symbols vanish there as well.

It also follows that the symbols $A^{c_0\theta}_\mu(x,\xi)$ of
$A^{c_0\theta}_\mu$, $\mu=\nu,\nu'$ satisfy for some fixed $C_1$
\begin{multline}\label{k1}
|\partial^j_{x_n}\partial^\alpha_{x'}\partial^k_{x_n}\partial^\ell_{\xi_n}\partial^\beta_{\xi'}A^{c_0\theta}_\mu(x,\xi)|
\lesssim_{c_0} \theta^{-|\alpha|-|\beta|}\la^{-|\beta|-\ell},
\quad
\text{and } \, A^{c_0\theta}_\mu(x,\xi)=0 
\\ \text{if } d_g(x,\overline{\gamma}_\mu)\ge C_1c_0\theta, \, \, \xi_n<0, 
\, \, \bigr|\xi'/|\xi|\bigr|\ge C_1  \theta, \, \, 
\text{or } \, \, |\xi/\la|\notin [C_1^{-1},C_1], \, \, \, \mu=\nu,\nu',
\end{multline}
with, as before, $\xi'=(\xi_1,\dots,\xi_{n-1})$.
Additionally, 
\begin{equation}\label{k1'} A_\nu^{c_0\theta}(x,\xi)=0 \, \, \text{if } \, \, \bigr|\xi'/|\xi|\bigr|
\ge C_1 c_0 \theta, \quad \text{and } \Phi_t(0,\eta)=(t\eta,\eta),
\end{equation}
if $\eta=(0, \dots, 0,1)$, with, as before, $\Phi_t$ being geodesic flow in $S^*\Omega$.

In what follows $c_0>0$ will be fixed later small enough, depending on $(M,g)$, so that we can
apply Lee's \cite{LeeBilinear} bilinear oscillatory integral estimates.   As in \eqref{k1}, various constants
in the inequalities we shall state depend on the constant $c_0$ that we shall eventually specify.  Also, as 
before $\theta$ will always be taken to be larger than $\la^{-1/8}$; however, we may assume it is small compared
to one by choosing the cutoff $B$ in the definition of $\tilde \sigma_\la$ to have small support.  Also, as above,
$x'=(x_1,\dots,x_{n-1})$ refers to the first $(n-1)$ coordinates.

We can now formulate the properties of the kernels which we shall require.

\begin{lemma}\label{loclemma}  
Fix $0<\delta\ll \tfrac12 \text{Inj }M$.  Assume further that $\mu=\nu,\nu'$ are as in 
\eqref{5.19} and let $K^{c_0\theta}_\mu$ be the kernel of $\tilde \sigma_\la A^{c_0\theta}_\mu$.  In the above
In the above 
coordinates if $c_0\ll 1$ we have
\begin{equation}\label{k2}
K^{c_0\theta}_{\la,\mu}(x,y)
=\la^{\frac{n-1}2} e^{i\la d_g(x,y)}
a_\mu(\la; x,y) +O(\la^{-N}), \, \, \mu=\nu,\nu',
\end{equation}
where
\begin{equation}\label{k3}
\bigl| \, \bigl(\tfrac\partial{\partial x_n})^{m_1}
\bigl(\tfrac\partial{\partial y_n})^{m_2} 
D^\beta_{x,y}a_\mu\, \bigr|
\le C_{m_1,m_2,\beta} \,  \theta^{-|\beta|}, \, \, \mu=\nu,\nu'.
\end{equation}
Furthermore, for small $\theta$ and $c_0$ there is a 
constant $C_0$ so that the above $O(\la^{-N})$ errors
can be chosen so that the amplitudes have the
following support properties:  First, if $\overline{\gamma}_\nu$ denotes the projection onto $M$ of the 
geodesic in \eqref{5.22} and $\overline{\gamma}_{\nu'}$ the one corresponding to $\nu'$
\begin{equation}\label{k4}
a_\mu(\la;x,y)=0 \, \, \,
\text{if } \, \, d_g(x,\overline{\gamma}_\mu)
+ d_g(y,\overline{\gamma}_\mu)\ge C_0c_0\theta, \, \, \, \mu=\nu, \nu',
\end{equation}
and 
\begin{equation}\label{4'} a_\mu(\la;x,y)=0 \, \, \,
\text{if } \, \, |x'|
+ |y'|\ge C_0\theta,
 \, \, \, \mu=\nu, \nu'.
\end{equation}
As well as, for small $\delta, \delta_0>0$ as in
\eqref{2.2}
\begin{equation}\label{k5}
a_\mu(\la;x,y)=0 \, \, \,
\text{if } \, \, |d_g(x,y)-\delta|\ge 2\delta_0\delta,  
 \, \, \text{or } \, \, x_n-y_n<0, 
\, \, \, \mu=\nu,\nu'.
\end{equation}
\end{lemma}

This lemma is just a small variation of Lemma 4.3 in \cite{SFIO2} (see also Lemma 3.2 in \cite{BlairSoggeRefined}).
We shall postpone its proof until the end of this section. 

Let us describe some properties of the phase function
\begin{equation}\label{5.31} 
\varphi(x,y)=d_g(x,y)
\end{equation}
of our kernels in \eqref{k2}.  First, in addition
to \eqref{5.23}, since we are working in the 
above Fermi normal coordinates we have
\begin{equation}\label{5.32}
\partial\varphi/\partial x_j, \,
\partial\varphi/\partial y_j=0, \, \,
j=1,\dots, n-1, \, \, 
\text{if } \, \, x'=y'=0.
\end{equation}
Consequently, by the last part of \eqref{k5}
\begin{equation}\label{5.33}
\tilde \varphi(x,y)=\varphi(x,y)-(x_n-y_n)
\end{equation}
vanishes to second order when $x'=y'=0$ and the 
amplitude is nonzero.  This means that if we use
the parabolic scaling $(x',x_n)\to (\theta x',x_n)$
we have
\begin{equation}\label{5.34}
D^\beta_{x,y} \bigl(\theta^{-2}
\tilde \varphi(\theta x',x_n,\theta y',y_n)\bigr)
=O_\beta(1) \, \, 
\text{if } \, \, |x'|, |y'|=O(1).
\end{equation}
By \eqref{k3} we also have
\begin{equation}\label{5.35}
D^\beta_{x,y} a_\mu(\la; \theta x',x_n,\theta y',y_n)
= O_\beta(1) \, \, 
\text{if } \, \, |x'|, |y'|=O(1).
\end{equation}

It also follows from Lemma~\ref{loclemma} and a 
straightforward calculation that, in order to prove
\eqref{5.20}, it suffices to show that
\begin{equation}\label{5.36}
\bigl\|(T_1f_1)(T_2f_2)\bigr\|_{L^{q/2}}
\lesssim_\e \la^{-\frac{2n}q +\e}\,
\theta^{-\frac2{n+2}} \,
\|f_1\|_2\|f_2\|_2, \, \,
q=\tfrac{2(n+2)}n,
\end{equation}
where
\begin{align*}(T_1f_1)(x)&=
\int e^{i\la \tilde \varphi(x,y)}
a_\nu(\la;x,y) \, f_1(y) \, dy
\\
(T_2f_2)(x)&=
\int e^{i\la \tilde \varphi(x,z)}
a_{\nu'}(\la;x,z) \, f_2(z) \, dz.
\end{align*}
As we may, in \eqref{5.36} we are neglecting the
$O(\la^{-N})$ error terms in Lemma~\ref{loclemma}. 
Also, as above, we clearly may replace $\varphi$
by $\tilde \varphi$ since, by \eqref{5.33}, the
difference is linear in the last variable.  Note
also that by \eqref{4'} we have
$$(T_1f_1)(x)=(T_2f_2)(x)=0 \quad \text{if } \, 
|x'|\ge C_0\theta.$$

Next, we note that in order to prove \eqref{5.36}, 
by Minkowski's inequality and the Schwarz inequality,
if we define the ``frozen'' bilinear oscillatory
integral operators
\begin{multline}\label{5.37}
\bigl(B^{y_n,z_n}_{\la,\nu,\nu'}\bigr)(h_1,h_2)(x)
=
\\
\iint e^{i\la (\tilde \varphi(x,y',y_n)+\tilde
\varphi(x,z',z_n))}
a_\nu(\la;x,y',y_n)a_{\nu'}(\la;x,z',z_n)h_1(y')
h_2(z')\, dy' dz',
\end{multline}
then it suffices to prove that
\begin{equation}\label{5.38}
\bigl\|
B^{y_n,z_n}_{\la,\nu,\nu'}(h_1,h_2)\bigr\|_{L^{q/2}
(\{x: \, |x'|\le C_0\theta\})}
\lesssim_\e \la^{-\frac{2n}q+\e}
\theta^{-\frac2{n+2}} \, \|h_1\|_2 \|h_2\|_2.
\end{equation}
We note that $B^{y_n,z_n}_{\la,\nu,\nu'}(h_1,h_2)$
factors as the product of two oscillatory integral
operators involving the $(x,y')$ variables.  The
two phase functions are
\begin{equation}\label{5.39}
\phi_{y_n}(x,y')=
\tilde \varphi(x,y',y_n) \, \, \,
\text{and } \, \,
\phi_{z_n}(x,z')=
\tilde \varphi(x,z',z_n).
\end{equation}

In order to apply Lee's \cite{LeeBilinear}
 bilinear oscillatory integral estimates when
$n\ge3$ or H\"ormander's \cite{HormanderFLP} when
$n=2$ we need another simple consequence
of Lemma~\ref{loclemma} which gives us key
separation properties of the supports of the amplitudes.

\begin{lemma}\label{seplemma}  Let $\delta<1/8$ in 
\eqref{2.2} be given.  Then we can fix $c_0$ as in
\eqref{5.17} so that there are constants
$c_\delta, C_\delta\in (0,\infty)$ so that for
sufficiently small $\theta$ and $|x'|\le C_0\theta$, with
$C_0$ as in \eqref{4'} we have
\begin{equation}\label{5.40} 
\text{if } \,
a_\nu(\la;x,y) \cdot a_{\nu'}(\la;x,z)\ne 0
\, \, \,
\text{then } \, \, |y'|, |z'|\le C_\delta \theta
\quad \text{and } \, 
|y'-z'|\ge c_\delta \theta.
\end{equation}
Additionally, 
for sufficiently small $\theta$ we have
\begin{equation}\label{5.41}
\text{if } \, a_\mu(\la,x,y)\ne 0
\, \, \,
\text{then } \, |\delta-(x_n-y_n)|\le 4\delta_0\delta,
\, \, \mu=\nu,\nu'.
\end{equation}
\end{lemma}

\begin{proof}
The first assertion in \eqref{5.40} follows trivially
from \eqref{4'}.  To see the other part, we note that
by \eqref{k4} if the product of the amplitudes in \eqref{5.40}
is nonzero then we must have, for a fixed constant
$C_1$,
$x\in {\mathcal T}_{C_1c_0\theta}(\overline{\gamma}_\nu)
\cap {\mathcal T}_{C_1c_0\theta}(\overline{\gamma}_{\nu'})$, $y\in {\mathcal T}_{C_1c_0\theta}(\overline{\gamma}_\nu)$ and $z\in {\mathcal T}_{C_1c_0\theta}(\overline{\gamma}_{\nu'})$.  By
\eqref{k5}, we must also have that
$d_g(x,y), d_g(x,z)\in [\delta-2\delta_0\delta,
\delta+2\delta_0\delta]$ for our small $\delta_0>0$.  
Since we are assuming \eqref{5.19} the tubes
of width $\approx c_0\theta$ intersect at
angle $\approx \theta$, which implies that
$|y'-z'|\approx \theta$ if the product in 
\eqref{5.40} is nonzero and $c_0$ and $\theta$
are small.

The other, assertion, \eqref{5.41} just follows
from \eqref{4'} and \eqref{k5} if $\theta$ is small enough.\end{proof}

We have collected the main ingredients that will 
allow us to prove the bilinear oscillatory integral
estimates \eqref{5.38}, which will complete the proof
of Proposition~\ref{prop5.3}.

To prove \eqref{5.38}, in addition to following
the proof of \cite{LeeBilinear}[Theorem 1.3],
we shall also follow the related arguments 
in \cite{BlairSoggeRefined} which proved analogous
bilinear estimates for $n=2$ using the simpler
classical bilinear oscillatory integral estimates
implicit in H\"ormander~\cite{HormanderFLP}.

Just as in \cite{LeeBilinear} we first perform
a parabolic scaling as in \eqref{5.34} and \eqref{5.35}
to be able to apply the main estimate, Theorem 1
in Lee~\cite{LeeBilinear}.   So, for small 
$\la^{-1/8}\le \theta \ll 1$, we let
\begin{equation}\label{5.42}
\phi^\theta_{y_n}(x', x_n,y')=\theta^{-2}
\tilde \varphi(\theta x', x_n, \theta y',y_n)
\quad \text{and } \, \,
\phi^\theta_{z_n}=\theta^{-2}
\tilde \varphi(\theta x, x_n,\theta z',z_n),
\end{equation}
and corresponding amplitudes
\begin{equation}\label{5.43}
a^\theta_\nu(\la;x,y)=a_\nu(\theta x',x_n,
\theta y',y_n) 
\quad \text{and } \, a^\theta_{\nu'}(\la;x,z)=a_\nu(\theta x',x_n,
\theta z',z_n).
\end{equation}
Then, as we noted before
$$D^\beta_{x,y}a^\theta_\mu =O_\beta(1), \, 
\mu=\nu,\nu' \, \, \,
\text{and } \, \, D^\beta_{x,y}\phi_j=O_\beta(1),
\, \, \phi_1=\phi^\theta_{y_n}, \, \,
\phi_2=\phi^\theta_{z_n}.$$
By Lemma~\ref{seplemma} we also have the key
separation properties for small enough $\theta$
\begin{multline}\label{5.44}
\text{if } \,
a^\theta_\nu(\la;x,y) a^\theta_{\nu'}(\la;x,z)\ne 0
\\
\text{then  } \, \,
|y'|, |z'|=O(1), 
\, \, |y'-z'|\ge c_\delta \, \, \,
\text{and } \, \,
|y_n-z_n|\le 8\delta_0\delta,
\end{multline}
with $\delta$ and $\delta_0$ as in \eqref{2.2}.

Additionally, by a simple scaling argument, our
remaining task, \eqref{5.38} is equivalent to
the following bounds for small enough $\theta$:
\begin{equation}\label{5.45}
\bigl\| B^{\theta,y_n,z_n}_{\la,\nu,\nu'}
(h_1,h_2)\bigr\|_{L^{q/2}(\{x: \, |x'|\le C_0\})}
\lesssim_\e (\la \theta^2)^{-\frac{2n}q+\e}
\|h_1\|_2\|h_2\|_2, \, q=\tfrac{2(n+2)}n,
\end{equation}
where we have the scaled version of \eqref{5.38}
\begin{multline}\label{5.46}
B^{\theta,y_n,z_n}_{\la,\nu,\nu'}
(h_1,h_2)(x)
=
\\
\iint e^{i(\la\theta^2)[
\phi^\theta_{y_n}(x,y')+\phi^\theta_{z_n}(x,z')]}
a^\theta_\nu(\la;x,y)a^\theta_{\nu'}(\la;x,z)
h_1(y')h_2(z') \, dy'dz'.
\end{multline}

To prove this, let us see how we can use our 
earlier observation based on \eqref{5.32} and 
\eqref{k5} that $\tilde \varphi$
vanishes to second order when $(x',y')=(0,0)$ 
to see that the scaled phase functions in \eqref{5.46}
closely resemble Euclidean ones if $\theta$ is small
which will allow us to verify the hypotheses in
Lee's bilinear oscillatory integral theorem
\cite{LeeBilinear}[Theorem 1.3] if $\delta,
\delta_0>0$ in \eqref{2.2} are fixed small enough.

To do this, consider the following $(n-1)\times (n-1)$
Hessians
\begin{multline}\label{5.47}
A(x_n,y_n)=\frac{\partial^2\tilde \varphi}
{\partial y'_j\partial y'_k}(0,x_n,0,y_n),
\, \,
B(x_n,y_n)= \frac{\partial^2\tilde \varphi}
{\partial x'_j\partial y'_k}(0,x_n,0,y_n),
\\
\text{and } \, \, 
C(x_n,y_n)= \frac{\partial^2\tilde \varphi}
{\partial x'_j\partial x'_k}(0,x_n,0,y_n).
\end{multline}
Then the Taylor expansion about $(x',y')=(0,0)$ is
\begin{multline}\label{5.48}
\tilde \varphi(x',x_n,y',y_n)=
\tfrac12 (y')^t A(x_n,y_n) y'
+(x')^t B(x_n,y_n)y'
+\tfrac12 (x')^tC(x_n,y_n)x' 
\\
+
r(x',x_n,y',y_n),
\end{multline}
where $r(x',x_n,y',y_n)$ vanishes to third order
at $(x',y')=(0,0)$ and so
\begin{equation}\label{5.49}
D^\beta_{x,y} r^\theta(x,y)=O(\theta),
\quad 
\text{if } \, r^\theta(x',x_n,y',y_n)
=\theta^{-2}r(\theta x',x_n,\theta y',y_n).
\end{equation}
This means that $r^\theta\to 0$ in the $C^\infty$
topology as $\theta\to 0$.

To utilize \eqref{5.48} we shall use parabolic scaling
and the following standard lemma (c.f. \cite[\S 5.1]{SFIO2}) saying that the 
phase functions that arise satisfy the Carleson-Sj\"olin condition.

\begin{lemma}\label{cslemma}  Let $A(x_n,y_n)$ 
and $B(x_n,y_n)$ be as in \eqref{5.47}.  
Then if $\delta,\delta_0>0$ in \eqref{2.2} are small 
enough
\begin{equation}\label{5.50}
\det B(x_n,y_n)=
\det \frac{\partial^2 \tilde \varphi(0,x_n,0,y_n)}
{\partial x'_j \partial y'_k}\ne 0
\quad \text{if } \, \,
a^\theta_\nu \cdot a^\theta_{\nu'}\ne 0.
\end{equation}
Also, on the support of $a^\theta_\nu \cdot a^\theta_{\nu'}$,
$-(\tfrac\partial{\partial x_n} A(x_n,y_n))^{-1}
=-(\frac\partial{\partial x_n}\frac{\partial^2
\varphi}{\partial y'_j \partial y'_k}
(0,x_n,0,y_n))^{-1}$ is positive definite, i.e.,
\begin{equation}\label{5.51}
\xi^t \bigl(-\frac\partial{\partial x_n}
A(x_n,y_n)\bigr)^{-1}\xi, \, \,
\xi^t \bigl(-\frac\partial{\partial x_n}
A(x_n,z_n)\bigr)^{-1}\xi
\ge c_\delta |\xi|^2
\quad \text{if }\, a^\theta_\nu \cdot a^\theta_{\nu'}\ne 0,
\end{equation}
and also
\begin{equation}\label{5.52}
\bigl| \frac\partial{\partial x_n}A(x_n,y_n)\xi\bigr|
\ge c_\delta |\xi|, 
\, \, \bigl| \frac\partial{\partial x_n}A(x_n,z_n)\xi\bigr|
\ge c_\delta |\xi|,
\end{equation}
for some $c_\delta>0$.
\end{lemma}

\begin{proof}  Recall that by \eqref{5.31} and \eqref{5.33} $\tilde \varphi(x,y)=d_g(x,y)-(x_n-y_n)$.  As
a result,
$$A(x_n,y_n)=\tfrac{\partial^2}{\partial y'_j
\partial{y'_k}} d_g(0,x_n,0,y_n)
\, \, \text{and } \, \, 
B(x_n,y_n) =\tfrac{\partial^2}{\partial x'_j
\partial{y'_k} }d_g(0,x_n,0,y_n).
$$
Since we are working in Fermi normal coordinates we have
$d_g(x,y)=|x-y|+O(|x-y|^2)$ if $x'=0$.  From this we
deduce that
$$B(x_n,y_n)=-(x_n-y_n)^{-1}I_{n-1}+O(1),
$$
which yields \eqref{5.50} if $\delta,\delta_0>0$ in 
\eqref{2.2} are small since then $d_g(x,y)\approx \delta$
on the support of the amplitudes.  Since we similarly
have
$$\tfrac\partial{\partial x_n}A(x_n,y_n)=-(x_n-y_n)^{-2}
I_{n-1}+O(|x_n-y_n|^{-1}),$$
we similarly obtain \eqref{5.51} and \eqref{5.52} if
$\delta, \delta_0$ are small.  \end{proof}

Let us use \eqref{5.48} and \eqref{5.49} and this 
lemma to prove our remaining estimate \eqref{5.45} 
using the estimate \cite{LeeBilinear}[Theorem 1.1]
of Lee.  As we shall see, it is crucial for us
that $-\tfrac\partial{\partial x_n}A(x_n,y_n)$
is positive definite.

Note that, in addition to the $\theta$ parameter,
\eqref{5.45} also involves the $(y_n,z_n)$ 
parameters.  For simplicity, let us first see how 
Lee's result yields \eqref{5.45} in the case where
these two parameters agree, i.e., $y_n=z_n$.  We
then will argue that if $\delta_0$ in \eqref{2.2} and
hence \eqref{5.44} is fixed small enough we can
also handle the case where $y_n\ne z_n$ due to the fact
that Lee's estimates are valid under small perturbations.

To do this, we first note that the parabolic scaling
in \eqref{5.49}, which agrees with that in \eqref{5.42},
preserves the first three terms in the right side
of \eqref{5.48} as they are quadratic.  Also, in proving
\eqref{5.45}, we may subtract 
$\tfrac12 (x')^t C(x_n,y_n)x'$ from
$\phi^\theta_{y_n}$ and $\tfrac12 (x')^t C(x_n,z_n)x'$
from $\phi^\theta_{z_n}$ as these quadratic terms
do not involve $y'$.  We point out that this
trivial reduction also works if $y_n\ne z_n$.

Next, note that by \eqref{5.50} and our temporary assumption that
$y_n=z_n$, after making a linear change of variables depending 
on $(x_n,y_n)$, we may reduce to the case where
$B(x_n,y_n)=I_{n-1}$, the $(n-1)\times (n-1)$ identity matrix.  This means
that for the special case where $y_n=z_n$ we have reduced to verifying that
\eqref{5.45} is valid where now
\begin{multline}\label{5.53}
\phi^\theta_{y_n}(x',x_n;y')=\langle x',y'\rangle +\tfrac12 \sum_{j,k=1}^{n-1} \frac{\partial^2\tilde \varphi}{\partial y'_j\partial y'_k}
(0,x_n,0,y_n)y'_j y'_k + \tilde r^\theta(x',x_n,y',y_n)
\\
=\langle x',y'\rangle + (y')^tA(x_n,y_n) y' +  \tilde r^\theta(x',x_n,y',y_n),
\end{multline}
with $\tilde r^\theta$ denoting $r^\theta$ written in the new $x$ variables coming from $B(x_n,y_n)$.  For later use, note that
if we change variables according to $y_n$ as above, then for $z_n$ near $y_n$ if
\begin{equation}\label{5.54}
B(x_n,y_n,z_n)=(B(x_n,z_n))^t \, ((B(x_n,y_n)^{-1})^t =I_{n-1}+O(|y_n-z_n|),
\end{equation}
then
\begin{align}\label{5.55}
\phi^\theta_{z_n}(x,z')&=\langle x', B(x_n,y_n,z_n)y'\rangle + \tfrac12 \sum_{j,k=1}\frac{\partial^2 \tilde \varphi}{\partial y'_j\partial y'_k}(0,x_n,0,z_n)
+\tilde r^\theta(x,z)
\\
&= \phi^\theta_{y_n}(x,z')+O(|y_n-z_n|). \notag
\end{align}

We fix $\delta$ and $\delta_0$ in \eqref{2.2} so that the conclusions of Lemma~\ref{loclemma} and \ref{cslemma} are valid.  We can
also finally fix $c_0$ so that the results in Lemma~\ref{seplemma} are valid.  If we only needed to handle the case where $y_n=z_n$
then the above choice of $\delta_0$ would work; however, as we shall see, to handle the case where $y_n\ne z_n$ we shall 
need to choose $\delta_0$ small enough to exploit the last part of \eqref{5.44}.

Let us now verify that we can apply \cite[Theorem 1.1]{LeeBilinear} to obtain \eqref{5.45} for sufficiently small $\theta$.  This would complete
the proof of Proposition~\ref{prop5.3}.

We recall that we are assuming for the moment that $y_n=z_n$ and that we have reduced matters to the case where
$B(x_n,y_n)=I_{n-1}$ and $C(x_n,y_n)=0$ in \eqref{5.48} and so
\begin{equation}\label{5.56}
\phi^\theta_{y_n}(x,y')=\langle x',y'\rangle +\tfrac12 (y')^t A(x_n,y_n)y' + \tilde r^\theta(x,y),
\end{equation}
with $\tilde r^\theta$ satisfying \eqref{5.49}.

By \eqref{5.49} and \eqref{5.56} we have
\begin{equation}\label{5.57}
\frac{\phi^\theta_{y_n}}{\partial x'}(x,y')= y' +\frac{\partial \tilde r^\theta}{\partial x'}=y'+\e(\theta,x,y),
\end{equation}
where $y'\to \e(\, \cdot\, )$ and its derivatives are $O(\theta)$.  Thus, for small enough $\theta$, the inverse function also satisfies
\begin{equation}\label{5.58}
y'\to \bigl( \frac{\partial \phi^\theta_{y_n}}{\partial x'}(x',x_n, \, \cdot \, )\bigr)^{-1}(y') =y' +\tilde \e(\theta,x,y),
\end{equation}
where
\begin{equation}\label{5.59}
D^\beta_{y'}\tilde \e(\theta,x,y)=O_\beta(\theta).
\end{equation}

Next, define in the notation of \cite{LeeBilinear}
\begin{multline}\label{5.60}
q^\theta_s(x',x_n,y')=\tfrac\partial{\partial x_n} \phi^\theta_s\bigl(x',x_n; \, \bigl(\tfrac{\partial \phi^\theta_s}{\partial x'}(x',x_n, \, \cdot \, )\bigr)^{-1}
(y')\bigr)
\\
=\tfrac\partial{\partial x_n}\phi^\theta_s(x',x_n, y'+\tilde \e(\theta,x,y',s)), \, \, s=y_n,z_n,
\end{multline}
as well as
\begin{multline}\label{5.61}
\delta_{y_n,z_n}^\theta(x',x_n; y',z')=
\\
\partial_{y'}q^\theta_{y_n}(x',x_n; \partial_{x'}\phi^\theta_{y_n}(x',x_n,y'))-
\partial_{y'}q^\theta_{z_n}(x',x_n; \partial_{x'}\phi^\theta_{z_n}(x',x_n,z')).
\end{multline}
Even though we are assuming for now that $y_n=z_n$ these two quantities will be needed for $y_n\ne z_n$ as well
to be able to allow us to use \cite[Theorem 1.1]{LeeBilinear} to obtain \eqref{5.45}.  The conditions \cite[(1.4)]{LeeBilinear}
needed to ensure these bounds are
\begin{multline}\label{5.62}
\bigl|\langle \partial^2_{x'y'}\phi^\theta_{y_n}(x,y') \delta^\theta_{y_n,z_n}, \,
\bigl[\partial^2_{x'y'}\phi^\theta_{y_n}(x,y')\bigr]^{-1} \,
\bigl[\partial^2_{y'y'}q^\theta_{y_n}(x; \partial_{x'}\phi^\theta_{y_n}(x,y'))\bigr]^{-1}
\delta^\theta_{y_n,z_n}\rangle \bigr|>0, 
\\
\delta^\theta_{y_n,z_n}=\delta^\theta_{y_n,z_n}(x',x_n;y',z') \, \, \text{on  supp }(a^\theta_\nu \cdot a^\theta_{\nu'}),
\end{multline}
as well as
\begin{multline}\label{5.63}
\bigl|\langle \partial^2_{x'y'}\phi^\theta_{z_n}(x,z') \delta^\theta_{y_n,z_n}, \,
\bigl[\partial^2_{x'y'}\phi^\theta_{z_n}(x,z')\bigr]^{-1} \,
\bigl[\partial^2_{y'y'}q^\theta_{z_n}(x; \partial_{x'}\phi^\theta_{z_n}(x,z'))\bigr]^{-1}
\delta^\theta_{y_n,z_n}\rangle \bigr|>0, 
\\
\delta^\theta_{y_n,z_n}=\delta^\theta_{y_n,z_n}(x',x_n;y',z') \, \, \text{on  supp }(a^\theta_\nu \cdot a^\theta_{\nu'}),
\end{multline}

By  \eqref{5.49}, \eqref{5.53}, \eqref{5.58}, \eqref{5.59} and \eqref{5.60} for small $\theta$ we have
\begin{equation}\label{5.64}
\bigl( \partial^2_{y'y'} q^\theta_{y_n}(x',x_n;y')\bigr)^{-1} = \bigl( \tfrac{\partial A}{\partial x_n}(x_n,y_n)\bigr)^{-1} +O(\theta),
\end{equation}
and also by  \eqref{5.57}, \eqref{5.58} and \eqref{5.59}
\begin{equation}\label{5.65}
\partial^2_{x'y'}\phi^\theta_{y_n}(x',x_n,y')=I_{n-1}+O(\theta),
\end{equation}
as well as
\begin{equation}\label{5.66}
\bigl( \partial^2_{x'y'}\phi^\theta_{y_n}(x',x_n,y')\bigr)^{-1}=I_{n-1}+O(\theta),
\end{equation}

By  \eqref{5.53}, \eqref{5.52}, \eqref{5.60} and the separation condition in \eqref{5.40} if $y_n=z_n$ we
have
\begin{equation}\label{5.67}
|\delta^\theta_{y_n,z_n}(x',x_n;y',z')|>0 \, \, \, \text{on  supp }(a^\theta_\nu \cdot a^\theta_{\nu'}),
\end{equation}
if $\theta$ is small.  Thus, in this case the quantities inside the absolute values in \eqref{5.62} and \eqref{5.63} both equal
\begin{equation}\label{5.68}
\langle \delta^\theta_{y_n,y_n}(x',x_n;y',z'), \, 
\bigl(\tfrac{\partial A}{\partial x_n}(x_n,y_n)\bigr)^{-1} \delta^\theta_{y_n,y_n}(x',x_n;y',z')\rangle
+O(\theta) \, \,  \text{on supp }(a^\theta_\nu \cdot a^\theta_{\nu'}).
\end{equation}
Therefore, by \eqref{5.51} and \eqref{5.67} the conditions \eqref{5.62} and \eqref{5.63} are valid when $y_n=z_n$.  Thus
by \cite[Theorem 1.1]{LeeBilinear} , we obtain \eqref{5.38} in this case.

If $y_n\ne z_n$ in \eqref{5.38}, we must replace $\delta^\theta_{y_n,y_n}$ by $\delta^\theta_{y_n,z_n}$.  In order to accommodate
this, we first need to use that, by the last part of \eqref{5.41}
$$\delta^\theta_{y_n,z_n}(x',x_n;y',z')=\delta^\theta_{y_n,y_n}(x',x_n;y',z')+O(\delta_0)
\, \,  \, \text{on supp }(a^\theta_\nu \cdot a^\theta_{\nu'}).
$$
This means that if we replace $O(\theta)$ by $O(\theta+\delta_0)$ in \eqref{5.68}, then the quantity in \eqref{5.62} is of this form.

The other condition, \eqref{5.63} involves the phase function $\phi^\theta_{z_n}$ and the corresponding $q^\theta_{z_n}$.  
However, if $B=B(x_n,y_n,z_n)$ is as in \eqref{5.54}, then we have the analog of \eqref{5.57} where we replace
the first term the right by $By'$  and the analog of \eqref{5.58} where we replace the first term in the right side by $B^{-1}y'$.  Also, clearly
$\tfrac{\partial A}{\partial x_n}(x_n,z_n)=\tfrac{\partial A}{\partial x_n}(x_n,y_n)+O(|y_n-z_n|)$.
As a result $q^\theta_{z_n}=q^\theta_{y_n}+O(|y_n-z_n|)=q^\theta_{y_n}+O(\delta_0)$ if $a^\theta_\nu \cdot a^\theta_{\nu'}\ne 0$.
Also, by \eqref{5.54} the analogs of \eqref{5.65} and \eqref{5.66} remain valid if $y_n$ is replaced by $z_n$ provided that $O(\theta)$
there is replaced by $O(\theta+\delta_0)$.  So, like \eqref{5.62}, if we replace $O(\theta)$ by $O(\theta+\delta_0)$ in 
\eqref{5.68}, then the quantity in \eqref{5.63} is of this form.

Consequently, if $\delta_0$ in \eqref{2.2} is fixed small enough, and, as above, $\theta$ is small, we conclude that
the condition (1.4) in \cite{LeeBilinear} is valid, which yields \eqref{5.37} and thus completes the proof of Proposition~\ref{prop5.3}.

%

\noindent{\bf 5.3. Proof of Lemma~\ref{loclemma}.}
To finish matters we need to prove the properties
of the microlocalized kernels that we used.

\begin{proof}[Proof of Lemma~\ref{loclemma}]
  The straightforward proof is almost identical to that
of Lemma 3.2 in \cite{BlairSoggeRefined} or Lemma 4.3 in \cite{SBLog}; however, we shall present it for
the sake of completeness.  Note note that when
$\theta\approx 1$, this result is standard.  See,
e.g., Lemma 4.3 in \cite{SFIO2}, and the proof of
our results are just a small variation on that of this 
standard one.

We recall that
\begin{equation}\label{k6}
\tilde \sigma_\la =(2\pi)^{-1}\int e^{i\la t}
\bigl(B\circ e^{-itP}\bigr)\, \Hat \rho(t) \, dt.
\end{equation}
Since $\bigl(Be^{-itP}\bigr)(x,y)$ is smooth
near $(x,y,t)$ if $d_g(x,y)\ne |t|$ and
$\Hat \rho(t)=0$ for $|t-\delta|>\delta_0$, we clearly
have  the first part of \eqref{k5}.  
The second part similarly comes from the fact that by \eqref{k1} the symbol of $A^{c_0\theta}_\mu$,
$\mu=\nu,\nu'$, 
vanishes 
 if $\xi$ is not in a small conic neighborhood of $(0,\dots,0,1)$.

Recall $B\in S^0_{1,0}$ has symbol $B(x,\xi)$ vanishing
when $|\xi|$ is not comparable to $\la$ or when
$(x,\xi)$ is not in a small conic neighborhood of $(0,(0,\dots,0,1))$ if $\delta$ is small.
Therefore, using the
calculus of Fourier integrals for $|t|<2\delta$
with $\delta$ as in \eqref{2.2}, modulo 
smoothing errors,
\begin{multline}\label{k7}
\bigl(Be^{-itP}\bigr)(x,y)=
(2\pi)^{-n}\int e^{iS(t,x,\xi)-iy\cdot \xi}\alpha(t,x,\xi)\, d\xi
\\
=(2\pi)^{-n}\la^n \int e^{i\la (S(t,x,\xi)-y\cdot\xi) }
\alpha(t,x,\la \xi) \, d\xi,
\end{multline}
where $\alpha\in S^0_{1,0}$ also vanishes when $|\xi|$
is not comparable to $\la$ or $(x,\xi)$ is not in a small conic
neighborhood of $(0,(0,\dots,0,1))$.
Also,  the phase function here
$S$ is homogeneous of degree one in $\xi$ and is a
generating function for the half-wave group $e^{-itP}$.
Thus, $S$ solves the eikonal equation,
\begin{equation}\label{k8}
\partial_t S(t,x,\xi)= -p(x,\nabla_x S(t,x,\xi)), \quad
S(0,x,\xi)=x\cdot \xi,
\end{equation}
and, if $\Phi_t$ here denotes the Hamilton flow in
$T^*M\backslash 0$ associated with $p(x,\xi)$, 
\begin{equation}\label{k9}
\Phi_{-t}(x,\nabla_x S)=(\nabla_\xi S,\xi),
\end{equation}
and
\begin{equation}\label{k10}
\text{det } \bigl(\tfrac{\partial^2 S}{\partial x
\partial \xi}\bigr) \ne 0.
\end{equation}
Additionally, by the above facts regarding
$\alpha\in S^0_{1,0}$, for $|t|<2\delta$ we have
\begin{equation}\label{k11}
\partial^j_t\partial_{x,\xi}^\beta
\alpha(t,x,\la \xi)=O(1) \, \, \,
\text{and } \, \, \,
\alpha(t,x,\la\xi)=0 \, \,
\text{if } \, \, |\xi|\notin [C^{-1}, C]]
\end{equation}
for some uniform constant $C$.

By \eqref{k7} and \eqref{k8}, we have
\begin{equation}\label{k12}
\tilde \sigma_\la (x,y)=(2\pi)^{-n-1}\la^n
\iint 
e^{i\la [t+S(t,x,\xi) - y\cdot \xi   ] }
\, \Hat \rho(t) \, \alpha(t,x,\la \xi)  \, d\xi dt
+O(\la^{-N}).
\end{equation}
This implies that
\begin{multline}\label{k15}
K^{c_0\theta}_{\la,\mu}(x,y)
\\
=(2\pi)^{-2n-1} \la^{2n}
\iint e^{i\la [t+ S(t,x,\xi)-z\cdot \xi +(z-y)\cdot \eta ]} \, 
\Hat \rho(t)\, 
   \alpha(t,x,\la\xi) \, A^{c_0\theta}_\mu(z,\la \eta) 
dz d\xi  d\eta dt \\
 +O(\la^{-N}),  \, \, \mu=\nu,\nu'.
\end{multline}

By \eqref{k1'} and a simple integration by parts argument
we have $K^{c_0\theta}_{\la,\mu}(x,y)
=O(\la^{-N})$ if $d_g(y,\overline{\gamma}_\mu)\ge
C_1c_0\theta$, $\mu=\nu,\nu'$.  

If $\mu=\nu$, let us
prove that this is the case also if $d_g(x,\overline{\gamma}_\nu)\ge C_1c_0\theta$, for large
enough $C_1$.  To do so, we note that by \eqref{k9} and
the fact that we are working in Fermi normal coordinates
about $\overline{\gamma}_\nu$, we have
\begin{multline*}\nabla_\xi 
(S(t_0,x_0,\xi)-z_0\cdot \xi
)=0,
\\ 
\text{if } \, \xi=(0,\dots,0,\xi_n),  \, \, \xi_n>0, \,\,
x_0,z_0\in \overline{\gamma}_\nu \, \, 
\text{and } \, x_0-z_0=t_0(0,\dots,0,1), \, \, t_0\approx \delta.
\end{multline*}
Note that for such $x_0,z_0$ we have $t_0=d_g(x_0,z_0)$.
By \eqref{k10}
we have for $t_0\approx \delta$ and $z_0\in \overline{\gamma}_\nu$ 
and $x$ near $x_0$
$$|\nabla_\xi (  
S(t_0,x,\xi)-z_0\cdot\xi
)|
\approx d_g(x,x_0), \, \, \text{if } \, 
 \xi=(0,\dots,0,\xi_n),  \, \, \xi_n>0.$$
By \eqref{k1} and \eqref{k1'} this implies that  we must have for $t_0\approx \delta$
\begin{multline}\label{stat}
\bigl| \nabla_{\xi,\eta,z}\bigl(S(t_0,x,\xi)-z\cdot \xi +(z-y)\cdot \eta\bigr)\bigr| \ge c\theta,
\\
 \text{if } \, \, d_g(x,\overline{\gamma}_\nu)\ge C_1c_0\theta, \, \,
\text{and } \, \alpha(t,x,\la\xi)A^{c_0\theta}_\nu(z,\la\eta)\ne 0,
\end{multline}
for some constant $c>0$ if $C_1$ is fixed large enough.
Since $\theta\ge \la^{-1/8}$, we obtain \eqref{k4}
for $\mu=\nu$ from this and \eqref{k15} via a
simple integration by parts argument.  We similarly
obtain \eqref{k4} for $\mu=\nu'$ if we work in
Fermi normal coordinates about $\overline{\gamma}_{\nu'}$.  We obtain \eqref{4'} from \eqref{k4} since we are
assuming that $\nu-\nu'=O(1)$ which forces the $O(c_0\theta)$ tubes described in \eqref{k4} about
$\overline{\gamma}_\nu$ and $\overline{\gamma}_{\nu'}$
to be a $O(\theta)$ distance apart, and $(x_1,\dots,x_{n-1})=0$ on $\overline{\gamma}_\nu$.

It remains to prove that the kernels are as in
\eqref{k2} with amplitudes satisfying \eqref{k3}.  Let
$$\Psi(t,x,y,z,\xi, \eta)=
t+S(t,x,\xi)-z\cdot \xi +(z-y)\cdot \eta
$$
be the phase function in the oscillatory integral
in \eqref{k15}.  Then, at a stationary point where
$$\nabla_{z,\xi,\eta,t}\Psi=0,$$
we must have $y=z$ and hence
$\Psi=d_g(x,y)$, due to the fact that
$S(t,x,\xi)-z\cdot \xi=0$ and,
as we just pointed out, $t=d_g(x,y)$ at points
where the $\xi$-gradient vanishes.  Additionally,
it is straightforward to see that 
$$\text{det } \frac{\partial^2\Psi}{\partial(\xi,t)
\partial(\xi,t)} \ne 0.$$  
This follows from the
proof of Lemma 5.1.3 in \cite{SFIO2}.  It also clearly implies  that the $(3n+1) \times (3n+1)$ Hessian of the phase
function in \eqref{k15} satisfies
$$\text{det } \frac{\partial^2\Psi}{\partial(z,\xi,\eta,t)
\partial(z,\xi,\eta,t)} \ne 0.$$  
Also,
considering the $z$-gradient, we have $\xi=\eta$ at stationary points.
Thus, by
\eqref{k9}, the oscillatory integral in 
\eqref{k15} has an expansion (see 
H\"ormander~\cite[Theorem 7.7.5]{HV1}) where the
leading term is a  dimensional constant times
\begin{multline}\label{k17}
\la^{\frac{n-1}2}e^{i\la t}
\, \Hat \rho(t)
\alpha(t,x,\la\xi) A_\mu(y,\la \xi), 
\quad
\text{if } \, \, t=d_g(x,y), \, \, 
p(x,\nabla_x S(t,x,\xi))=1, \, \,
\\
\text{and } \, \, 
\Phi_{-t}(x,\zeta)=(y,\xi), \, \text{with } \, 
\zeta=\nabla_x S(t,x,\xi) \, \, \text{and } \, 
y=\nabla_\xi S(t,x,\xi).
\end{multline}
Thus, here $\xi=\xi(x,y)\in S^*_y\Omega$ is the
unit covector over $y$ of the  unit-speed geodesic in $S^*\Omega$
 which passes through $(x,\zeta)$  at time $t=d_g(x,y)$, $t\in \text{supp } \Hat \rho$,
and starts at $(y,\xi)$
.

Consequently, since we are working in Fermi normal coordinates
about $\overline{\gamma}_\nu$,  it follows that
 $\xi ((0,\dots,0,x_n),(0,\dots,0,y_n))\equiv (0,\dots,0,1)$ when $(x_1,\dots,x_{n-1}) =(y_1,\dots,y_{n-1})=0$.
 Consequently,
we have $\partial_{x_n}^j\partial^k_{y_n}\xi_\ell(x,y)=O(\theta)$, $\ell=1,\dots,n-1$, if the kernels are not
$O(\la^{-N})$.
Therefore, it follows from
from \eqref{k1} that
$$  \Hat\rho(d_g(x,y)) \,  \alpha(d_g(x,y),x, \la \xi(x,y)) \, A^{c_0\theta}_\mu(y,\la \xi(x,y)) 
$$
satisfies the bounds in \eqref{k3} 
 if $(x,y)$ are not
in the regions described in \eqref{k4}, \eqref{4'}
or \eqref{k5} where the kernels are $O(\la^{-N})$.
Thus, the leading term in the stationary phase
expansions for the oscillatory integrals in \eqref{k15}
have the desired form.  The same will be
true for the other terms which involve increasing powers of $\la^{-3/4}$ by a straightforward
variant of \cite[(7.7.1)]{HV1}.  
\end{proof}

\bibliography{refs}
\bibliographystyle{abbrv}

\end{document}